\pdfoutput=1
\documentclass{shinyart}

\let\polishL=\L

\usepackage[utf8]{inputenc}

\usepackage{shinybib}

\usepackage{enumitem}
\usepackage{autonum}

\usepackage{asymptote}
\begin{asydef}
    import graph;
    import math;
\end{asydef}

\newcommand{\eps}{\varepsilon}
\renewcommand{\phi}{\varphi}
\newcommand{\N}{\mathbb{N}}
\newcommand{\R}{\mathbb{R}}
\newcommand{\Rbar}{\overline{\mathbb{R}}}

\newcommand{\inner}[1]{\langle #1 \rangle}
\newcommand{\norm}[1]{\|#1\|}
\newcommand{\vnorm}[1]{|#1|}

\newcommand{\wkto}{\rightharpoonup}

\DeclareMathOperator{\sign}{\mathrm{sign}}
\DeclareMathOperator{\dom}{\mathrm{dom}}

\newcommand{\polar}[1]{#1^{\circ}}

\newcommand{\term}{\emph}

\newcommand{\iprod}[2]{\langle #1,#2\rangle}

\newcommand{\B}{B}

\newcommand{\abs}[1]{|#1|}

\newcommand{\inv}[1]{#1^{-1}}
\newcommand{\grad}[1]{\nabla #1}
\newcommand{\freevar}{\mkern0.5\thinmuskip{\boldsymbol\cdot}\mkern0.5\thinmuskip}

\newcommand{\Union}\bigcup
\newcommand{\Isect}\bigcap
\newcommand{\union}\cup
\newcommand{\isect}\cap
\newcommand{\bigunion}\bigcup
\newcommand{\bigisect}\bigcap

\newcommand{\defeq}{\coloneqq}
\newcommand{\downto}{\searrow}

\newcommand{\subdiff}{\partial}
\newcommand{\ind}{\iota}
\newcommand{\noise}{\delta}

\DeclareMathOperator*{\argmin}{arg\,min}

\DeclareMathOperator*{\esssup}{ess\,sup}

\DeclareMathOperator{\closure}{cl}

\DeclareMathOperator{\conv}{conv}

\DeclareMathOperator{\epi}{epi}

\DeclareMathOperator{\graph}{Graph}

\def\Laplacian{\Delta}

\def\realopt#1{\widehat #1}

\newcommand{\setto}{\rightrightarrows}

\def\Primal{u}
\def\Dual{v}
\def\coPrimal{\xi}
\def\coDual{\eta}
\def\FdVar{z}
\def\coFdVar{\zeta}
\def\dcVar{z}
\def\DerivConeSym{V}
\def\DerivConeDSym{W}
\newcommand{\DerivCone}[1][\relax]{\DerivConeSym\ifx\relax#1\relax\else(#1)\fi}
\newcommand{\DerivConePrime}[1][\relax]{\tilde\DerivConeSym\ifx\relax#1\relax\else(#1)\fi}
\newcommand{\DerivConeX}[2][]{\DerivConeSym_{#2}\ifx\relax#1\relax\else(#1)\fi}
\newcommand{\DerivConeF}[1][]{\DerivConeX[#1]{\subdiff F^*}}
\newcommand{\DerivConeG}[1][]{\DerivConeX[#1]{\subdiff G}}
\newcommand{\DerivConeDXt}[3][]{\DerivConeDSym_{#3}^{#2}\ifx\relax#1\relax\else(#1)\fi}
\newcommand{\DerivConeDXtp}[3][]{\DerivConeDSym_{#3}^{\circ,#2}\ifx\relax#1\relax\else(#1)\fi}
\newcommand{\DerivConeDFt}[2][]{\DerivConeDXt[#1]{#2}{\subdiff F^*}}

\def\dir#1{\Delta #1}
\def\alt#1{#1'}
\def\altalt#1{#1''}

\def\theData{y^\noise}

\def\baseu{{\bar{u}}}

\def\realoptq{{\realopt{q}}}
\def\realoptu{{\realopt{u}}}
\def\realoptpsi{{\realopt{\Dual}}}

\def\lipnum{\ell}
\newcommand{\lip}[1]{\lipnum_{#1}}
\newcommand{\tildelip}[1]{{\widetilde \lipnum}_{#1}}

\newcommand{\kvbound}{b}
\newcommand{\grbound}{a}

\newcommand{\frechetCod}[1][]{\widehat D^*_{#1}}

\def\somesetmap{R}
\def\altsetmap{J}
\def\basesetmap{R_0}
\def\somesetmapfindim{R}

\usepackage[textsize=footnotesize,colorinlistoftodos]{todonotes}

\renewcommand{\L}{\mathcal{L}}

\newsavebox{\accentbox}
\newcommand{\compositeaccents}[2]{%
  \sbox\accentbox{$#2$}#1{\usebox\accentbox}}

\title{Stability of saddle points via explicit coderivatives of pointwise subdifferentials}

\author{%
    Christian Clason\thanks{Faculty of Mathematics, University Duisburg-Essen, 45117 Essen, Germany (\email{christian.clason@uni-due.de})}
    \and 
    Tuomo Valkonen\thanks{Department  of  Applied  Mathematics  and  Theoretical  Physics,  University  of  Cambridge, UK (\email{tuomov@iki.fi})}
}

\hypersetup{
    pdftitle={Stability of saddle points via explicit coderivatives of pointwise subdifferentials},
    pdfauthor={Christian Clason, Tuomo Valkonen},
    pdfkeywords={Metric regularity, Aubin property, Saddle points, Fréchet coderivatives, Integral functionals}
}

\begin{document}

\maketitle

\begin{abstract}
    We derive stability criteria for saddle points of a class of nonsmooth optimization problems in Hilbert spaces arising in PDE-constrained optimization, using metric regularity of infinite-dimensional set-valued mappings. A main ingredient is an explicit pointwise characterization of the regular coderivative of the subdifferential of convex integral functionals. This is applied to several stability properties for parameter identification problems for an elliptic partial differential equation with non-differentiable data fitting terms.
\end{abstract}

\section{Introduction}

This work is concerned with optimization problems of the form
\begin{equation}
    \label{eq:prob_convex}
    \min_{u\in X} G(u) + F(K(u))
\end{equation}
for proper, convex and lower semicontinuous functionals $G: X\to \Rbar:=\R\cup \{\infty\}$ and $F: Y\to\Rbar$ and a  Fréchet-differentiable operator $K:X\to Y$ between two Hilbert spaces $X$ and $Y$, in particular for integral functionals on $L^2(\Omega)$ of the form $F(u) = \int_\Omega f(u(x))\,dx$ for a convex integrand $f:\R\to\R$. 
Under suitable regularity assumptions, a minimizer $\bar u\in X$ satisfies 
\begin{equation}
\label{eq:oc}
\left\{\begin{aligned}
    -\bar p &\in \grad{K}(\bar u)^*(\partial F(K(\bar u))),\\
    \bar p &\in \partial G(\bar u),
\end{aligned}\right.
\end{equation}
for some $\bar p\in X^*$. Here, $\partial F$ denotes the convex subdifferential of $F$ and $\grad{K}(u)^*$ the adjoint of the Fréchet derivative of $K$ at $u$. Using convex duality, one can characterize the primal-dual pair $(\bar u,\bar p)$ via the saddle point ($\bar u,\bar \Dual$) of the Lagrangian
\begin{equation}
    \label{eq:lagrangian}
    L(u,\Dual) := G(u) + \inner{K(u),\Dual} - F^*(\Dual),
\end{equation}
where $F^*:Y^*\to\Rbar$
denotes the Fenchel conjugate of $F$; in this case, $\bar p$ and $\bar v$ are related via $\bar p = \grad K(\bar u)^* \bar \Dual$.

To fix ideas, a prototypical example is the $L^1$ fitting problem~\cite{ClasonJin:2011}
\begin{equation}
    \label{eq:l1fit_problem}
    \min_u \norm{S(u) - \theData}_{L^1} + \frac{\alpha}2 \norm{u}_{L^2}^2.
\end{equation}
Here, $G(u) = \frac{\alpha}2\norm{u}_{L^2}^2$, $K(u)=S(u) - \theData$, where $S$ maps $u$ to the solution $y$ of $-\Laplacian y+uy=f$ for given $f$, and $F(y) = \norm{y}_{L^1}$. This formulation is appropriate if the given data $y^\delta$ is corrupted by impulsive noise. Here, $F^*(\Dual) = \ind_{\{|\Dual(x)|\leq 1\}}(\Dual)$,
where $\ind_A$ denotes the indicator function of the set $A$ in the sense of convex analysis \cite{Hiriart-Urruty:1993}.
A second example is the Morozov (constrained) formulation of inverse problems appropriate for data subject to uniformly distributed noise \cite{Clason:2012},
\begin{equation}
    \label{eq:linffit_problem}
    \min_u \frac{\alpha}2 \norm{u}_{L^2}^2
    \quad\text{s.\,t.}\quad |S(u)(x) -\theData(x) |\leq \noise  \quad\text{a.\,e. in } \Omega.
\end{equation}
Here, $G$ and $K$ are as before, while $F(y) = \ind_{\{|y(x)|\leq \noise\}}(y)$ and hence $F^*(\Dual)=\noise \norm{\Dual}_{L^1}$. A similar problem arises in optimal control of partial differential equations with state constraints.

As we show in \cref{sec:saddle-vi}, critical points of the saddle point problem \eqref{eq:prob_convex} may be characterized concisely through the variational inclusion
\begin{equation}
    0 \in \basesetmap(\Primal, \Dual),
\end{equation}
where we define the set-valued mapping $\basesetmap: X \times Y \setto X \times Y$ by
\begin{equation}
    \basesetmap(\Primal, \Dual) \defeq
        \begin{pmatrix}
            \subdiff G(\Primal) + \grad K(\Primal)^* \Dual \\
            \subdiff F^*(\Dual) - K(\Primal)
        \end{pmatrix}.
\end{equation}
Our goal then is to study the stability of solutions to \eqref{eq:prob_convex} resp.~\eqref{eq:oc} through set-valued analysis of this mapping. 
The main tool is a form of Lipschitz continuity of $\inv\basesetmap$ known as the \emph{Aubin property} (or \emph{pseudo-Lipschitz} or \emph{Lipschitz-like} property). This is also called the \emph{metric regularity} of $\basesetmap$; see, e.g., \cite{Rockafellar:1998,Mordukhovich:2006,Mordukhovich:1993,Dontchev:2014,Ioffe:2015}.
In contrast, second-order stability analysis for PDE-constrained optimization problems is usually focused on the stability of the optimal values and of minimizers (as opposed to saddle-points) and is based on sufficient second-order conditions based on directional derivatives, often in stronger topologies; nonsmoothness typically arises from pointwise constraints or, more recently, sparsity penalties.
We refer to \cite{BonnansShapiro,Casas:2015} as well as the literature cited therein.

Since the problem \eqref{eq:prob_convex} is nonsmooth, the first-order conditions involve proper subdifferential inclusions and hence the second-order analysis required for showing metric regularity involves set-valued derivatives. Considerable effort has been expended on obtaining explicit representations of these derivative, although up to now primarily in the finite-dimensional setting, e.g., in \cite{Mordukhovich:2001,Henrion:2002,Henrion:2013a,Henrion:2014,Outrata:2015a}, with a focus on normal cones arising from inequality constraints.
The difficulty in the infinite-dimensional setting stems from the fact that there exists a variety of more or less abstract definitions of such objects, see, e.g., \cite{Mordukhovich:2006}, although more explicit characterizations can be obtained in some concrete situations \cite{Henrion:2010}. 
Here, by exploiting the fact that the nonsmooth functionals are defined pointwise via convex integrands, we are able to explicitly compute regular coderivatives pointwise using the finite-dimensional theory from \cite{Rockafellar:1998}; see also~\cite{Mordukhovich:2006} for further developments on their calculus. One of the main contributions of this work is therefore to further narrow the gap between the concrete finite-dimensional and the abstract infinite-dimensional settings.

Besides being of inherent interest, e.g., for showing stability of the solution of the inverse problem with respect to $\noise$, metric regularity is also relevant to convergence of optimization methods. In the context of the saddle point problem for the Lagrangian \eqref{eq:lagrangian}, it is required for the nonlinear primal-dual hybrid gradient method of \cite{Valkonen:2008}. More widely, through the equivalence \cite{Bolte:2010} of the Aubin property to the recently eminent Kurdyka-{\polishL}ojasiewicz property \cite{Kurdyka:1998,Lojasiewicz:1970,Lojasiewicz:1963}, metric regularity is relevant to the convergence of a wide range of descent methods \cite{Bolte:2013}. It can also be used to directly characterize the convergence of certain basic optimization methods \cite{Bolte:2010,Klatte:2002,Klatte:2009}.
Metric regularity is also closely related to the concept of tilt-stability, mainly studied in finite dimensions, see, e.g., \cite{Rockafellar:1998b,Mordukhovich:2012,Drusvyatskiy:2013,Eberhard:2012,Lewis:2013}, but recently also in infinite dimensions \cite{Mordukhovich:2013,Outrata:2015a}. An extended concept incorporating tilt stability is that of full stability \cite{Levy:2000,Mordukhovich:2014}.

We also note that when the non-linear saddle-point problem can be written as the minimization of a difference of convex functions -- as any $C^2$ objective can \cite{Tuy:1995} -- detailed characterizations exist in the finite-dimensional case of local minima \cite{Valkonen:2010} and sensitivity \cite{Valkonen:2008}. Moreover, a set-valued analysis of the solvability of such programs with further symmetric cone structure is performed in \cite{Valkonen:2013}. 
In certain cases, with a finite-dimensional control $u$ in an otherwise infinite-dimensional problem, it is also possible to do away with the regularizer $G$ \cite{DeLosReyes:2015a}.

This work is organised as follows. In \cref{sec:pointwise_deriv}, we derive pointwise characterizations of second-order subdifferentials or generalized Hessians of integral functionals on $L^2$ and give examples for several functionals commonly occurring in variational methods for inverse problems, image processing, and PDE-constrained optimization. These results are used in \cref{sec:stability_vi} to give an explicit form of the Mordukhovich criterion  for set-valued mappings in Hilbert spaces, in particular for those arising from subdifferentials of the integral functionals considered in the preceding section. \cref{sec:stability_sp} further specializes this to the case of set-valued mappings arising from the first-order optimality conditions \eqref{eq:oc} and gives sufficient conditions for several stability properties such as stability with respect to perturbation of the data. Finally, \cref{sec:stability_paramid} discusses the satisfiability of these conditions in the specific case of the model parameter identification problems \eqref{eq:l1fit_problem} and~\eqref{eq:linffit_problem}, where it will turn out that stability can only be guaranteed after either introducing a Moreau--Yosida regularization or a projection to a finite-dimensional subspace in~$F$.

\section{Derivatives and coderivatives in \texorpdfstring{$\scriptstyle L^2(\Omega)$}{L²(Ω)}}\label{sec:pointwise_deriv}

Sadly, we cannot as in \cite{Valkonen:2014} directly use the clean finite-dimensional theory from \cite{Rockafellar:1998} to show the Aubin property of $\inv \basesetmap$ through the Mordukhovich criterion \cite{Mordukhovich:1992}.
We have to delve into the various complications of the infinite-dimensional setting as presented in \cite{Mordukhovich:2006}. The first one is the multitude of different definitions of set-valued generalized derivatives.
Luckily however, as it will turn out, because of the \emph{pointwise nature} of the non-smooth functionals whose second derivatives we require, we will be able to compute the pointwise differentials using the finite-dimensional theory, and limit ourselves to the regular coderivative in infinite dimensions.
Although the results of this section hold for integral functionals on $L^p$ for any $1\leq p<\infty$, we restrict the presentation to the Hilbert space $L^2$ for simplicity (and since we will make use of the Hilbert space structure of the saddle-point problem \eqref{eq:oc} later anyway).

\subsection{Set-valued mappings and coderivatives}

We first collect some notations and definitions for set-valued mappings in Hilbert spaces, following \cite{Rockafellar:1998,Mordukhovich:2006} and simplifying the setting of the latter to Hilbert spaces.
The symbols $X, Y, Q$, and $W$ generally stand for (infinite-dimensional) Hilbert spaces, which we identify throughout with their duals via the Riesz isomorphism.
With $x \in X$, we then denote by $B(x, r)$ the open ball of radius $r>0$. The closure of a set $A$ we denote by $\closure A$.

\begin{definition}
    Let $U \subset X$ for $X$ a Hilbert space. Then we define the set of \emph{Fréchet (or regular) normals} to $U$ at $u\in U$ by 
    \[
        \widehat N(u; U)
        \defeq
        \left\{
            z \in X
            \,\middle|\,
            \limsup_{U \ni \alt{u} \to u}\frac{\iprod{z}{\alt{u}-u}}{\norm{\alt{u}-u}} \le 0
        \right\}
    \]
    and the set of \emph{tangent vectors} by 
    \begin{equation}
        \label{eq:tangent}
        T(u; U) \defeq \left\{ z \in X \,\middle|\, \text{exist } \tau^i \downto 0 \text{ and } u^i \in U \text{ such that } z=\lim_{i \to \infty} \frac{u^i-u}{\tau^i} \right\}.
    \end{equation}
For a convex set $U$, these coincide with the usual normal and tangent cones of convex analysis.
\end{definition}

For our general results, we will need to impose some geometric regularity assumptions.
\begin{definition}
    \label{def:geomderiv}
    We say that a tangent vector $z \in T(u; U)$ is \term{derivable} if there exists an $\eps>0$ and a curve $\xi: [0, \eps] \to U$ such that
    \[
        \xi_+'(0) \defeq \lim_{t \downto 0} \frac{\xi(t)-\xi(0)}{t}
    \]
    exists with $\xi_+'(0)=z$ and $\xi(0)=u$.
    We say that $U$ is \term{geometrically derivable} if for every $u \in U$, every $z \in T(u; U)$ is derivable.
\end{definition}

It is easy to see that (cf.~\cite[Prop.~6.2]{Rockafellar:1998}) $U$ is geometrically derivable if and only if $T(u; U)$ for each $u \in U$ is defined by a full limit, i.e., we replace in \eqref{eq:tangent} the existence of $\tau^i \downto 0$ by the requirement of the existence of $u^i$ for any sequence $\tau^i \downto 0$.

For any cone $\DerivCone \subset X$, we also define the \emph{polar cone}
\begin{equation}
    \label{eq:polar}
    \polar \DerivCone \defeq \{ \dcVar \in X \mid \iprod{\dcVar}{\dcVar'} \le 0 \text{ for all } \dcVar' \in \DerivCone\}.
\end{equation}

We use the notation $\somesetmap: Q \setto W$ to denote a set-valued mapping $\somesetmap$ from $Q$ to $W$; i.e., for every $q \in Q$ holds $\somesetmap(q) \subset W$. For $\somesetmap: Q \setto W$, we define the domain $\dom \somesetmap \defeq \{q\in Q\mid\somesetmap(q)\neq \emptyset\}$ and graph $\graph\somesetmap \defeq \{(q,w)\subset Q\times W\mid w\in\somesetmap(q)\}$. 
The regular coderivatives of such maps are defined graphically with the help of the normal cones. 
\begin{definition}
    \label{def:frechetcod}
    Let $Q$ and $W$ be Hilbert spaces, and
    $\somesetmap: Q \setto W$ with $\dom \somesetmap \ne \emptyset$. 
    We then define the \emph{regular coderivative} 
    $\frechetCod \somesetmap(q | w): W \setto Q$ of $\somesetmap$ at $q\in Q$ for $w\in W$ as 
    \begin{equation}
        \label{eq:frechetcod}
        \frechetCod \somesetmap(q | w)(\dir{w}) \defeq
        \left\{ \dir{q} \in Q \mid (\dir{q}, -\dir{w}) \in \widehat N((q, w); \graph \somesetmap) \right\}.
    \end{equation}
    We also define the \emph{graphical derivative} $D\somesetmap(q|w): Q \setto W$ by
    \begin{equation}
        \label{eq:graphderiv-0}
        D \somesetmap(q | w)(\dir{q}) \defeq
        \left\{ \dir{w} \in W \mid (\dir{q}, \dir{w}) \in T((q, w); \graph \somesetmap) \right\}.
    \end{equation}
\end{definition}

The graphical derivative may also be written as \cite{Mordukhovich:2006,Rockafellar:1998}
\begin{equation}
    \label{eq:graphderiv}
    D\somesetmap(q|w)(\dir{q}) = \limsup_{t \downto 0,\, \dir{\alt{q}} \to \dir{q}} \frac{\somesetmap(q+t \dir{\alt{q}})-w}{t}.
\end{equation}
Here $\limsup_{\alt{t} \to t} A_{\alt{t}}$ stands for the \emph{outer limit} of a sequence of sets $\{A_t \subset W\}_{t \in T}$ over an index set $T$, defined as
\[
    \limsup_{t \to \infty} A_t \defeq \bigl\{w \in W \bigm| \text{for each } i \in \N \text{ exist } t^i \in T,\, w^i \in A_{t^i},\,\text{s.t. } t^i \to t\text{ and } w^i \to w\bigr\}.
\]
Observe that $D\somesetmap(q|w): Q \setto W$ whereas $\frechetCod\somesetmap(q|w): W \setto Q$. Indeed, if $\somesetmap(q)=Aq$ for a linear operator $A$ between Hilbert spaces, then for $w=Aq$ holds
\[
    D\somesetmap(q|w) = A \qquad\text{and}\qquad \frechetCod\somesetmap(q|w)=A^*.
\]
The former is immediate from  \eqref{eq:graphderiv} (see also, e.g., \cite[Ex.~8.34]{Rockafellar:1998}), while the latter is contained in \cite[Cor.~1.39]{Mordukhovich:2006}.

We say that $\graph \somesetmap$ is \emph{locally closed} at $(q,w)$, if there exists a closed set $U \subset Q \times W$ such that $\graph \somesetmap \isect U$ is closed.
For any convex lower semicontinuous function $f: Q\to\Rbar$‚ the graph of the subdifferential $\partial f$, considered as a set-valued mapping, is closed.
This is an immediate consequence of the definition of the convex subdifferential.
Finally, $\somesetmap$ is called \emph{proto-differentiable} if $\graph \somesetmap$ is geometrically derivable. 

\subsection{Second-order derivatives of pointwise functionals}
\label{sec:second-deriv}

Let $X=L^2(\Omega; \R^m)$ for an open domain $\Omega \subset \R^n$ and $G:X\to\Rbar$ be given by
\begin{equation}
    \label{eq:g-pointwise-integral}
    G(u)=\int_\Omega g(x, u(x)) \,d x
\end{equation}
for some integrand $g: \Omega \times \R^m \to (-\infty, \infty]$.
Here we assume that
\begin{enumerate}[label=(\roman*)]
    \item $g$ is normal, i.e., the epigraphical mapping $x\mapsto \epi g(x,\cdot)\subset \R^m\times \R$ is closed-valued and measurable,
    \item $g$ is proper and convex, i.e., the mapping $z\mapsto g(x,z)$ is proper and convex for each fixed $x\in\Omega$.
    \item $\subdiff g$ is \emph{pointwise a.\,e.~proto-differentiable}, i.e., the mapping $z\mapsto \partial_z g(x,z)$ is proto-differentiable for a.\,e. $x\in\Omega$.
\end{enumerate}
We call an integrand satisfying (i--iii) \emph{regular}.
Note that (i) already implies that the mapping $z\mapsto g(x,z)$ is lower semicontinuous for each fixed $x\in\Omega$ and that $g(x,u(x))$ is measurable for each $u\in X$ \cite[Prop.~14.28]{Rockafellar:1998}. 
Examples of normal integrands are finite-valued Carathéodory functions \cite[Ex.~14.29]{Rockafellar:1998} and indicator functions of a closed-valued Borel-measurable mapping $C: \Omega \setto \R^m$ \cite[Ex.~14.32]{Rockafellar:1998}. 
For a normal integrand, \eqref{eq:g-pointwise-integral} is well-defined, and $G(u) <\infty$ if and only if $u(x)\in\dom g(x,\cdot)$ almost everywhere \cite[Prop.~14.58]{Rockafellar:1998}. 

Proto-differentiability is more restrictive but holds for a large class of practically relevant examples. In particular, under the present assumptions that $g(x, \freevar)$ is a proper, convex, lower semicontinuous function on a finite-dimensional domain, (iii) is equivalent to $g(x, \freevar)$ being \emph{twice epi-differentiable}; see \cite[Def.~13.6]{Rockafellar:1998} for the (technical) definition as well as \cite[Prop.~8.41, Ex.~13.30, Thm.~13.40]{Rockafellar:1998}. 
It is therefore satisfied, e.g., for the maximum of a finite number of $C^2$ functions \cite[Ex.~13.16]{Rockafellar:1998} and for proper, convex and piecewise linear-quadratic functions \cite[Prop.~13.9]{Rockafellar:1998}.
More general ways to verify the proto-differentiability of $\subdiff g(x, \freevar)$ even without convexity include the concepts of \emph{full amenability} \cite[Def.~10.23 \& Cor.~13.41]{Rockafellar:1998} and \emph{prox-regularity} \cite[Def.~13.27 \& Thm.~13.40]{Rockafellar:1998}.

Since $X=L^2(\Omega;\R^m)$ is decomposable, it suffices to have existence of at least one $u_0\in \dom G$ to be able to compute pointwise the Fenchel conjugate 
\begin{equation}
    \label{eq:g-pointwise-conjugate}
    G^*(u) = \int_\Omega g^*(x, u(x)) \,d x
\end{equation}
and the convex subdifferential
\begin{equation}
    \label{eq:g-pointwise-subdiff}
    \subdiff G(u) = \{ \coPrimal \in X
                        \mid
                        \coPrimal(x) \in \subdiff g(x, u(x)) \text{ for a.\,e. } x \in \Omega \},
\end{equation}
where conjugate and subdifferential of $g(x,z)$ are understood as taken with respect to $z$ for $x$ fixed; see \cite[Thm.~3C]{Rockafellar:1976} and \cite[Cor.~3F]{Rockafellar:1976}, respectively.

In order to calculate $\frechetCod \subdiff G(x)$, we observe that
\[
    \graph[\subdiff G]=\{ (u, \coPrimal) \in X \times X \mid \coPrimal(x) \in \subdiff g(x, u(x))
                                                     \text{ for a.\,e. } x \in \Omega \}.
\]
Since $g$ is normal and convex, $\graph \subdiff g$ is measurable and closed-valued \cite[Prop.~14.56]{Rockafellar:1976}.
Thus, in the definition of $\widehat N(\hat u, \hat \coPrimal; \graph[\subdiff G])$, we are dealing with a sequence in
\[
    X \times X = [L^2(\Omega; \R^m)]^2 \simeq L^2(\Omega; \R^{2m}).
\]
To derive an expression for $\frechetCod \subdiff G$, it therefore suffices to prove the following result.

\begin{proposition}
    \label{prop:pointwise-normal-cone}
    Let $U \subset L^2(\Omega; \R^m)$ have the form
    \[
        U = \left\{ u \in L^2(\Omega; \R^m)
            \mid u(x) \in C(x) \text{ for a.\,e. } x \in \Omega
        \right\}
    \]
    for a Borel-measurable mapping $C: \Omega \setto \R^m$ with $C(x) \subset \R^m$ geometrically derivable for almost every $x \in \Omega$.
    Then for every $u \in U$ we have
    \begin{equation}
        \label{eq:l2-pointwise-normal-cone}
        \widehat N(u; U) = \left\{ z \in L^2(\Omega; \R^m)
                             \mid z(x) \in \widehat N(u(x); C(x)) \text{ for a.\,e. } 
                                  x \in \Omega
                              \right\},
    \end{equation}
    and
    \begin{equation}
        \label{eq:l2-pointwise-tangent-cone}
        T(u; U) = \left\{ z \in L^2(\Omega; \R^m)
                             \mid z(x) \in T(u(x); C(x)) \text{ for a.\,e. } 
                                  x \in \Omega
                              \right\}.
    \end{equation}
\end{proposition}

\begin{proof}
    We start with \eqref{eq:l2-pointwise-normal-cone}.
    Recalling the definition of $\widehat N(u; U)$, we have to find all $z \in L^2(\Omega; \R^m)$ satisfying
    \[
        \limsup_{U \ni \alt{u} \to u}
            \frac{\iprod{z}{\alt{u}-u}}{\norm{\alt{u}-u}} \le 0,
    \]
    where the inner product and norm are now in $L^2(\Omega; \R^m)$, and
    the convergence is strong convergence in this space, within the subset $U$.

    Let us take a sequence $u^i \to u$, ($i=1,2,3,\ldots$) and let $\eps > 0$ be arbitrary. We denote
        $v^i \defeq u^i-u$.
    Then,
    \[
        L_i \defeq \frac{\iprod{z}{u^i-u}}{\norm{u^i-u}}
        =
        \frac{\int_\Omega \iprod{z(x)}{v^i(x)} \,d x}
             {\left(\int_\Omega \abs{v^i(x)}^2 \,d x\right)^{1/2}}.
    \]
    We let $Z_1=Z_1^i$, $Z_2$ and $Z_3$ be sets with Lebesgue measure $\L^m(\Omega \setminus Z_j) \le \eps/3$ for each $j=1,2,3$ and satisfying, respectively, the conditions 
    \begin{align}
        \label{eq:pointwise-norm-estimate}
        &\vnorm{v^i(x)} \le 3\inv\eps \norm{v^i}
        \quad
        (x \in Z_1^i),
        \\
        \label{eq:pointwise-boundedness}
        &z \text{ is bounded on } Z_2,\\
        \intertext{and}
        \label{eq:pointwise-remainder-integral}
        &\left(\int_{\Omega \setminus Z_3} \vnorm{z(x)}^2 \,d x\right)^{1/2}
        \le \eps/3.
    \end{align}
    To see how \eqref{eq:pointwise-norm-estimate} can hold, we take $\tilde Z_1^i$ as the set of $x \in \Omega$ satisfying \eqref{eq:pointwise-norm-estimate}. Then
    \[
        \norm{v^i}^2
        \ge \int_{\Omega \setminus \tilde Z_1^i} \vnorm{v^i(x)}^2 \,d x
        \ge 3\inv\eps \L^m(\Omega \setminus \tilde Z_1^i) \norm{v^i}^2,
    \]
    which gives
    \[
        \eps/3 \ge \L^m(\Omega \setminus \tilde Z_1^i).
    \]
    We may therefore take $Z_1^i = \tilde Z_1^i$.
    The proofs that \eqref{eq:pointwise-boundedness} and \eqref{eq:pointwise-remainder-integral} can hold are similarly elementary.

    With
    \[
        Z^i \defeq Z_1^i \isect Z_2 \isect Z_3,
    \]
    we have
    \[
        \L^m(\Omega \setminus Z^i) \le \eps.
    \]
    We calculate
    \[
        \begin{aligned}
            L_i
            &
            = \frac{\int_{\Omega \setminus Z^i} \iprod{z(x)}{v^i(x)} \,d x}{\norm{v^i}}
              +\frac{\int_{Z^i} \iprod{z(x)}{v^i(x)} \,d x}{\norm{v^i}}
            \\
            &
            \le \frac{\norm{\chi_{\Omega \setminus Z^i}z} \norm{v^i}}{\norm{v^i}}
                + \int_{Z^i} \frac{\iprod{z(x)}{v^i(x)}}{\vnorm{v^i(x)}} 
                   \cdot \frac{\abs{v^i(x)}}{\norm{v^i}} \,d x
            \\
            &
            \le \norm{\chi_{\Omega \setminus Z^i}z}
                + 3\inv\eps \int_{Z^i} \max\left\{0, \frac{\iprod{z(x)}{v^i(x)}}{\vnorm{v^i(x)}}\right\} \,d x.
            \\
            &
            \le \norm{\chi_{\Omega \setminus Z^i}}\norm{z}
                + 3\inv\eps \int_{Z_2} \max\left\{0, \frac{\iprod{z(x)}{v^i(x)}}{\vnorm{v^i(x)}}\right\} \,d x.
            \\
            &
            \le \eps^{1/2}\norm{z}
                + 3\inv\eps \int_{Z_2} \max\left\{0, \frac{\iprod{z(x)}{v^i(x)}}{\vnorm{v^i(x)}}\right\} \,d x.
        \end{aligned}
    \]
    If now for almost every $x \in \Omega$ we have that $z(x) \in \widehat N(u(x); C(x))$, then we deduce using the boundedness of $z$ on $Z_2$ and reverse Fatou's inequality that
    \[
        \limsup_{i \to \infty} L_i \le \eps^{1/2} \norm{z}.
    \]
    Since $\eps>0$ was arbitrary, we deduce
    \[
        \widehat N(u; U) \supset \{ z \in L^2(\Omega; \R^m)
                             \mid z(x) \in \widehat N(u(x); C(x)) \text{ for a.\,e. } 
                                  x \in \Omega
                             \}.
    \]
    This proves one direction of \eqref{eq:l2-pointwise-normal-cone}, which therefore holds even without geometric derivability. Now we have to prove the other direction, where we do need this assumption.

    So, let $z \in \widehat N(u; U)$. We have to show that $z(x) \in \widehat N(u(x); C(x))$ for a.\,e.~$x \in \Omega$. Suppose this does not hold. Using the standard polarity relationship $\widehat N(u(x); C(x))=\polar{[T(u(x); C(x))]}$, e.g., from \cite[Thm.~6.28]{Rockafellar:1998}, we can find $\delta > 0$ and a Borel set $E \subset \Omega$ of finite positive Lebesgue measure such that for each $x \in E$ there exists $w(x) \in T(u(x); C(x))$ with $\norm{w(x)}=1$ and $\iprod{z(x)}{w(x)} \ge \delta$. 
    We may without loss of generality take $C(x)$ geometrically derivable for each $x \in E$. By \cref{def:geomderiv} there then exists for each $x \in E$ a curve $\xi(\freevar, x): [0, \eps(x)] \to C(x)$ such that $\xi'_+(0, x)=w(x)$ and $\xi(0, x)=u(x)$. Let us pick $c \in (0, \delta)$. By replacing $E$ by a subset of positive measure, we may by Egorov's theorem assume the existence of $\eps > 0$ such that
    \[
        \vnorm{\xi(t, x)-\xi(0, x)-w(x)t} \le c t
        \quad (t \in [0, \eps],\, x \in E).
    \]

    Let us define
    \[
        \tilde u^t(x) \defeq \begin{cases}
                        \xi(t, x), & x \in E, \\
                        u(x), & x \in \Omega \setminus E.
                    \end{cases}
    \]
    With $v^t \defeq \tilde u^t -u$, we have $v^t(x)=\xi(t, x)-\xi(0, x)$ for $x \in E$, and $v^t(x)=0$ for $x \in \Omega \setminus E$. Therefore, for $t \in (0, \eps]$ and some $c'>0$ there holds
    \[
        \norm{v^t}
        = \left(\int_E \vnorm{\xi(t, x)-\xi(0, x)}^2\, d x\right)^{1/2}
        \le  \left(\int_E (\vnorm{w(x)}t+ ct)^2\, d x\right)^{1/2}
        \le c't.
    \]
    Likewise
    \[
        \iprod{z(x)}{v^t(x)}
        \ge \iprod{z(x)}{w(x)}t - \vnorm{z(x)} \cdot \vnorm{\xi(t, x)-\xi(0, x)-wt}
        \ge \delta t - ct.
    \]
    It follows
    \begin{equation}
        \limsup_{t \downto 0} \int_{E} \frac{\iprod{z(x)}{v^t(x)}}{\norm{v^t}} \,d x
        \ge
        \limsup_{t \downto 0} 
        \frac{\L^m(E)(\delta t - c t)}{c't}
        = \frac{\L^m(E)(\delta-c)}{c'} > 0.
    \end{equation}
    With $u^i \defeq \tilde u^{1/i}$ for $i \in \N$, it follows that $\lim_{i \to \infty} L_i > 0$.
    This provides a contradiction to $z \in \widehat N(u; U)$. Thus $z(x) \in \widehat N(u(x); C(x))$ for a.\,e.~$x \in \Omega$, finishing the proof of \eqref{eq:l2-pointwise-normal-cone}.

    \bigskip

    We still have to show \eqref{eq:l2-pointwise-tangent-cone}. The inclusion
    \[
        T(u; U) \subset \left\{ z \in L^2(\Omega; \R^m)
                             \mid z(x) \in T(u(x); C(x)) \text{ for a.\,e. } 
                                  x \in \Omega
                              \right\}
    \]
    follows from the defining equation \eqref{eq:tangent} and the fact that a sequence convergent in $L^2(\Omega)$ converges, after possibly passing to a subsequence, pointwise almost everywhere. 

    For the other direction, we take for almost every $x \in \Omega$ a tangent vector $z(x) \in T(u(x); C(x))$ at $u(x) \in C(x)$. For the inclusion \eqref{eq:l2-pointwise-tangent-cone}, we only need to consider the case $z \in L^2(\Omega; \R^m)$. By geometric derivability, we may find for a.\,e. $x \in \Omega$ an $\eps(x)>0$ and a curve $\xi(\freevar, x): [0, \eps(x)] \to C(x)$ such that $\xi(0, x)=u(x)$ and $\xi'_+(0, x)=z(x)$. In particular, for any given $c>0$, we may find $\eps_c(x) \in (0, \eps(x)]$ such that
    \begin{equation}
        \label{eq:tangent-xi}
        \frac{\vnorm{\xi(t, x)-\xi(0, x)-z(x)t}}{t} \le c
        \quad (t \in (0, \eps_c(x)],\, \text{a.\,e. } x \in \Omega).
    \end{equation}
    For $t>0$, let us set
    \[
        E_{c,t} \defeq \{ x \in \Omega \mid \eps_c(x)\geq t \}.
    \]
    If we define
    \[
        \tilde u^{c,t}(x) \defeq 
        \begin{cases}
            \xi(t, x), & x \in E_{c,t}, \\
            u(x), & x \in \Omega \setminus E_{c,t},
        \end{cases}
    \]
    then by \eqref{eq:tangent-xi}, $\vnorm{\tilde u^{c,t}(x)-u(x)} \le t (c+\vnorm{z(x)})$ for a.\,e.~$x \in \Omega$, so that
    \begin{equation}
        \label{eq:tangent-approx}
        \norm{\tilde u^{c,t}-u} \le \left(\int_\Omega t^2 (c+\vnorm{z(x)})^2\right)^{1/2}
        \le t (c\sqrt{\L^m(\Omega)}+\norm{z}).
    \end{equation}
    Moreover, \eqref{eq:tangent-xi} also gives
    \begin{equation}
        \label{eq:tangent-construct}
        \begin{aligned}[t]
        \frac{\norm{\tilde u^{c,t}-u-z t}}{t}
        &
        \le \frac{1}{t}
         \left(\int_{E_{c,t}} \vnorm{\xi(t,x)-\xi(0,x)-z(x)t}^2 \, d x\right)^{1/2}
         \!+
            \frac{1}{t}
            \left(\int_{\Omega \setminus E_{c,t}} \vnorm{z(x)t}^2 \, d x\right)^{1/2}
        \\
        &
        \le c\sqrt{\L^m(\Omega)} + \norm{z\chi_{\Omega \setminus E_{c,t}}}.
        \end{aligned}
    \end{equation}
    For each $i \in \N$ we can find $t_i>0$ such that $\norm{z\chi_{\Omega \setminus E_{1/i,t_i}}} \le 1/i$. This follows from Lebesgue's dominated convergence theorem and the fact that $\L^m(\Omega \setminus E_{c,t}) \to 0$ as $t \to 0$.
    The estimates \eqref{eq:tangent-approx} and \eqref{eq:tangent-construct} thus show that $u^i \defeq u^{1/i,t_i}$ satisfy $u^i \to u$ and $(u^i-u)/t_i \to z$. Therefore $z \in T(u; U)$, finishing the proof of \eqref{eq:l2-pointwise-tangent-cone}.
\end{proof}

As a corollary, we may calculate $\frechetCod \subdiff G(u|w)$ for $G$ of the form \eqref{eq:g-pointwise-integral}.

\begin{corollary}
    \label{corollary:g-form}
    Let $G: L^2(\Omega; \R^m) \to \R$ have the form \eqref{eq:g-pointwise-integral} for some regular integrand $g$.
    Then the regular coderivative of $\subdiff G$ at $u$ for $\coPrimal$ in the direction $\dir\coPrimal$, where $u,\coPrimal,\dir\coPrimal \in L^2(\Omega; \R^m)$, is given by
    \begin{equation}
        \label{eq:g-form}
        \frechetCod{[\subdiff G]}(u|\coPrimal)(\dir\coPrimal)
        =
        \left\{
            \dir{u} \in L^2(\Omega; \R^m)
            \,\middle|\,
            \begin{array}{r}
            \dir{u}(x) \in \frechetCod{[\subdiff g(x, \freevar)]}(u(x)|\coPrimal(x))(\dir\coPrimal(x)) \\
            \text{ for a.\,e. } x \in \Omega
            \end{array}
        \right\}.
    \end{equation}
    Likewise, the graphical derivative at $u$ for $\coPrimal$ in the directon $\dir u$ is given by 
    \begin{equation}
        \label{eq:g-form-graph}
        D{[\subdiff G]}(u|\coPrimal)(\dir u)
        =
        \left\{
            \dir{\coPrimal} \in L^2(\Omega; \R^m)
            \,\middle|\,
            \begin{array}{r}
            \dir{\coPrimal}(x) \in D{[\subdiff g(x, \freevar)]}(u(x)|\coPrimal(x))(\dir u(x)) \\
            \text{ for a.\,e. } x \in \Omega
            \end{array}
        \right\}.
    \end{equation}
\end{corollary}

\begin{proof}
    As we have already remarked in the beginning of the present \cref{sec:second-deriv}, the sets
    \[
        C(x) \defeq \{ (\FdVar, \coPrimal) \in \R^m \times \R^m \mid \coPrimal \in \subdiff g(x, \FdVar) \}
    \]
    are by \cite[Prop.~8.41, Ex.~13.30, Thm.~13.40]{Rockafellar:1998} geometrically derivable for regular integrands $g$. The present result therefore follows by direct application of \cref{prop:pointwise-normal-cone} to the set $U=\graph \subdiff G$ with $\subdiff G$ given in \eqref{eq:g-pointwise-subdiff}.
\end{proof}

More generally, we have the following.
\begin{corollary}
    \label{corollary:p-form}
    Let $P: Q \sim L^2(\Omega; \R^m) \setto W \sim L^2(\Omega; \R^k)$ have the form
    \[
        P(q) = \{ w \in L^2(\Omega; \R^m) \mid w(x) \in p(x, q(x))\, \text{for a.\,e. }  x \in \Omega\}
    \]
    for some Borel-measurable and pointwise a.\,e.~proto-differentiable set-valued function $p: \Omega \times \R^m \setto \R^k$.
    Then the regular coderivative of $P$ at $q$ for $w$ in the direction $\dir{w}$, where $q \in L^2(\Omega; \R^m)$, and $w,\dir{w} \in L^2(\Omega; \R^k)$, is given by
    \[
        \frechetCod{P}(q|w)(\dir{w})
        =
        \left\{
            \dir{q} \in L^2(\Omega; \R^m)
            \,\middle|\,
            \begin{array}{r}
            \dir{q}(x) \in \frechetCod{[p(x, \freevar)]}(q(x)|w(x))(\dir{w}(x)), \\
            \text{ for a.\,e. } x \in \Omega
            \end{array}
        \right\}.
    \]
    Likewise,
    \[
        D{P}(q|w)(\dir{q})
        =
        \left\{
            \dir{w} \in L^2(\Omega; \R^m)
            \,\middle|\,
            \begin{array}{r}
            \dir{w}(x) \in D{[p(x, \freevar)]}(q(x)|w(x))(\dir{q}(x)), \\
            \text{ for a.\,e. } x \in \Omega
            \end{array}
        \right\}.
    \]
\end{corollary}

\begin{corollary}
    \label{cor:cod-p-plus-smooth}
    Let $P$ be a pointwise set-valued functional as in \cref{corollary:p-form}, and let $h: Q \to W$ be single-valued and Fréchet differentiable. Then 
    \begin{equation}
        \label{eq:d-p-plus-smooth-cod}
        \frechetCod{(P+h)}(q|w)(\dir{w})
        = \frechetCod{P}(q|w-h(q))(\dir{w}) + [\grad h(q)]^*\dir{w},
    \end{equation}
    and
    \begin{equation}
        \label{eq:d-p-plus-smooth}
        D{(P+h)}(q|w)(\dir{q})
        = DP(q|w-h(q))(\dir{q}) + \grad h(q)\dir{q}.
    \end{equation}
\end{corollary}

\begin{proof}
    Similarly to the finite-dimensional case in \cite[Ex.~10.43]{Rockafellar:1998}, the rule \eqref{eq:d-p-plus-smooth} follows immediately from the defining equation \eqref{eq:graphderiv}.
    The rule \eqref{eq:d-p-plus-smooth-cod} is immediate consequence of the sum rule for regular coderivatives \cite[Thm.~1.62]{Mordukhovich:2006}.
\end{proof}

\begin{remark}
    Using \eqref{eq:l2-pointwise-normal-cone}, it is not difficult to obtain the characterization
    \begin{equation}
        \label{eq:pointwise-limiting-normal}
        N(u; U) = \left\{ z \in L^2(\Omega; \R^m)
                             \mid z(x) \in N(u(x); C(x)) \text{ for a.\,e. } 
                                  x \in \Omega
                              \right\}
    \end{equation}
    of the \term{limiting normal cone} $N(u; U) \defeq \limsup_{U \ni u' \to u} \widehat N(u'; U)$. The proof is based on $L^2$ convergence giving pointwise a.\,e.~convergence for a subsequence, and in the other direction, reindexing finite-dimensional sequences to get sequences convergent in $L^2$.
    The expression \eqref{eq:pointwise-limiting-normal} then allows obtaining corresponding versions of the corollaries above for the \term{limiting coderivative} $D^*$ (which enjoys a richer calculus) in place of the regular coderivative $\widehat D^*$.
    
    However, the stability analysis which is the focus of this work rests on the relation between the regular coderivative and the graphical derivative, discussed in the following section, which does not hold for the limiting coderivative (which has a similar relation with the regular derivative). In particular, we cannot work in the same way with the convexified graphical derivative which is a key step in our analysis (see \cref{sec:linear-cone} below).
    Hence, we do not treat this case in detail here.
\end{remark}

\subsection{The finite-dimensional coderivative in terms of the graphical derivative}
\label{sec:findimco}

\Cref{corollary:g-form} and \cref{corollary:p-form} give us computable expressions for the coderivative for pointwise set-valued mappings in infinite dimensions in terms of the \emph{co}derivative in finite dimensions. It is, however, often easier to work with the graphical derivative \eqref{eq:graphderiv}. From \cite[Prop.~8.37]{Rockafellar:1998} we find for $\somesetmapfindim: \R^m \setto \R^n$ that
\begin{equation}
    \label{eq:upper-adj-cod}
    \frechetCod \somesetmapfindim(q|w) = [D\somesetmapfindim(q|w)]^{*+}.
\end{equation}
Here, for a general set-valued mapping $\altsetmap: Q \setto W$, the \emph{upper adjoint} $\altsetmap^{*+}:W\setto Q$ is defined via
\[
    \altsetmap^{*+}(\dir{w})
    \defeq
    \{
        \dir{q} \in Q
        \mid
        \iprod{\dir{q}}{\dir{\alt{q}}} \le \iprod{\dir{w}}{\dir{\alt{w}}}
        \text{ when }
        \dir{\alt{w}} \in \altsetmap(\dir{\alt{q}})
    \}.
\]

In general, the graph of the regular coderivative need not be a convex set. It is often more convenient -- and for our analysis sufficient -- to work with its convexification.
To see this, observe first that by definition of the upper adjoint, and minding the negative sign of $\dir{\alt{w}}$ in \eqref{eq:frechetcod}, the relation \eqref{eq:upper-adj-cod} is equivalent to
\[
    \widehat N((q, w); {\graph \somesetmapfindim})=\polar{T((q, w); {\graph \somesetmapfindim})}.
\]
In particular -- simply through the definitions of polarity and convexity, -- the convex hull of the tangent cone satisfies
\[
    \widehat N((q, w); {\graph \somesetmapfindim})=\polar{[\conv T((q, w); {\graph \somesetmapfindim})]}.
\]
Defining $\widetilde{D \somesetmapfindim}(q|w)$ via
\[
    \graph \widetilde{D \somesetmapfindim}(q|w)
    =\conv \graph[D \somesetmapfindim(q|w)],
\]
we therefore deduce
\begin{equation}
    \label{eq:upper-adj-cod-tilde}
    \frechetCod \somesetmapfindim(q|w) = [D\somesetmapfindim(q|w)]^{*+} = [\widetilde{D\somesetmapfindim}(q|w)]^{*+},
\end{equation}
where the first equality holds due to the finite-dimensional setting, while the second equality holds generally due to the properties of convex hulls and polars.

The following central results shows that for the pointwise functionals that are the focus of this work, both equalities hold even in the infinite-dimensional setting.
\begin{theorem}
    \label{thm:p-graph-deriv}
    Let $P: Q \sim L^2(\Omega; \R^m) \setto W \sim L^2(\Omega; \R^k)$ have the form
    \[
        P(q) = \{ w \in L^2(\Omega; \R^m) \mid w(x) \in p(x, q(x))\, \text{for a.\,e. }  x \in \Omega\}.
    \]
    with proto-differentiable $p(x, \freevar): \R^m \setto \R^k$ (a.\,e.~$x \in \Omega$) 
    having locally closed graph at $(q(x),w(x))$ for a.\,e.~$x \in \Omega$.
    Then
    \[
        \frechetCod P(q|w) = [DP(q|w)]^{*+} = [\widetilde{DP}(q|w)]^{*+},
    \]
    i.e.,
    \[
        \begin{split}
        \dir{q} \in \frechetCod P(q|w)(\dir{w})
        &
        \iff
        \iprod{\dir{q}}{\dir{\alt{q}}} \le \iprod{\dir{w}}{\dir{\alt{w}}}
            \text{ when } \dir{\alt{w}} \in DP(q|w)(\dir{\alt{q}})
        \\
        &
        \iff
        \iprod{\dir{q}}{\dir{\alt{q}}} \le \iprod{\dir{w}}{\dir{\alt{w}}}
            \text{ when } \dir{\alt{w}} \in \widetilde{DP}(q|w)(\dir{\alt{q}}).
        \end{split}
    \]
  \end{theorem}

\begin{proof}
    Let us define $p_x \defeq p(x, \freevar)$.
    From \cref{corollary:p-form}, we have the equivalence
    \[
        \dir{{w}}(x) \in Dp_x(q(x)|w(x))(\dir{{q}}(x))
        \text{ for a.\,e. } x \in \Omega
        \iff
        \dir{{w}} \in DP(q|w)(\dir{{q}}).
    \]
    Provided that $\graph p_x$ is locally closed for a.\,e.~$x \in \Omega$, we may thus calculate
    \begin{equation}
        \label{eq:somesetmap-upper-adjoint-transformation-first-step}
        \begin{aligned}[t]
        \dir{q} \in \frechetCod P(q|w)(\dir{w})
        &
        \iff
        \dir{q}(x) \in \frechetCod p_x(q(x)|w(x))(\dir{w}(x))
        \quad\text{for a.\,e. } x \in \Omega
        \\
        &
        \iff
        \iprod{\dir{q}(x)}{\dir{\alt{q}}(x)} \le \iprod{\dir{w}(x)}{\dir{\alt{w}}(x)}
        \quad\text{for a.\,e. } x \in \Omega
        \\
        &
        \phantom{\iff}
        \quad
        \text{when } \dir{\alt{w}}(x) \in Dp_x(q(x)|w(x))(\dir{\alt{q}}(x))
        \quad\!\text{for a.\,e. } x \in \Omega
        \\
        &
        \iff
        \iprod{\dir{q}(x)}{\dir{\alt{q}}(x)} \le \iprod{\dir{w}(x)}{\dir{\alt{w}}(x)}
        \quad\text{for a.\,e. } x \in \Omega
        \\
        &
        \phantom{\iff}
        \quad
        \text{when } \dir{\alt{w}} \in DP(q|w)(\dir{\alt{q}}).
        \end{aligned}
    \end{equation}
   Here we are still computing the upper adjoint pointwise. Clearly \eqref{eq:somesetmap-upper-adjoint-transformation-first-step} however implies
    \begin{equation}
        \label{eq:somesetmap-upper-adjoint-transformation-first-impl}
        \dir{q} \in \frechetCod P(q|w)(\dir{w})
        \implies
        \iprod{\dir{q}}{\dir{\alt{q}}} \le \iprod{\dir{w}}{\dir{\alt{w}}}
        \text{ when } \dir{\alt{w}} \in DP(q|w)(\dir{\alt{q}}).
    \end{equation}
    Further, if there exists a set $E \subset \Omega$ with $\L^m(\Omega \setminus E)>0$
    and
    \[
        \iprod{\dir{q}(x)}{\dir{\alt{q}}(x)} > \iprod{\dir{w}(x)}{\dir{\alt{w}}(x)}
        \qquad (x \in E),
    \]
    then by constructing
    \[
        \altalt{\dir{q}}(x) \defeq (1+t\chi_E(x))\dir{q}(x),
        \quad
        \altalt{\dir{w}}(x) \defeq (1+t\chi_E(x))\dir{w}(x),
    \]
    we observe for sufficient large $t$ the condition
    \[
        \iprod{\dir{q}}{\altalt{\dir{q}}} > \iprod{\dir{w}}{\altalt{\dir{w}}}.
    \]
    Moreover, by the pointwise character of $P$, also
    $\altalt{\dir{w}} \in DP(q|w)(\altalt{\dir{q}})$.
    Thus the implication in \eqref{eq:somesetmap-upper-adjoint-transformation-first-impl} actually holds both ways, which is exactly what we set out to prove.

    Finally, $[DP(q|w)]^{*+} = [\widetilde{DP}(q|w)]^{*+}$ always holds, as we have already observed in \eqref{eq:upper-adj-cod-tilde}.
\end{proof}

Similarly to \cref{cor:cod-p-plus-smooth}, we have the following immediate corollary.
\begin{corollary}
    \label{cor:d-p-plus-smooth}
    Let $P$ be a pointwise set-valued functional as in \cref{thm:p-graph-deriv}, and let $h: Q \to W$ be single-valued and Fréchet differentiable. Then 
    \begin{equation}
        \label{eq:p-plus-h-upper-adjoint}
        \frechetCod (P+h)(q|w) = [D(P+h)(q|w)]^{*+}.
    \end{equation}
\end{corollary}
\begin{proof}
    Let us set $\somesetmap \defeq P+h$, and recall from \cref{cor:cod-p-plus-smooth} that
    \[
        \dir{q} \in \frechetCod \somesetmap(q|w)(\dir{w}) = \frechetCod P(q|w-h(q))(\dir{w})+[\grad h(q)]^*\dir{w}.
    \]
    This is the same to say as that  $\dir{\bar q} \defeq \dir{q} -[\grad h(q)]^*\dir{w}\in \frechetCod P(q|\alt{w})(\dir{w})$ for $\alt{w} \defeq w-h(q)$.
    By \cref{thm:p-graph-deriv} this holds if and only if
    \[
        \iprod{\dir{\bar q}}{\dir{\alt{q}}} \le \iprod{\dir{w}}{\dir{\alt{w}}}
        \text{ when } \dir{\alt{w}} \in DP(q|\alt{w})(\dir{\alt{q}}),
    \]
    or equivalently
    \[
        \iprod{\dir{q}}{\dir{\alt{q}}} \le \iprod{\dir{w}}{\dir{\alt{w}}+\grad h(q)\dir{\alt{q}}}
        \text{ when } \dir{\alt{w}} \in DP(q|\alt{w})(\dir{\alt{q}}),
    \]
    which is just the same as
    \[
        \iprod{\dir{q}}{\dir{\alt{q}}} \le \iprod{\dir{w}}{\dir{\alt{w}}}
        \text{ when } \dir{\alt{w}} \in DP(q|\alt{w})(\dir{\alt{q}})+\grad h(q)\dir{\alt{q}}.
    \]
    Now we just use \eqref{eq:d-p-plus-smooth} to derive \eqref{eq:p-plus-h-upper-adjoint}.
\end{proof}

\begin{corollary}
    \label{corollary:g-form-graphderiv}
    Let $G: L^2(\Omega; \R^m) \to \R$ have the form \eqref{eq:g-pointwise-integral} for some regular integrand $g$.
    Then the graphical derivative of $\subdiff G$ at $u$ for $\coPrimal$ in the direction $\dir{u}$, where $u,\coPrimal,\dir{u} \in L^2(\Omega; \R^m)$, is given by
    \begin{equation}
        \label{eq:g-d-form}
        D{[\subdiff G]}(u|\coPrimal)(\dir{u})
        =
        \left\{
            \dir\coPrimal \in L^2(\Omega; \R^m)
            \,\middle|\,
            \begin{array}{r}
            \dir\coPrimal(x) \in D{[\subdiff g(x, \freevar)]}(u(x)|\coPrimal(x))(\dir{u}(x)) \\
            \text{ for a.\,e. } x \in \Omega
            \end{array}
        \right\}.
    \end{equation}
    Moreover,
    \begin{equation}
        \frechetCod{} [\partial G](u|\coPrimal) = [D[\partial G](u|\coPrimal)]^{*+}.
    \end{equation}
\end{corollary}

\begin{proof}
    The claim follows directly from \cref{thm:p-graph-deriv} with $p(x, z)=\subdiff g(x, z)$. Local closedness of $\graph \subdiff g(x, \freevar)$ is a consequence of the lower semicontinuity of $g$.
\end{proof}

\subsection{Examples}\label{sec:pointwise:examples}
We now study specific cases of the finite- and infinite-dimensional second-order generalized derivatives, relevant to our model problems \eqref{eq:l1fit_problem} and \eqref{eq:linffit_problem}. Other examples satisfying the assumptions are the piecewise linear-quadratic ``multi-bang'' and switching penalties introduced in \cite{ClasonKunisch:2013} and \cite{ClasonKunisch:2016}, respectively.

\subsubsection{Squared \texorpdfstring{$\scriptstyle L^2(\Omega;\R^m)$}{L²(Ω;ℝᵐ)} norm}
\label{sec:l2-fitting-analysis}

The following result is standard; see, e.g., \cite[Ex.~8.34]{Rockafellar:1998}.
\begin{lemma}
    \label{lemma:two-norm-squared-deriv}
    With $z \in \R^m$, let $g(\FdVar)=\frac12\norm{\FdVar}^2$. Then $\subdiff g$ is proto-differentiable with
    \begin{equation}
        D(\subdiff g)(\FdVar|\coFdVar)(\dir\FdVar) =
        \begin{cases}
            \dir\FdVar, & \coFdVar=\FdVar, \\
            \emptyset, & \text{otherwise}.
        \end{cases}
    \end{equation}
\end{lemma}
From  \cref{corollary:g-form-graphderiv}, we immediately obtain
\begin{corollary}\label{cor:g-form-l2norm}
    With $g(\FdVar)=\frac12\norm{\FdVar}^2$ and $\Omega \subset \R^n$ an open bounded domain, let
    \[
        G(\Primal) \defeq \int_\Omega g(\Primal(x)) \,d x
        \qquad (\Primal \in L^2(\Omega; \R^m)).
    \]
    Then
    \[
        D(\subdiff G)(\Primal|\coPrimal)(\dir\Primal) = \dir\Primal
        \quad\text{and}\quad
        \frechetCod(\subdiff G)(\Primal|\coPrimal)(\dir\coPrimal) = \dir\coPrimal.
    \]
\end{corollary}

\subsubsection{Indicator function}
\label{sec:l1-fitting-analysis}

The following lemma is useful for computing $D[\subdiff F^*](\Dual|\coDual)$ for the problem \eqref{eq:l1fit_problem}. Its claim in the one-dimensional case ($m=1$) is illustrated in \cref{fig:indicator}.

\begin{lemma}
    \label{lemma:DsubdiffF}
    With $\FdVar \in \R^m$, let $f(\FdVar)=\ind_{\B(0,\alpha)}(\FdVar)$. Then $\subdiff f$ is proto-differentiable with
    \begin{equation}
        \label{eq:d-subdiff-indicator}
        D(\subdiff f)(\FdVar|\coFdVar)(\dir\FdVar) = 
        \begin{cases}
            \norm{\coFdVar} \dir\FdVar / \alpha + \R \FdVar, 
             & \norm{\FdVar}=\alpha,\, \coFdVar \in (0, \infty) \FdVar,\, \iprod{\FdVar}{\dir\FdVar}=0, \\
             [0, \infty) \FdVar, 
             & \norm{\FdVar}=\alpha,\, \norm{\coFdVar}=0,\, \iprod{\FdVar}{\dir\FdVar} =0, \\
             \{0\}, 
             & \norm{\FdVar}=\alpha,\, \norm{\coFdVar}=0,\, \iprod{\FdVar}{\dir\FdVar} < 0,  \\
             \{0\}, & \norm{\FdVar} < \alpha,\, \norm{\coFdVar}=0, \\
            \emptyset, & \text{otherwise}.
        \end{cases}
    \end{equation}
    In particular, if $m=1$, then
    \begin{align}
        \label{eq:d-subdiff-indicator-1d}
        D(\subdiff f)(\FdVar|\coFdVar)(\dir\FdVar) &= 
        \begin{cases}
            \R, 
             & \abs{\FdVar}=\alpha,\, \coFdVar \in (0, \infty) \FdVar,\, \dir\FdVar=0, \\
            [0, \infty) \FdVar,
             & \abs{\FdVar}=\alpha,\, \coFdVar=0,\, \dir\FdVar = 0,  \\
            \{0\}, 
             & \abs{\FdVar}=\alpha,\, \coFdVar = 0,\, \FdVar \dir\FdVar < 0, \\
            \{0\}, & \abs{\FdVar} < \alpha,\, \coFdVar=0,\, \\
            \emptyset, & \text{otherwise},
        \end{cases}
        \intertext{as well as}
        \label{eq:d-subdiff-indicator-1d-conv}
        \widetilde{D(\subdiff f)}(\FdVar|\coFdVar)(\dir\FdVar) &= 
        \begin{cases}
            \R, 
             & \abs{\FdVar}=\alpha,\, \coFdVar \in (0, \infty) \FdVar,\, \dir\FdVar=0, \\
            [0, \infty) \FdVar,
             & \abs{\FdVar}=\alpha,\, \coFdVar=0,\,  \FdVar \dir\FdVar \le 0, \\
            \{0\}, & \abs{\FdVar} < \alpha,\, \coFdVar=0,\, \\
            \emptyset, & \text{otherwise}.
        \end{cases}
    \end{align}
\end{lemma}

\begin{figure}
    \centering
    \begin{subfigure}{0.30\textwidth}
        \centering
        \begin{asy}
            unitsize(50,50);
            real l=1.5;
            real eps=0.3;
            pair gfstart=(-1, -l);
            pair gfend=(1, l);
            path gf=gfstart--(-1, 0)--(1, 0)--gfend;
            draw(gf, dashed, Arrows);

            draw((-eps, 0)--(eps, 0), linewidth(1.1), Arrows);
            dot((0, 0));
            label("(iv)", (0, 0), 1.5*N);

            draw((-1,-eps)--(-1,0)--(-1+eps, 0), linewidth(1.1), Arrows);
            dot((-1, 0));
            label("(ii)", (-1, -eps), 1.5*W);
            label("(iii)", (-1+eps, 0), 1.5*N);

            draw((1,0.75+eps)--(1, 0.75-eps), linewidth(1.1), Arrows);
            dot((1, 0.75));
            label("(i)", (1, 0.75), 1.5*E);
        \end{asy}
        \caption{$D(\subdiff f)(z|\zeta)$}
    \end{subfigure}
    \hfill
    \begin{subfigure}{0.30\textwidth}
        \centering
        \begin{asy}
            unitsize(50,50);
            real l=1.5;
            real eps=0.3;
            pair gfstart=(-1, -l);
            pair gfend=(1, l);
            path gf=gfstart--(-1, 0)--(1, 0)--gfend;
            draw(gf, dashed, Arrows);

            draw((-eps, 0)--(eps, 0), linewidth(1.1), Arrows);
            dot((0, 0));

            fill((-1,-eps)--(-1,0)--(-1+eps, 0)..controls (-1+eps/sqrt(2), -eps/sqrt(2))..cycle, gray(.8));
            draw((-1,-eps)--(-1,0)--(-1+eps, 0), linewidth(1.1), Arrows);
            dot((-1, 0));

            draw((1,0.75+eps)--(1, 0.75-eps), linewidth(1.1), Arrows);
            dot((1, 0.75));
        \end{asy}
        \caption{$\widetilde{D(\subdiff f)}(z|\zeta)$}
    \end{subfigure}
    \hfill
    \begin{subfigure}{0.30\textwidth}
        \centering
        \begin{asy}
            unitsize(50,50);
            real l=1.5;
            real eps=0.3;
            pair gfstart=(-1, -l);
            pair gfend=(1, l);
            path gf=gfstart--(-1, 0)--(1, 0)--gfend;
            draw(gf, dashed, Arrows);

            draw((0, -eps)--(0, eps), linewidth(1.1), Arrows);
            dot((0, 0));

            fill((-1-eps,0)--(-1,0)--(-1, eps)..controls (-1-eps/sqrt(2), eps/sqrt(2))..cycle, gray(.8));
            draw((-1-eps,0)--(-1,0)--(-1, eps), linewidth(1.1), Arrows);
            dot((-1, 0));

            draw((1-eps,0.75)--(1+eps, 0.75), linewidth(1.1), Arrows);
            dot((1, 0.75));
        \end{asy}
        \caption{$\widehat D^*(\subdiff f)(z|\zeta)$}
    \end{subfigure}
    \caption{Illustration of the graphical derivative and regular coderivative for $\subdiff f$ with $f=\ind_{[-1,1]}$. The dashed line is $\graph \subdiff f$. The dots indicate the base points $(z, \zeta)$ where the graphical derivative or coderivative is calculated, and the thick arrows and gray areas the directions of $(\dir z, \dir \zeta)$ relative to the base point.
    The labels (i) etc. indicate the corresponding case of \eqref{eq:d-subdiff-indicator-1d}.
    }
    \label{fig:indicator}
\end{figure}

\begin{proof}
    The proto-differentiability of $\subdiff f$ follows from the fact that $f$ is twice epi-differentiable; see \cite[Ex.~13.17 \& Thm.~13.40]{Rockafellar:1998}, writing $B{(0,\alpha)}=\{x\in\R^n\mid \norm{x}^2 \in (-\infty, \alpha]\}$ for the twice continuously differentiable mapping $x\mapsto \norm{x}^2$ and the polyhedral set $(-\infty,\alpha]$ satisfying the contraint qualification.

    For the full proof of \eqref{eq:d-subdiff-indicator}, using second-order subgradient theory from \cite{Rockafellar:1998}, we refer to \cite{Valkonen:2014}.%
        \footnote{There is a small omission in \cite[Lemma 4.2]{Valkonen:2014}, that actually causes $\widetilde{D(\subdiff f)}(\FdVar|\coFdVar)(\dir\FdVar)$ instead to be calculated. In calculating the subdifferential of (4.18) therein, at the end of the proof of the lemma, the cases $\iprod{y}{w}=0$ and $\iprod{y}{w} < 0$ need to be calculated separately to give the two different sub-cases of ($\norm{z}=\alpha$ and $\norm{\zeta}=0$) in our expression \eqref{eq:d-subdiff-indicator}.}
    For completeness, we provide here an elementary proof of the one-dimensional case~\eqref{eq:d-subdiff-indicator-1d}.
    We have
    \begin{equation}
        \label{eq:indicator-1d-subdiff}
        \subdiff f(\FdVar)=
        \begin{cases}
            [0, \infty)\FdVar, & \abs{\FdVar}=\alpha, \\
            \{0\}, & \abs{\FdVar} < \alpha, \\
            \emptyset, & \text{otherwise}.
        \end{cases}
    \end{equation}
    If $\coFdVar \in \subdiff f(\FdVar)$ and $\dir{\coFdVar} \in  D(\subdiff f)(\FdVar|\coFdVar)(\dir\FdVar)$, there exists by \eqref{eq:graphderiv} sequences $t^i \downto 0$, $\dir{\FdVar^i} \to \dir{\FdVar}$, and $\coFdVar^i \in \subdiff f(\FdVar+t^i\dir{\FdVar^i})$ such that
    \begin{equation}
        \label{eq:indicator-1d-limit}
        \dir{\coFdVar}=\lim_{i \to \infty} \frac1{t^i}(\coFdVar^i-\coFdVar).
    \end{equation}
    We proceed by case distinction.
    \begin{enumerate}[label=(\roman*)]
        \item If $\abs{\FdVar}=\alpha$, $\dir{\FdVar}=0$, and  $\coFdVar \in (0, \infty)\FdVar$, choosing $z^i=0$, we can for any $\dir{\coFdVar} \in \R$ and large enough $i$ take $\coFdVar^i=t^i\dir{\coFdVar}+\coFdVar \in [0, \infty)\FdVar = \subdiff f(\FdVar)$. Thus we obtain the first case in \eqref{eq:d-subdiff-indicator-1d}.

        \item Let us then suppose $\abs{\FdVar}=\alpha$, $\dir{\FdVar}=0$, but $\coFdVar=0$.
            In this case, choosing $z^i=0$, we have by \eqref{eq:indicator-1d-subdiff} free choice of $\coFdVar^i \in [0, \infty) \FdVar$. Picking $\dir{\coFdVar} \in [0, \infty) \FdVar$ and setting $\coFdVar^i \defeq t^i\dir{\coFdVar}$, we deduce that $\coFdVar^i \in [0, \infty) \FdVar \subset D(\subdiff f)(\FdVar|\coFdVar)(\dir{\FdVar})$. Since $-(0, \infty) \FdVar$ is clearly not approximable from $[0, \infty) \FdVar$,  we obtain the second case of \eqref{eq:d-subdiff-indicator-1d}.

        \item If $\abs{\FdVar}=\alpha$ and $\FdVar\dir{\FdVar}>0$, then $\subdiff f(\FdVar+t^i\dir{\FdVar^i})=\emptyset$ for large $i$. Therefore it must hold that $\FdVar\dir{\FdVar} \le 0$.
            If $\dir{\FdVar} \ne 0$, it follows that $\coFdVar^i=0$ (for large $i$).
            Since $\coFdVar$ is fixed, the limit \eqref{eq:indicator-1d-limit} does not exist unless $\coFdVar=0$, in which case also $\dir{\coFdVar}=0$.
            This is covered by the third case of~\eqref{eq:d-subdiff-indicator-1d}.

        \item If $\abs{\FdVar}<\alpha$, then $\coFdVar^i=\coFdVar=0$, so we get the fourth case in \eqref{eq:d-subdiff-indicator-1d}.

        \item If $\abs{\FdVar}=\alpha$ and $\dir{\FdVar}=0$, but $\coFdVar \in -(0, \infty)\FdVar$, we see that
            \[
                \sign{\coFdVar}\,\frac1{t^i}(\coFdVar^i-\coFdVar)>\frac{1}{t^i} \abs{\coFdVar},
            \]
            so the limit \eqref{eq:indicator-1d-limit} cannot exist. Therefore the coderivative is empty.

         Likewise, we obtain the empty coderivative if $\abs{\FdVar}>\alpha$, since even $\subdiff f(\FdVar)$ is empty and $\coFdVar$ does not exist. Together, we obtain the final case in \eqref{eq:d-subdiff-indicator-1d}.
    \end{enumerate}

    Finally, regarding $\widetilde{D(\subdiff f)}(\FdVar|\coFdVar)$ with $m=1$, we see that only the case $\abs{\FdVar}=\alpha$ and $\coFdVar=0$ is split into two sub-cases in \eqref{eq:d-subdiff-indicator-1d}, yielding an altogether non-convex $\graph[D(\subdiff f)(\FdVar|\coFdVar)]$. Taking the convexification of this set yields \eqref{eq:d-subdiff-indicator-1d-conv}; cf.~\cref{fig:indicator}.
\end{proof}

\begin{corollary}\label{cor:g-form-indicator}
    Let $f^*(\FdVar) \defeq \ind_{[-\alpha, \alpha]}(\FdVar)$ and
    \[
        F^*(\Dual) \defeq \int_\Omega f^*(\Dual(x)) \,d x
        \qquad (\Dual \in L^2(\Omega)).
    \]
    Then
    \begin{align}
        \widetilde{D[\subdiff F^*]}(\Dual|\coDual)(\dir\Dual)&=
        \begin{cases}
            \DerivConeF[\Dual|\coDual]^\circ, & \dir\Dual \in \DerivConeF[\Dual|\coDual] \text{ and } \coDual \in \subdiff F^*(\Dual),\\
            \emptyset, & \text{otherwise},
        \end{cases}
        \intertext{and}
        \frechetCod{[\subdiff F^*]}(\Dual|\coDual)(\dir\coDual)&=
        \begin{cases}
            {\DerivConeF[\Dual|\coDual]}^\circ, & -\dir\coDual \in \DerivConeF[\Dual|\coDual] \text{ and } \coDual \in \subdiff F^*(\Dual),\\
            \emptyset, & \text{otherwise},
        \end{cases}
    \end{align}
    for the cone
    \begin{align}
            \DerivConeF[\Dual|\coDual] &=
    \{
        \dcVar \in L^2(\Omega) 
        \mid
        \dcVar(x)\Dual(x)\le 0 \text{ if } \abs{\Dual(x)}=\alpha \ \text{ and }\ 
        \dcVar(x)\coDual(x)\ge 0
    \}
    \intertext{and its polar}
     \polar{\DerivConeF[\Dual|\coDual]}&=
    \{
        \nu \in L^2(\Omega) 
        \mid
        \nu(x)\Dual(x)\ge 0 \text{ if } \coDual(x)=0\ \text{ and }\
        \nu(x) = 0 \text{ if } \abs{\Dual(x)}<\alpha
    \}.
\end{align}
\end{corollary}

\begin{proof}
    The claim about the graphical derivative follows from \cref{corollary:g-form-graphderiv} and \cref{lemma:DsubdiffF}, using the fact that the indicator function of a closed convex set is normal. 
    The regular coderivative formula follows from the more general \cref{prop:frechetcod-cone-operator} in the appendix. Here, in the derivation of the explicit form of the polar cone $\polar{\DerivConeF[\Dual|\coDual]}$, we use the fact that $D[\subdiff F^*](\Dual|\coDual)(\dir\Dual)$ is non-empty if and only if 
    \begin{equation}
        \label{eq:D-fstar-l1-constr}
        \coDual(x)\Dual(x)=\abs{\coDual(x)}
        \quad\text{and}\quad \abs{\Dual(x)} \le \alpha
        \qquad (x \in \Omega).
        \qedhere
    \end{equation}
\end{proof}
\begin{remark}
    If $(\Dual,\coDual)$ satisfy the strict complementarity condition  $\abs{\Dual(x)}<\alpha$ or $\abs{\coDual(x)}>0$  for a.\,e.~$x \in \Omega$, the degenerate second and third case in \eqref{eq:d-subdiff-indicator-1d} (corresponding to the gray areas in \cref{fig:indicator}) do not occur, and the cone simplifies to
     \[
         \DerivConeF[\Dual|\coDual] \defeq
         \{ \dcVar \in L^2(\Omega) \mid \dcVar(x)=0 ~\text{if}~ \abs{\Dual(x)}=\alpha,\, x \in \Omega \}.
     \]
    Note that points $x\in\Omega$ where a degenerate case occurs are precisely those where there is no \term{graphical regularity} of $\partial f$ at $(\Dual(x), \coDual(x))$. We refer to \cite[Thm.~8.40]{Rockafellar:1998} for the definition of this concept, which we do not require in the present work.
\end{remark}

\subsubsection{\texorpdfstring{$\scriptstyle L^1(\Omega;\R^m)$}{L¹(Ω;ℝᵐ)} norm}

The following lemma is useful for computing $D[\subdiff F^*]$ for the problem \eqref{eq:linffit_problem}. Its claim in the one-dimensional case ($m=1$) is illustrated in \cref{fig:1norm-1d}.
\begin{lemma}
    \label{lemma:one-norm-findim}
    With $\FdVar \in \R^m$, let $f^*(\FdVar)=\norm{\FdVar}_2$. Then $\subdiff f^*$ is proto-differentiable with
    \begin{equation}
        \label{eq:d-subdiff-1norm}
        D(\subdiff f^*)(\FdVar|\coFdVar)(\dir\FdVar) = 
        \begin{cases}
            \left(\frac{I-(\FdVar \otimes \FdVar) / \norm{\FdVar}^2}{\norm{\FdVar}^2}\right) \dir\FdVar,
                & \FdVar \ne 0,\, \coFdVar\norm{\FdVar}=\FdVar, \\
                \{\coFdVar\}^\perp, 
                & \FdVar=0,\, \dir\FdVar \ne 0,\, \coFdVar\norm{\dir\FdVar}=\dir\FdVar, \\
                \polar{\{\coFdVar\}}, 
                & \FdVar=0,\, \dir\FdVar = 0,\, \norm{\coFdVar} = 1, \\
                \R^m,
                & \FdVar=0,\, \dir\FdVar = 0,\, \norm{\coFdVar} < 1, \\
            \emptyset,
                & \text{otherwise.}
        \end{cases}
    \end{equation}
    In particular, if $m=1$, then
    \begin{align}
        \label{eq:d-subdiff-1norm-1d}
        D(\subdiff f^*)(\FdVar|\coFdVar)(\dir\FdVar) &= 
        \begin{cases}
            \{0\}
                & \FdVar \ne 0,\, \coFdVar = \sign \FdVar, \\
            \{0\}, 
                & \FdVar=0,\, 
                \dir\FdVar \in (0, \infty)\zeta,
                \\
            (-\infty,0]\coFdVar, 
                & \FdVar=0,\, \dir\FdVar = 0,\, \abs{\coFdVar} = 1, \\
            \R,
                & \FdVar=0,\, \dir\FdVar = 0,\, \abs{\coFdVar} < 1, \\
            \emptyset,
                & \text{otherwise,}
        \end{cases}
    \intertext{as well as}
        \label{eq:d-subdiff-1norm-1d-conv}
        \widetilde{D(\subdiff f^*)}(\FdVar|\coFdVar)(\dir\FdVar) &= 
        \begin{cases}
            \{0\}
                & \FdVar \ne 0,\, \coFdVar = \sign \FdVar, \\
            (-\infty,0]\coFdVar, 
                & \FdVar=0,\, 
                \dir\FdVar \in [0, \infty)\zeta,\,
                \abs{\coFdVar} = 1, \\
            \R,
                & \FdVar=0,\, \dir\FdVar = 0,\, \abs{\coFdVar} < 1, \\
            \emptyset,
                & \text{otherwise.}
        \end{cases}
    \end{align}
\end{lemma}

\begin{figure}
    \centering
    \begin{subfigure}{0.30\textwidth}
        \centering
        \begin{asy}
            unitsize(50,50);
            real l=1.25;
            real eps=0.3;
            pair gfstart=(-l, -1);
            pair gfend=(l, 1);
            path gf=gfstart--(0, -1)--(0, 1)--gfend;
            draw(gf, dashed, Arrows);

            draw((0, -eps)--(0, eps), linewidth(1.1), Arrows);
            dot((0, 0));
            label("(iv)", (0,0), 1.5*W);

            draw((-eps,-1)--(0,-1)--(0, -1+eps), linewidth(1.1), Arrows);
            dot((0, -1));
            label("(iii)", (0,-1+eps), 1.5*E);
            label("(ii)", (-eps,-1), 1.5*S);

            draw((0.5*l-eps, 1)--(0.5*l+eps, 1), linewidth(1.1), Arrows);
            dot((0.5*l, 1));
            label("(i)", (0.5*l,1), 1.5*N);
        \end{asy}
        \caption{$D(\subdiff f^*)(z|\zeta)$}
    \end{subfigure}
    \hfill
    \begin{subfigure}{0.30\textwidth}
        \centering
        \begin{asy}
            unitsize(50,50);
            real l=1.25;
            real eps=0.3;
            pair gfstart=(-l, -1);
            pair gfend=(l, 1);
            path gf=gfstart--(0, -1)--(0, 1)--gfend;
            draw(gf, dashed, Arrows);

            draw((0, -eps)--(0, eps), linewidth(1.1), Arrows);
            dot((0, 0));

            fill((-eps, -1)--(0,-1)--(0, -1+eps)..controls (-eps/sqrt(2), -1+eps/sqrt(2))..cycle, gray(.8));
            draw((-eps, -1)--(0,-1)--(0, -1+eps), linewidth(1.1), Arrows);
            dot((0, -1));

            draw((0.5*l-eps, 1)--(0.5*l+eps, 1), linewidth(1.1), Arrows);
            dot((0.5*l, 1));

            // For alignment
            label("$\phantom{(ii)}$", (-eps,-1), 1.5*S);
            label("$\phantom{(i)}$", (0.5*l,1), 1.5*N);
        \end{asy}
        \caption{$\widetilde{D(\subdiff f^*)}(z|\zeta)$}
    \end{subfigure}
    \hfill
    \begin{subfigure}{0.30\textwidth}
        \centering
        \begin{asy}
            unitsize(50,50);
            real l=1.25;
            real eps=0.3;
            pair gfstart=(-l, -1);
            pair gfend=(l, 1);
            path gf=gfstart--(0, -1)--(0, 1)--gfend;
            draw(gf, dashed, Arrows);

            draw((-eps,0)--(eps, 0), linewidth(1.1), Arrows);
            dot((0, 0));

            fill((0, -1-eps)--(0,-1)--(eps, -1)..controls (eps/sqrt(2), -1-eps/sqrt(2))..cycle, gray(.8));
            draw((0, -1-eps)--(0,-1)--(eps, -1), linewidth(1.1), Arrows);
            dot((0, -1));

            draw((0.5*l, 1+eps)--(0.5*l, 1-eps), linewidth(1.1), Arrows);
            dot((0.5*l, 1));
        \end{asy}
        \caption{$\widehat D^*(\subdiff f^*)(z|\zeta)$}
    \end{subfigure}
    \caption{Illustration of the graphical derivative and regular coderivative for $\subdiff f^*$ with $f^*=\abs{\freevar}$. The dashed line is $\graph \subdiff f$. The dots indicate the base points $(z, \zeta)$ where the graphical derivative or coderivative is calculated, and the thick arrows and gray areas the directions of $(\dir z, \dir \zeta)$ relative to the base point.
    The labels (i) etc. indicate the corresponding case of \eqref{eq:d-subdiff-1norm-1d}.
    }
    \label{fig:1norm-1d}
\end{figure}

\begin{proof}
    In the case $m=1$, the proto-differentiability of $\subdiff f^*$ follows from the fact that $f^*$ is piecewise linear and hence twice epi-differentiable; see \cite[Prop.~13.9 \& Thm.~13.40]{Rockafellar:1998}. For general $m \in \N$, we may use the twice epi-differentiability of $f(x) = \iota_{\B(0, 1)}(x)$ established in the proof of \cref{lemma:DsubdiffF} and the conjugate relationship in \cite[Thm.~13.21]{Rockafellar:1998} together with \cite[Thm.~13.40]{Rockafellar:1998}.

    It remains to verify the expressions \eqref{eq:d-subdiff-1norm}--\eqref{eq:d-subdiff-1norm-1d-conv}.
    We have for any $m\in \N$ that
    \[ 
        \subdiff f^*(\FdVar)=
        \begin{cases}
            \left\{\frac{\FdVar}{\norm{\FdVar}}\right\}, & \FdVar \ne 0 \\
            \closure B(0, 1), & \FdVar = 0.
        \end{cases}
    \]
    We again proceed by case distinction.
    \begin{enumerate}[label=(\roman*)]
        \item For $\FdVar \ne 0$ necessarily therefore $D(\subdiff f^*)(\FdVar|\coFdVar)=\emptyset$ unless $\coFdVar=\FdVar/\norm{\FdVar}$, which yields the last case. 

        \item If $\FdVar \ne 0$ and $\coFdVar=\FdVar/\norm{\FdVar}$, for any
    \[
                \alt{\FdVar}=\FdVar+t \dir{\alt{\FdVar}}/\norm{\dir{\alt{\FdVar}}}
    \]
            with $\alt{\FdVar} \to z$ and $t \downto 0$, we have also $\subdiff f^*(\alt{\FdVar})=\alt{\FdVar}/\norm{\alt{\FdVar}}$. The first case in \eqref{eq:d-subdiff-1norm} now follows immediately from computing the outer limit
    \begin{equation}
        \limsup_{t \downto 0, \dir{\alt{\FdVar}} \to \dir\FdVar} \frac{\subdiff f^*(\alt{\FdVar})-\coFdVar}{t}
        =\limsup_{t \downto 0, \dir{\alt{\FdVar}} \to \dir\FdVar} \frac{\alt{\FdVar}/\norm{\alt{\FdVar}}-\coFdVar}{t}
        =\grad\left(\frac{\FdVar}{\norm{\FdVar}}\right) \dir\FdVar.
    \end{equation}

        \item If $\FdVar=0$, and $\dir\FdVar \ne 0$, then $\alt{\FdVar} \ne 0$ and $\alt{\FdVar}/\norm{\alt{\FdVar}}=\dir{\alt{\FdVar}}/\norm{\dir{\alt{\FdVar}}}$. Therefore
            \begin{equation}
                \limsup_{t \downto 0, \dir{\alt{\FdVar}} \to \dir\FdVar} \frac{\subdiff f^*(\alt{\FdVar})-\coFdVar}{t}
                =\limsup_{t \downto 0, \dir{\alt{\FdVar}} \to \dir\FdVar} \frac{\dir{\alt{\FdVar}}/\norm{\dir{\alt{\FdVar}}}-\coFdVar}{t}
            \end{equation}
            will only have limits if $\coFdVar$ lies on the boundary of $B(0, 1)$, and indeed $\coFdVar=\dir\FdVar/\norm{\dir\FdVar}$. This gives the limit $\{\coFdVar\}^\perp$, i.e., the second case.

        \item If $\FdVar=0$ and $\dir\FdVar=0$, then we may pick $\coFdVar \in \closure B(0, 1)$ arbitrarily by choosing also $\dir{\alt{\FdVar}}=0$.
            If $\norm{\coFdVar}=1$, then we obtain the limit
            \[
                \limsup_{t \to \infty} \frac1{t}(\B(0,1)-\coFdVar)=\{\dir\coFdVar \in \R^m \mid \iprod{\dir\coFdVar}{\coFdVar}\le 0\}=\polar{\{\coFdVar\}}
            \]
            and hence the third case.

        \item In the same situation, choosing $\norm{\coFdVar}<1$ gives the limit $\R^m$ and hence the fourth case.
    \end{enumerate}

    Finally, \eqref{eq:d-subdiff-1norm-1d} is a trivial specialization of \eqref{eq:d-subdiff-1norm}, while regarding $\widetilde{D(\subdiff f^*)}(\FdVar|\coFdVar)$ with $m=1$, we see that only the case $\FdVar=0$ and $\abs{\coFdVar}=1$ is split into two sub-cases in \eqref{eq:d-subdiff-1norm-1d}. These produce an altogether non-convex $\graph[D(\subdiff f^*)(\FdVar|\coFdVar)]$. Taking the convexification of this set yields \eqref{eq:d-subdiff-1norm-1d-conv}; cf.~\cref{fig:1norm-1d}.
\end{proof}

\begin{corollary}\label{cor:g-form-l1norm}
    Let $f^*(\FdVar) \defeq \noise\abs{\FdVar}$ and
    \[
        F^*(\Dual) \defeq \int_\Omega f^*(\Dual(x)) \,d x
        \qquad (\Dual \in L^2(\Omega)).
    \]
    Then
    \begin{align}
        \widetilde{D[\subdiff F^*]}(\Dual|\coDual)(\dir\Dual)&=
        \begin{cases}
            \DerivConeF[\Dual|\coDual]^\circ, & \dir\Dual \in \DerivConeF[\Dual|\coDual] \text{ and } \coDual \in \subdiff F^*(\Dual),\\
            \emptyset, & \text{otherwise},
        \end{cases}
        \intertext{and}
        \frechetCod{[\subdiff F^*]}(\Dual|\coDual)(\dir\coDual)&=
        \begin{cases}
            {\DerivConeF[\Dual|\coDual]}^\circ, & -\dir\coDual \in \DerivConeF[\Dual|\coDual] \text{ and } \coDual \in \subdiff F^*(\Dual),\\
            \emptyset, & \text{otherwise},
        \end{cases}
    \end{align}
    for the cone
\begin{align}
       \label{eq:linfty-fitting-z}
    {\DerivConeF[\Dual|\coDual]} &=
    \{
        \dcVar \in L^2(\Omega) 
        \mid
        \dcVar(x)\coDual(x) \ge 0
        \text{ if } \Dual(x)=0\ \text{ and }\ 
        (\noise-\abs{\coDual(x)})\dcVar(x)=0
    \},
    \intertext{and its polar}
    \label{eq:linfty-fitting-nu}
    \polar{{\DerivConeF[\Dual|\coDual]}} &=
    \{
        \nu \in L^2(\Omega) 
        \mid
        \nu(x)\coDual(x)\le 0 \text{ if } \abs{\coDual(x)}=\noise\ \text{ and }\ 
        \Dual(x)\nu(x)=0
    \}.
\end{align}
\end{corollary}
\begin{proof}
    The claim about the graphical derivative follows from \cref{corollary:g-form-graphderiv} and \cref{lemma:one-norm-findim}, using the fact that $g(\FdVar)=|z|$ is finite-valued and Lipschitz continuous and hence normal.
    The regular coderivative formula follows from the more general \cref{prop:frechetcod-cone-operator} in the appendix. To derive the explicit form of the polar cone $\polar{\DerivConeF[\Dual|\coDual]}$, we employ the fact that $D[\subdiff F^*](\Dual|\coDual)(\dir\Dual)$ is non-empty if and only if 
    \begin{equation}
        \label{eq:linfty-fitting-constr}
        \abs{\coDual(x)}\le \noise
        \quad\text{and}\quad
        \Dual(x)\coDual(x)=\abs{\Dual(x)}.
        \qedhere
    \end{equation}
\end{proof}
\begin{remark}
    If $(\Dual,\coDual)$ satisfy the strict complementarity condition  $\Dual(x) \ne 0$ or $\abs{\coDual(x)}<\noise$ for a.\,e.~$x \in \Omega$, the degenerate second and third case in \eqref{eq:d-subdiff-1norm-1d} (corresponding to the gray areas in \cref{fig:1norm-1d}) do not occur, and the cone simplifies to
    \[
        \DerivConeF[\Dual|\coDual] \defeq
        \{ \dcVar \in L^2(\Omega) \mid \dcVar(x)=0 ~\text{if}~ \Dual(x)=0,\, x \in \Omega \}.
    \]
    Again, points $x\in\Omega$ where a degenerate case occurs are precisely those where graphical regularity fails to hold for $\partial f^*$ at $(\Dual(x), \coDual(x))$. 
\end{remark}

\subsubsection{Spatially varying integrands}

Let $\alpha, \beta \in L^2(\Omega)$ with $\alpha(x) < \beta(x)$ for a.\,e.~$x \in \Omega$.
Define
\[
    f(x, \FdVar) \defeq \ind_{[\alpha(x), \beta(x)]}(\FdVar)
    \qquad (x \in \Omega; \FdVar \in \R).
\]
This example is useful for spatially or temporally varying ``tube'' constraints, which arise in the regularization of inverse problems subject to variable noise levels~\cite{Dong:2011}. The indicator function of temporally variable constraints also appears in Moreau's sweeping process, which is a model for several phenomena from nonsmooth mechanics such as elastoplasticity~\cite{Kunze:2000}.

Due to the measurability of $\alpha$ and $\beta$, the integrand $f$ is proper, convex and normal \cite[Ex.~14.32]{Rockafellar:1998}, such that the subdifferential $\subdiff f(x,\freevar)$ can be computed pointwise. Furthermore, $f$ is a.\,e. proto-differentiable as the indicator function of the convex polyhedral set $[\alpha(x),\beta(x)]$; see again \cite[Ex.~13.17 \& Thm.~13.40]{Rockafellar:1998}.
By simple pointwise application of \cref{lemma:DsubdiffF} we can thus compute  $D[\subdiff f(x,\freevar)]$.
We therefore deduce the applicability of \cref{corollary:g-form-graphderiv} to
\[
    F(\Dual) = \int_\Omega f(x, \Dual(x)) \,d x,
\]
and obtain a pointwise characterization of $D(\subdiff F)$ similar to \cref{cor:g-form-indicator}.

Clearly, we can analogously modify \cref{cor:g-form-l2norm} (squared $L^2(\Omega;\R^m)$ norm) and \cref{cor:g-form-l1norm} ($L^1(\Omega;\R^m)$ norm) by, e.g., introducing a spatially varying weight in each norm.

\section{Stability of variational inclusions}\label{sec:stability_vi}

To pave the way towards studying the stability of saddle point systems in the following section, we now recall general concepts for the study of variational inclusions and develop general results that quickly specialize to saddle point systems in $L^2$.

\subsection{Metric regularity and the Mordukhovich criterion}

Our stability analysis is based on the following set-valued Lipschitz property \cite{Aubin:1990,Rockafellar:1998,Mordukhovich:2006}, also known as  the \emph{Aubin property} of $\inv \somesetmap$.

\begin{definition}
    We say that the set-valued mapping $\somesetmap: Q \setto W$ is \emph{metrically regular} at $\realopt w$ for $\realopt q$ if $\graph \somesetmap$ is locally closed and there exist $\rho, \delta, \lipnum > 0$ such that
    \begin{equation}
        \label{eq:inverse-aubin}
        \inf_{p\,:\, w \in \somesetmap(p)} \norm{q-p}
            \le \lipnum \norm{w-\somesetmap(q)}
        \quad\text{ for any $q,w$ such that }
         \norm{q-\realopt{q}} \le \delta,
         \,
         \norm{w-\realopt{w}} \le \rho
         .
    \end{equation}
    We denote the infimum over valid constants $\lipnum$ by $\lip{\inv \somesetmap}(\realopt w|\realoptq)$,
    or $\lip{\inv \somesetmap}$ for short when there is no ambiguity about the point $(\realopt w, \realoptq)$.
\end{definition}

A simplified view, indicating why this concept is useful, can be seen by taking $\realoptq$ satisfying $0 \in \somesetmap(\realoptq)$. Setting $q=\realoptq$ and $\realopt{w}=0$ in \eqref{eq:inverse-aubin}, we then obtain
\begin{equation}
    \label{eq:sensitivity1}
    \inf_{p\,:\, w \in \somesetmap(p)} \norm{\realoptq-p}
        \le \lipnum_{\inv \somesetmap}(0|\realoptq) \norm{w}
    \quad\text{ for any $w$ such that }
     \norm{w} \le \rho
     .
\end{equation}
Therefore, if we perturb the variational inclusion $0 \in \somesetmap(\realoptq)$ -- typically an optimality condition -- by a small linear perturbation $w$, we will still find a nearby solution to the perturbed problem. We will later see that for our problems of interest, we can encode variations in data and in an additional Moreau-Yosida regularization parameter into $w$. We therefore need to estimate $\lipnum_{\inv \somesetmap}$, for which 
the following \emph{Mordukhovich criterion} \cite{Mordukhovich:1992} will be useful.
It is also contained in \cite[Thm.~4.7]{Mordukhovich:2006} and 
simplified here to our Hilbert space setting from the original 
Asplund space setting.
\begin{theorem}
    \label{thm:ell}
    Let $\somesetmap: Q \setto W$ be a set-valued mapping between Hilbert spaces $Q$ and $W$. 
    Suppose $\graph \somesetmap$ is locally closed around $(q, w) \in \graph \somesetmap$.
    Then
    \[
        \lip{\somesetmap}(q|w)
        =\inf_{t>0} \sup \left\{
            \norm{\frechetCod \somesetmap(\alt{q}|\alt{w})} \,\middle|\,
                \alt{q} \in \B(q, t),\, \alt{w} \in \somesetmap(\alt{q}) \isect \B(w, t)
            \right\}.
    \]
\end{theorem}
Here, for positively homogeneous $M: W \setto Q$, we have defined
\[
    \norm{M} \defeq \sup\{ \norm{q} \mid q \in M(w),\, \norm{w} \le 1\}.
\]

If $\somesetmap$ satisfies the regularity assumption $\frechetCod \somesetmap(q|w) = [D\somesetmap(q|w)]^{*+}$ (which is the case for pointwise mappings due to \cref{thm:p-graph-deriv}), we may translate \cref{thm:ell} to be expressed in terms of the graphical derivative $D \somesetmap$, where by the second equation in \eqref{eq:upper-adj-cod-tilde} it suffices to consider the convexification $\widetilde{D \somesetmap}$. This is the content of the next proposition.
\begin{proposition}
    \label{prop:lip-estim-regular}
    Let $\somesetmap: Q \setto W$ be a set-valued mapping between Hilbert spaces $Q$ and $W$. 
    Suppose $\graph \somesetmap$ is locally closed around $(q, w) \in \graph \somesetmap$ and 
    \begin{equation}
        \label{eq:somesetmap-cod-d-upper-adjoint}
        \frechetCod \somesetmap(q|w) = [D\somesetmap(q|w)]^{*+}.
    \end{equation}
    Then
    \begin{equation}
        \label{eq:lip-inv}
        \lip{\inv \somesetmap}(w|q)
        =
        \inf_{t>0} \sup \left\{
            \tildelip{\inv \somesetmap}(\alt{w}|\alt{q})
            \,\middle|\,
            \alt{w} \in \B(w, t),\, \alt{q} \in \B(q, t),\, \alt{w} \in \somesetmap(\alt{q})
        \right\},
    \end{equation}
    with
    \begin{equation}
        \label{eq:tildelip-inv-r-2}
            \tildelip{\inv\somesetmap}(\alt{w}|\alt{q})
            \defeq
            \sup\left\{
                \norm{\dir{w}}
                \,\middle|\,
                \begin{array}{l}
                    \dir{q} \in Q,\, \dir{w} \in W,\, \norm{\dir{q}} \le 1, \text{ satisfying }
                    \\
                    \iprod{\dir{q}}{\dir{\alt{q}}} \le \iprod{\dir{w}}{\dir{\alt{w}}}
                    ~\text{when }
                    \dir{\alt{w}} \in \widetilde{D\somesetmap}(\alt{q}|\alt{w})(\dir{\alt{q}})
                \end{array}
            \right\}.
    \end{equation}
\end{proposition}

\begin{proof}
    From the \cref{def:frechetcod} of $\frechetCod \somesetmap(q|w)$ and $\frechetCod \inv \somesetmap(w|q)$ through $\widehat N((w,1); \graph \somesetmap)$, we observe from the definitions that
    \[
        \dir{w} \in \frechetCod \somesetmap(w|q)(\dir{q})
        \iff
        - \dir{q} \in \frechetCod \inv \somesetmap(q|w)(-\dir{w}).
    \]
    Applied to $\inv \somesetmap$, \cref{thm:ell} therefore gives
    \begin{equation}
        \label{eq:inv-s-lip-estimate-coderivative}
        \begin{aligned}[t]
            \lip{\inv \somesetmap}(w|q)
            &
            =\inf_{t>0} \sup \left\{
                \norm{\frechetCod \inv \somesetmap(\alt{w}|\alt{q})} \,\middle|\,
                    \alt{w} \in \B(w, t),\, \alt{q} \in \inv \somesetmap(\alt{w}) \isect \B(q, t)
                \right\}
            \\
            &
            =\inf_{t>0} \sup \left\{
                \norm{\inv{[\frechetCod \somesetmap(\alt{q}|\alt{w})]}} \,\middle|\,
                    \alt{w} \in \B(w, t),\, \alt{q} \in \B(q, t),\, \alt{w} \in \somesetmap(\alt{q})
                \right\}
            \\
            &
            =\inf_{t>0} \sup \left\{
                \norm{\dir{w}}
                \,\middle|\,
                \begin{array}{l}
                \dir{q} \in \frechetCod \somesetmap(\alt{q}|\alt{w})(\dir{w}),\,
                \norm{\dir{q}} \le 1,\,
                \\
                \alt{w} \in \B(w, t),\, \alt{q} \in \B(q, t),\, \alt{w} \in \somesetmap(\alt{q})
                \end{array}
            \right\}
            \\
            &
            =\inf_{t>0} \sup \left\{
                \tildelip{\inv \somesetmap}(\alt{w}|\alt{q})
                \,\middle|\,
                \alt{w} \in \B(w, t),\, \alt{q} \in \B(q, t),\, \alt{w} \in \somesetmap(\alt{q})
            \right\},
        \end{aligned}
    \end{equation}
    where
    \begin{equation}
        \label{eq:tilde-lip-inv-r}
        \tildelip{\inv \somesetmap}(\alt{w}|\alt{q})
        \defeq
        \sup
        \left\{
            \norm{\dir{w}}
            \,\middle|\,
            \dir{q} \in \frechetCod \somesetmap(\alt{q}|\alt{w})(\dir{w}),\,
            \norm{\dir{q}} \le 1
        \right\}.
    \end{equation}
    Referral to \eqref{eq:somesetmap-cod-d-upper-adjoint} and the fact that
    \[
         [D\somesetmap(q|w)]^{*+}= [\widetilde{D\somesetmap}(q|w)]^{*+},
    \]
    now establishes the claim with the expression \eqref{eq:tildelip-inv-r-2} for $\tildelip{\inv \somesetmap}(\alt{w}|\alt{q})$.
\end{proof}

\subsection{Graphical derivatives expressed with linear operators and cones}
\label{sec:linear-cone}
 
We now derive necessary and sufficient conditions for the Aubin property to hold for variational inclusions involving second-order set-valued derivatives of pointwise functionals. As seen in \cref{sec:pointwise:examples}, these commonly have the structure of a sum of a linear operator and a cone. In fact, for the following analysis, it suffices that the graphical derivatives merely contain such a sum in order to derive upper bounds; this will be important for treating discretization by projection in \cref{sec:discretization}.
We therefore assume therefore that $W=Q=L^2(\Omega; \R^N)$ and that
\begin{equation}
    \label{eq:linear-polar-form}
    \widetilde{D \somesetmap}(q|w)(\dir{q})
    \supset
    \begin{cases}
        T_q \dir{q} + \polar{\DerivCone[q|w]}, & \dir{q} \in \DerivCone[q|w], \\
        \emptyset, & \dir{q} \not\in \DerivCone[q|w],
    \end{cases}
\end{equation}
for some linear operator $T\defeq T_q: Q \to Q$, dependent on $q$ but not $w$, and a cone $\DerivCone \defeq \DerivCone[q|w] \subset Q$, dependent on both $q$ and $w$.
Here we recall from \eqref{eq:polar} that $\polar\DerivCone$ is the \emph{polar cone} of $\DerivCone$. Although it will not be needed in our analysis, an explicit characterization of the regular coderivatives of set-valued mappings satisfying \eqref{eq:linear-polar-form} (with equality) is derived in \cref{sec:coderivative} for completeness.

\enlargethispage{1cm}
Following the reasoning in \cite[Prop.~4.1]{Valkonen:2014}, we may, using the structural assumption \eqref{eq:linear-polar-form}, continue from \cref{prop:lip-estim-regular} to derive
\begin{equation}
    \label{eq:tildelip-inv-h-3}
    \begin{aligned}[t]
        \tildelip{\inv R}(w|q)
        &
        \le
        \sup\left\{
            \norm{\dir{w}}
            \,\middle|\,
            \begin{array}{l}
                \dir{q} \in Q,\, \dir{w} \in Q,\, \norm{\dir{q}} \le 1, 
                ~\text{satisfying}
                \\
                \qquad
                \iprod{\dir{w}}{T \dir{\alt{q}}+\dir{\alt{p}}} \le \iprod{\dir{q}}{\dir{\alt{q}}}
                \\
                \hfill\text{for }
                \dir{\alt{q}} \in \DerivCone,\, \dir{\alt{p}} \in \polar\DerivCone
            \end{array}
        \right\}
        \\[0.5em]
        &
        =
        \sup\left\{
            \norm{\dir{w}}
            \,\middle|\,
            \begin{array}{l}
                \dir{q} \in Q,\, \dir{w} \in \DerivCone,\, \norm{\dir{q}} \le 1, 
                ~\text{satisfying}
                \\
                \hfill
                \iprod{\dir{w}}{T \dir{\alt{q}}} \le \iprod{\dir{q}}{\dir{\alt{q}}}
                \text{ for }
                \dir{\alt{q}} \in \DerivCone
            \end{array}
        \right\}
        \\[0.5em]
        &
        =
        \sup \left\{
            \norm{\dir{w}}
            \,\middle|\,
            \begin{array}{l}
                \dir{q} \in Q,\, \dir{w} \in \DerivCone,\, \norm{\dir{q}} \le 1, 
                ~\text{satisfying}
                \\
                \hfill
                \iprod{T^* \dir{w}-\dir{q}}{\dir{\alt{q}}} \le 0
                \text{ for }
                \dir{\alt{q}} \in \DerivCone
            \end{array}
        \right\}
        \\[0.5em]
        &
        =
        \sup \left\{
            \norm{\dir{w}}
            \,\middle|\,
                \dir{q} \in Q,\, \dir{w} \in \DerivCone,\, \norm{\dir{q}} \le 1,\,
                T^* \dir{w}-\dir{q} \in \polar\DerivCone
        \right\}
        \\[.3em]
        &
        =
        \sup \left\{
            \norm{\dir{w}}
            \,\middle|\,
                \dir{w} \in \DerivCone,\, \inf_{\dcVar \in \polar\DerivCone} \norm{T^* \dir{w}-\dcVar} \le 1
        \right\}.
    \end{aligned}
\end{equation}
We illustrate this expression geometrically in \cref{fig:polar-tildelip}.
Observe also that if \eqref{eq:linear-polar-form} holds as an equality, then so does the first inequality in \eqref{eq:tildelip-inv-h-3}. That is, in this case
\[
    \tildelip{\inv R}(w|q)
        =
        \sup \left\{
            \norm{\dir{w}}
            \,\middle|\,
                \dir{w} \in \DerivCone,\, \inf_{\dcVar \in \polar\DerivCone} \norm{T^* \dir{w}-\dcVar} \le 1
        \right\}.
\]    

\begin{figure}
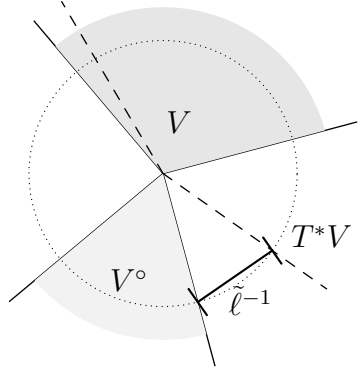

    \centering
    \asyinclude{polar.asy}
    \caption{A geometric illustration of the distance $\tildelip{\inv R}$ computed in \eqref{eq:tildelip-inv-h-3} for \eqref{eq:linear-polar-form}. The dashed line indicates the transformed cone $T^* \DerivCone$. Without loss of generality, we restrict $\dir{w}$ to lie on the unit sphere (dotted), in which case the distance between points on the unit sphere between the polar $\polar\DerivCone$ and $T^* \DerivCone$ gives the inverse Lipschitz constant.}
    \label{fig:polar-tildelip}
\end{figure}

\begin{remark}
    If $\DerivCone$ is a closed subspace, then $\polar \DerivCone=\DerivCone^\perp$, and  \eqref{eq:tildelip-inv-h-3} reduces to
    \begin{equation}
        \begin{aligned}[t]
            \tildelip{\inv R}(w|q)
            &
            \le
            \sup \left\{
                \norm{\dir{w}}
                \,\middle|\,
                \dir{w} \in \DerivCone,\, \norm{P_{\DerivCone} T^* \dir{w}} \le 1
            \right\}
            \\
            &
            =
            \sup \left\{
                \norm{P_{\DerivCone} \dir{w}}
                \,\middle|\,
                \dir{w} \in Q,\, \norm{P_{\DerivCone} T^* P_{\DerivCone} \dir{w}} \le 1
            \right\}.
        \end{aligned}
    \end{equation}
\end{remark}

We can use \eqref{eq:tildelip-inv-h-3} and the expansions above to estimate $\tildelip{\inv \somesetmap}(\alt{w}|\alt{q})$ for $\somesetmap=H_{\baseu}$ and $\somesetmap=\basesetmap$.
To study stability and metric regularity, we however still need to pass to
\[
        \lip{\inv R}(w|q)
        =
        \inf_{t>0} \sup \left\{
        \tildelip{\inv \somesetmap}(\alt{w}|\alt{q})
        \,\middle|\,
        \alt{w} \in \B(w, t),\, \alt{q} \in \B(q, t),\, \alt{w} \in \somesetmap(\alt{q})
    \right\}.
\]
This in essence involves a uniform $c>0$ in the condition
\[
    \inf_{\dcVar \in \polar{\DerivCone(\alt{q}|\alt{w})}} \norm{T_{\alt{q}}^* \dir{w}-\dcVar} \ge c\norm{\dir{w}}
    \qquad (\dir{w} \in \DerivCone(\alt{q}|\alt{w}))
\]
for all $(\alt{q}, \alt{w})$ close to $(q, w)$.

If we assume continuity of the mapping $\alt{q} \mapsto T_{\alt{q}}$, we can simplify this condition. The following lemma prepares the way for the stability analysis of saddle points in the next section (cf.~\eqref{eq:t-split} below).

\begin{lemma}
    \label{lemma:general-limit-polar-projection-lower-bound}
    Let $q, w \in Q=W=X \times Y$, and suppose that for $(\alt{q}, \alt{w})$ in a neighborhood $U$ of $(q, w)$, $\graph R \isect U$ is closed, \eqref{eq:somesetmap-cod-d-upper-adjoint} holds, and we have 
    \begin{equation}
        \label{eq:general-dr-expr}
        \widetilde{DR}(\alt{q}|\alt{w})(\dir{q}) \supset
        \begin{cases}
            T_{\alt{q}} \dir{q} + \polar{\DerivCone[\alt{q}|\alt{w}]}, & \dir{q} \in \DerivCone[\alt{q}|\alt{w}], \\
            \emptyset, & \dir{q} \not\in \DerivCone[\alt{q}|\alt{w}],
        \end{cases}
    \end{equation}
    for a cone $\DerivCone[\alt{q}|\alt{w}] \subset Y$.
    In addition to these structural assumptions, assume the continuity at $q$ of $\alt{q} \mapsto T_{\alt{q}}$, and for some $c>0$ the bound
    \begin{equation}
        \label{eq:grbound}
        \grbound(q|w; R) \defeq
            \sup_{t>0} \inf_{\substack{(\dir{w}, \dcVar) \in \DerivConeDXt[q|w]{t}{R}, \\ \dir{w} \ne 0}} 
            \frac{\norm{T_{q}^* \dir{w}-\dcVar}}{\norm{\dir{w}}}
            \ge c,
    \end{equation}
    where
    \begin{align}
        \label{eq:general-vt}
        \DerivConeDXt[q|w]{t}{R} & \defeq \Union\left\{
            \DerivCone[\alt{q}|\alt{w}] \times \polar{\DerivCone[\alt{q}|\alt{w}]}
            \,\middle|\,
            \alt{w} \in R(\alt{q}),\, \norm{\alt{q}-q} < t,\, \norm{\alt{w}-w} < t
            \right\} \subset Y^2.
    \end{align}
    Then
    \begin{equation}
        \label{eq:lip-inv-r-bound}
        \lip{\inv R}(w|q) \le \inv c.
    \end{equation}
    Moreover, if \eqref{eq:general-dr-expr} holds as an equality, then $\lipnum_{\inv R}(w|q) < \infty$  if and only if \( \grbound(q|w; R) > 0\).
\end{lemma}

\begin{proof}
    Suppose \eqref{eq:grbound} holds, and pick $c_1 \in (0, c)$.
    Whenever $t>0$ is small enough and $\alt{w}$ and $\alt{q}$ satisfy
    \[
        \alt{w} \in R(\alt{q}),\quad \norm{q-\alt{q}}<t\quad\text{and}\quad \norm{w-\alt{w}}<t,
    \]
    the bound \eqref{eq:grbound}, the continuity of $\alt{q} \mapsto T_{\alt{q}}$, and the inclusion
    \[
        \DerivCone(\alt{q}|\alt{w}) \times \polar{\DerivCone(\alt{q}|\alt{w})} \subset \DerivConeDXt[q|w]{t}{R},
    \]
    guarantee the estimate
    \[
        \norm{T_{\alt{q}} \dir{w} - \dcVar} \ge c_1 \norm{\dir{w}}
        \qquad (\dir{w} \in \DerivCone(\alt{q}|\alt{w}),\, \dcVar \in \polar{\DerivCone(\alt{q}|\alt{w})} ).
    \]
    The latter says that
    \[
        \tildelip{\inv R}(\alt{w}|\alt{q})  \le \inv c_1.
    \]
    By \eqref{eq:tildelip-inv-h-3} and \eqref{eq:inv-s-lip-estimate-coderivative}, therefore
    \[
        \lip{\inv R}(\alt{w}|\alt{q}) = \sup\{
            \tildelip{\inv R}(\alt{w}|\alt{q}) 
            \mid
            \alt{w} \in R(\alt{q}),\, \norm{q-\alt{q}}<t,\, \norm{w-\alt{w}}<t
            \}
        \le \inv c_1.
    \]
    Since $c_1 \in (0, c)$ was arbitrary, this proves \eqref{eq:lip-inv-r-bound}.

    If \eqref{eq:grbound} does not hold, and \eqref{eq:general-dr-expr} holds as an equality, we can, given $\eps>0$, find for every $t>0$ a pair
    $(\dir{w}, \dcVar) \in \DerivConeDXt[q|w]{t}{R} \setminus \{0\} \times Y$, such that
    \[
        \norm{T_{q}^* \dir{w}-\dcVar}\le\eps\norm{\dir{w}}.
    \]
    Thus, by the definition of $\DerivConeDXt[q|w]{t}{R}$, we can also find $\alt{q}$ and $\alt{w}$ satisfying
    \[
        \alt{w} \in R(\alt{q}),\quad \norm{q-\alt{q}}<t\quad\text{and}\quad \norm{w-\alt{w}}<t
    \]
    such that
    \[
        \dir{w} \in \DerivCone(\alt{q}|\alt{w})\quad\text{and}\quad \dcVar \in \polar{\DerivCone(\alt{q}|\alt{w})}.
    \]
    Recalling \eqref{eq:tildelip-inv-h-3}, which holds as an equality under the present assumption that \eqref{eq:general-dr-expr} holds as an equality, this implies that
    \[
        \tildelip{\inv R}(\alt{w}|\alt{q}) \ge \inv\eps.
    \]
    Since $t>0$ was arbitrary, we have as well that
    \[
        \lip{\inv R}(w|q) \ge \inv\eps.
    \]
    Finally, since $\eps>0$ was arbitrary, it follows that $\lipnum_{\inv R}(w|q) = \infty$  if \eqref{eq:grbound} does not hold.
\end{proof}

\section{Stability of non-linear saddle point systems}\label{sec:stability_sp}

We now apply the results of the preceding section to saddle points characterizing minimizers of nonsmooth optimization problems of the form \eqref{eq:prob_convex}. In particular, we assume that
\begin{equation}
    \label{eq:pde-f-choice}
    F^*(\Dual)=\int_\Omega f^*(\Dual(x)) \,d x
\end{equation}
for a proper, convex, lower semicontinuous $f^*$ and, motivated by the problems considered in the next section,
\begin{equation}
    \label{eq:pde-g-choice}
    G(u)=\int_\Omega g(u(x)) \,d x
    \quad
    \text{for}
    \quad
    g(\FdVar) = \frac{\alpha}{2} \abs{z}^2.
\end{equation}

\subsection{Non-linear saddle point systems as variational inclusions}
\label{sec:saddle-vi}

We first write the first-order optimality conditions \eqref{eq:oc} for the problem \eqref{eq:prob_convex} as an inclusion for a set-valued mapping and compute its derivative.

For $\realoptq=(\realopt\Primal,\realopt\Dual)$ to be a saddle point of \eqref{eq:lagrangian}, the Lagrangian $L$ has to satisfy
\[
    L(\realopt\Primal, \Dual)
    \le
    L(\realopt\Primal, \realopt\Dual)
    \le
    L(\Primal, \realopt\Dual)
    \qquad (\Primal \in X,\, \Dual \in Y).
\]
Since $-L(u, \freevar)$ is convex, proper, and lower semicontinuous for any $u \in X$, we deduce from the necessary and sufficient first-order optimality condition $0 \in \subdiff(-L(u, \freevar))(\realopt\Dual)$ for convex functions together with the sum rule \cite[Prop.~5.6]{Ekeland:1999} that $K(\realopt\Primal) \in \subdiff F^*(\realopt\Dual)$. 
We also see that
\[
    \realopt\Primal \in \argmin_\Primal~ G(u)+ \iprod{K(\Primal)}{\realopt\Dual}.
\]
Since $G$ is convex and $K \in C^1(X; Y)$, we can apply the calculus of Clarke's generalized derivative (which reduces to the Fréchet derivative and convex subdifferential for differentiable and convex functions, respectively; see, e.g., \cite[Chap.~2.3]{Clarke}) to deduce the overall system of critical point conditions
\begin{equation}
    \label{eq:oc-v2}
    \left\{ \begin{aligned}
    K(\realoptu) &\in \subdiff F^*(\realoptpsi),\\
    - [\grad K(\realoptu)]^* \realoptpsi &\in \subdiff G(\realoptu).
    \end{aligned} \right.
\end{equation}
This may be rewritten concisely as
\begin{equation}
    \label{eq:oc-h}
    0 \in H_{\realoptu}(\realoptq)
\end{equation}
for the monotone operator
\begin{equation}
    \label{eq:h-def}
    H_\baseu(u, \Dual) \defeq
        \begin{pmatrix}
            \subdiff G(u) + \grad K(\baseu)^* \Dual \\
            \subdiff F^*(\Dual) -\grad K(\baseu) u - c_\baseu
        \end{pmatrix},
    \quad
    \text{where}
    \quad
    c_\baseu \defeq K(\baseu)-\grad K(\baseu)\baseu.
\end{equation}
This is defined at an arbitrary base point $\baseu \in X$ for the linearization of $K$.
Here and generally we use the notation
\begin{equation}
    \label{eq:qw}
    q=(\Primal, \Dual) \in X \times Y
    \quad\text{and}\quad
    w=(\coPrimal, \coDual) \in X \times Y
\end{equation}
for combining (primal, dual) and (co-primal, co-dual) variable pairs, respectively. This nomenclature stems from $\Dual$ being the dual variable in the original saddle-point problem, whereas the co-primal and co-dual variables generally satisfy $w \in H_\baseu(q)$.

Alternatively, we may rewrite the critical point conditions \eqref{eq:oc-v2} as
\begin{equation}
    \label{eq:oc-r}
    0 \in \basesetmap(\realoptq)
\end{equation}
for
\begin{equation}
    \label{eq:r0-def}
    \basesetmap(u, \Dual) \defeq H_u(u, \Dual) = 
        \begin{pmatrix}
            \subdiff G(u) + \grad K(u)^* \Dual \\
            \subdiff F^*(\Dual) -K(u)
        \end{pmatrix}.
\end{equation}
The mapping $\basesetmap$ will be useful for general stability analysis, while $H_\baseu$ is critical for the primal-dual algorithm of \cite{Valkonen:2014}.

We can prove the following about these mappings.

\begin{proposition}
    \label{prop:dh}
    Let $G: X = L^2(\Omega; \R^m) \to (-\infty,\infty]$ and $F^*: Y = L^2(\Omega; \R^n) \to (-\infty,\infty]$ have the form \eqref{eq:g-pointwise-integral} for some regular integrands $g$ and $f^*$, respectively. 
    Let $K \in C^1(X; Y)$. Then $\graph H_\baseu$ is locally closed, and
    \begin{equation}
        \label{eq:dh}
        D H_\baseu(q|w)(\dir{q})
        =
        \begin{pmatrix}
            D{[\subdiff G]}(u|\coPrimal - \grad K(\baseu)^* \Dual)(\dir{u}) + \grad K(\baseu)^* \dir\Dual \\
            D{[\subdiff F^*]}(\Dual|\coDual + \grad K(\baseu)u + c_\baseu)(\dir\Dual) - \grad K(\baseu) \dir{u} \\
        \end{pmatrix},
    \end{equation}
    with $D[\subdiff G]$ and $D[\subdiff F^*]$ given by \eqref{eq:g-d-form}.
    Moreover \eqref{eq:somesetmap-cod-d-upper-adjoint} holds, i.e.,
    \begin{equation}
        \frechetCod H_\baseu(q|w) = [DH_\baseu(q|w)]^{*+}.
    \end{equation}
\end{proposition}

\begin{proof}
    That $\graph H_\baseu$ is locally closed is an immediate consequence of the lower semicontinuity of the convex functionals $G$ and $F^*$ and the continuity of $\grad K$.
    The expression \eqref{eq:dh} is an immediate consequence of \cref{cor:d-p-plus-smooth}, where we set
    \[
        h(u,\Dual) \defeq
            \begin{pmatrix}
                \grad K(\baseu)^*\Dual \\
                -\grad K(\baseu) u - c_\baseu
            \end{pmatrix}
        \quad\text{and}\quad
        P(u, \Dual) \defeq
            \begin{pmatrix}
                \subdiff G(u) \\
                \subdiff F^*(\Dual) \\
            \end{pmatrix},
    \]
    and observe that $h$ is not only smooth but linear with
    \[
        \begin{split} 
            \grad h(u, \Dual)= 
            \begin{pmatrix}
                0 & \grad K(\baseu)^* \\
                -\grad K(\baseu)u & 0
            \end{pmatrix}.
            \qedhere 
        \end{split}
    \]
\end{proof}

\begin{proposition}
    \label{prop:dr0}
    Let $G: X = L^2(\Omega; \R^m) \to (-\infty,\infty]$ and $F^*: Y = L^2(\Omega; \R^n) \to (-\infty,\infty]$ have the form \eqref{eq:g-pointwise-integral} for some regular integrands $g$ and $f^*$, respectively. 
    Let $K \in C^2(X; Y)$. Then $\graph \basesetmap$ is locally closed, and
    \begin{equation}
        \label{eq:dr0}
        D \basesetmap(q|w)(\dir{q})
        =
        \begin{pmatrix}
            D{[\subdiff G]}(u|\coPrimal - \grad K(\baseu)^* \Dual)(\dir{u}) + \grad_u [\grad K(u)^*\Dual]\dir\Primal + \grad K(u)^* \dir\Dual \\
            D{[\subdiff F^*]}(\Dual|\coDual + K(u))(\dir\Dual) - \grad K(u) \dir{u} \\
        \end{pmatrix},
    \end{equation}
    with $D[\subdiff G]$ and $D[\subdiff F^*]$ given by \eqref{eq:g-d-form}.
    Moreover \eqref{eq:somesetmap-cod-d-upper-adjoint} holds, i.e.,
    \begin{equation}
        \frechetCod \basesetmap(q|w) = [D\basesetmap(q|w)]^{*+}.
    \end{equation}
\end{proposition}

\begin{proof}
    Again, the fact that $\graph \basesetmap$ is locally closed is an immediate consequence of the lower semicontinuity of the convex functionals $G$ and $F^*$ and the continuity of $\grad K$. The expression \eqref{eq:dr0} is also again an immediate consequence of \cref{cor:d-p-plus-smooth}, where we set
    \[
        h_{0}(u,\Dual) \defeq
            \begin{pmatrix}
                \grad K(u)^*\Dual \\
                -K(u)
            \end{pmatrix}
        \quad\text{and}\quad
        P(u, \Dual) \defeq
            \begin{pmatrix}
                \subdiff G(u) \\
                \subdiff F^*(\Dual) \\
            \end{pmatrix},
    \]
    and observe that
    \[
        \grad h_0(u,\Dual)
        =\begin{pmatrix}
            \grad_u [\grad K(u)^*\Dual] & \grad K(u)^* \\
            -\grad K(u) & 0
        \end{pmatrix}, 
    \]
    where we denote $\grad_u [\grad K(u)^*\Dual] \defeq \grad(\tilde u \mapsto [\grad K(\tilde u)^*\Dual])(u)$, using the assumption that $K$ is twice differentiable.
\end{proof}

\begin{remark}
    Observe from \eqref{eq:dh} and \eqref{eq:dr0} that if $\baseu=u$, 
    \[
        D \basesetmap(q|w)=D H_u(q|w)+\begin{pmatrix} \grad_u [\grad K(u)^*\Dual]\dir\Primal \\ 0 \end{pmatrix}.
    \]
    Comparing \eqref{eq:h-def} and \eqref{eq:r0-def} also shows that in this case $\basesetmap(q)=H_u(q)$.
\end{remark}

Recalling \eqref{eq:inverse-aubin} and \eqref{eq:sensitivity1}, as well as \cref{prop:lip-estim-regular}, we see that in order to analyze the stability of \eqref{eq:oc}, resp.~\eqref{eq:oc-h}, we have to compute $\tildelip{\inv \basesetmap}(\alt{w}|\alt{q})$ in a neighborhood of $(\realoptq, 0)$. We will later see that this will be necessary both for $\baseu=\realoptu$ and $\baseu=\alt{u}$. 

\subsection{Lipschitz estimates for saddle points}

We now derive sufficient conditions for the Aubin property to hold for saddle points of \eqref{eq:oc-v2}. We proceed in several steps. First, we observe that provided that
if both $\widetilde{D[\subdiff G]}$ and $\widetilde{D[\subdiff F^*]}$ have individually the form \eqref{eq:linear-polar-form}, then the convexified graphical derivative $\compositeaccents{\widetilde}{D{H_{\bar u}}}(q|w)(\dir{q})$ also has the form \eqref{eq:linear-polar-form}. More precisely
\begin{equation}
    \label{eq:t-split}
    T_q = 
    \begin{pmatrix}
        \bar G_q & \bar K_\baseu^* \\
         - \bar K_\baseu & \bar F_q
    \end{pmatrix}
\end{equation}
for some linear operators $\bar G_q:X \to X$ and $\bar F_q: Y \to Y$ and $\bar K_\baseu=\grad K(\baseu)$, as well as the cone
\[
    \DerivCone[q|w]= \DerivConeG[\Primal|\coPrimal-\bar K_\baseu^* \Dual] \times \DerivConeF[\Dual|\coDual + \bar K_\baseu u + c_\baseu] \subset X \times Y.
\]
Since $G$ is assumed to be quadratic, we have $\DerivConeG[\Primal|\alt{\coPrimal}] \equiv X$, which gives the more specific structure
\begin{align}
    D [\subdiff G](u|\alt{\coPrimal})(\dir{u}) &= \bar G_q \dir{u}
    \\
    \intertext{and}
    \widetilde{D [\subdiff F^*]}(\Dual|\alt{\coDual})(\dir\Dual) &= 
    \begin{cases}
        \bar F_q \dir\Dual + \polar{\DerivConeF[\Dual|\alt{\coDual}]}, & \dir\Dual \in \DerivConeF[\Dual|\alt{\coDual}], \\
        \emptyset, & \dir\Dual \not\in \DerivConeF[\Dual|\alt{\coDual}].
    \end{cases}
    \label{eq:subspace-linear-structure}
\end{align}
We make, of course, the implicit assumption that $\alt{\coPrimal} \in \subdiff G(u)$ and $\alt{\coDual} \in \subdiff F^*(\Dual)$; if this does not hold, then the respective graphical derivatives are empty.

As we will see, it is difficult in general to guarantee the Aubin property. One way of doing so is to consider a Moreau--Yosida regularization of $F$, that is to replace $F^*$ by
\[
    F^*_\gamma(\Dual) \defeq F^*(\Dual) + \frac{\gamma}{2}\norm{\Dual}^2
\]
for some parameter $\gamma>0$; see, e.g., \cite[Chap.~12.4]{Bauschke:2011}.
The regular coderivative of the regularized subdifferential satisfies at least at non-degenerate points for some cone $\DerivConeF[\Dual|\coDual]$ the expression
\begin{equation}
    \label{eq:fstar-polar-form-huber}
    \widetilde{D{[\subdiff F_\gamma^*]}}(\Dual|\coDual)(\dir{\Dual})
    =
    \begin{cases}
        \gamma \dir{\Dual} + \polar{\DerivConeF[\Dual|\coDual]}, & \dir{\Dual} \in {\DerivConeF[\Dual|\coDual]}, \\
        \emptyset, & \dir{\Dual} \not\in {\DerivConeF[\Dual|\coDual]}.
    \end{cases}
\end{equation}
We denote the corresponding operator $H_\realoptu$ by $H_{\gamma,\realoptu}$.

From \cref{prop:dr0}, we observe that $\widetilde{DR_0}(q|w)(\dir{q})$ also has the form \eqref{eq:linear-polar-form} with \eqref{eq:t-split}, albeit with a different term $\bar K_\baseu$ and with $\bar G_q$ including the second-order term $\grad_u [\grad K(u)^*\Dual]$ from $K$.

\bigskip

We now specialize the results of \cref{sec:stability_vi} to the specific setting considered in this section. We therefore assume that $\bar F_q=\gamma I$ for some $\gamma \ge 0$ and that $\DerivConeG = X$. For the statement of the next lemma, we drop many of the subscripts and denote for short $T \defeq T_q$, $\bar K \defeq \bar K_\baseu$, $\bar G \defeq \bar G_q$, and $\tilde \DerivCone \defeq \DerivConeF[\Dual|\coDual]$.

\begin{lemma}
    \label{lemma:general-polar-projection-lower-bound}
    Let $\DerivCone = X \times \tilde\DerivCone \subset X \times Y$ be a cone, and let
    $\bar G: X \to X$ and $\bar K: X \to Y$ be bounded linear operators. For $\gamma \ge 0$, define 
    \begin{equation}
        T \defeq 
        \begin{pmatrix}
            \bar G & \bar K^* \\
             - \bar K & \gamma I
        \end{pmatrix}.
    \end{equation}
    Suppose $\bar G$ is self-adjoint and positive definite, i.e., there exists $c_G>0$ such that
    \begin{equation}
        \label{eq:general-c_g}
        \iprod{\bar G\coPrimal}{\coPrimal} \ge c_G^2 \norm{\coPrimal}^2 \qquad (\coPrimal \in X).
    \end{equation}
    Then, there exists $c>0$ such that
    \begin{equation}
        \label{eq:general-polar-projection-lower-bound}
        \inf_{\dcVar \in \polar\DerivCone} \norm{T^* w - \dcVar}^2 \ge c\norm{w}^2 \qquad (w \in \DerivCone)
    \end{equation}
    if and only if either of the following conditions hold:
    \begin{enumerate}[label=(\roman*)]
        \item\label{item:general-polar-projection-lower-bound-i}
            $\gamma>0$, 
            in which case $c=c(\gamma, c_G)$;
        \item\label{item:general-polar-projection-lower-bound-ii} 
            there exists $c_{K,\DerivCone}>$ such that
            \begin{equation}\displayindent0pt \displaywidth\columnwidth
                \label{eq:ckv-0}
                \inf_{\nu \in \polar{\tilde\DerivCone}} \norm{\bar K \inv{\bar G} \bar K^* \coDual - \nu}^2 \ge c_{K,\DerivCone} \norm{\coDual}^2
                \qquad (\coDual \in \tilde\DerivCone),
            \end{equation}
            in which case $c=c(\norm{\bar K}, c_G, c_{K,\DerivCone}) \le c_{K, \DerivCone}$.
    \end{enumerate}
\end{lemma}

\begin{proof}
    We first prove the sufficiency of \ref{item:general-polar-projection-lower-bound-i} and \ref{item:general-polar-projection-lower-bound-ii}.
    With $w=(\coPrimal,\coDual) \in \DerivCone = X \times \DerivConeF$, and $\dcVar=(0, \nu) \in \polar\DerivCone = \{0\} \times \polar \DerivConeF$, we calculate
    \begin{equation}
        \label{eq:general-polar-projection-lower-bound-est0}
        \begin{aligned}
        \norm{T^*w-\dcVar}^2
        & =
        \norm{\bar G \coPrimal-\bar K^*\coDual}^2
        +
        \norm{\bar K \coPrimal+\gamma\coDual-\nu}^2
        \\
        & =
        \norm{\bar G \coPrimal}^2+\norm{\bar K^*\coDual}^2-2\iprod{\bar G \coPrimal}{\bar K^*\coDual}
        \\ \MoveEqLeft[-1]
        +\norm{\bar K \coPrimal}^2+\gamma^2\norm{\coDual}^2+\norm{\nu}^2
        +2\gamma\iprod{\bar K \coPrimal}{\coDual}
        -2\iprod{\bar K \coPrimal}{\nu}
        -2\gamma\iprod{\coDual}{\nu}
        \\
        & =
        \norm{\bar G \coPrimal}^2+\norm{\bar K^*\coDual}^2-2\iprod{\bar G \coPrimal}{\bar K^*\coDual}
        +\norm{\bar K\coPrimal-\nu}^2
        +\gamma^2\norm{\coDual}^2+2\gamma\iprod{\bar K \coPrimal}{\coDual}
       -2\gamma\iprod{\coDual}{\nu}.
       \end{aligned}
    \end{equation}
    Assume first that $\gamma>0$. For arbitrary $\lambda\in [0,\gamma]$, we can insert the productive zero and use $(\lambda-\gamma)\iprod{\nu}{\coDual} \ge 0$ for all $\nu \in \polar {\tilde\DerivCone}$ and $\coDual \in {\tilde\DerivCone}$ to obtain
    \begin{equation}
        \label{eq:general-polar-projection-lower-bound-lambda-introduce}
        \begin{aligned}
        \norm{T^*w-\dcVar}^2
        & =
        \norm{\bar G \coPrimal}^2+\norm{\bar K^*\coDual}^2-2\iprod{(\bar G +\lambda I-\gamma I)\coPrimal}{\bar K^*\coDual}
        \\ \MoveEqLeft[-1]
        +2\lambda\iprod{\bar K \coPrimal}{\coDual}
        +\norm{\bar K\coPrimal-\nu}^2
        +\gamma^2\norm{\coDual}^2
        -2\gamma\iprod{\coDual}{\nu}\\
        & \ge
        \norm{\bar G \coPrimal}^2+\norm{\bar K^*\coDual}^2-2\iprod{(\bar G +\lambda I-\gamma I)\coPrimal}{\bar K^*\coDual}
        \\ \MoveEqLeft[-1]
        +2\lambda\iprod{\bar K\coPrimal-\nu}{\coDual}
        +\norm{\bar K\coPrimal-\nu}^2
        +\gamma^2\norm{\coDual}^2.
        \end{aligned}
    \end{equation}
    This we further estimate by application of Young's inequality for any $\rho_1,\rho_2>0$ as
    \begin{equation}
        \label{eq:general-polar-projection-lower-bound-est2}
        \begin{aligned}[t]
        \norm{T^*w-\dcVar}^2
        & \ge
        \norm{\bar G \coPrimal}^2+\norm{\bar K^*\coDual}^2-2\iprod{(\bar G +\lambda I-\gamma I)\coPrimal}{\bar K^*\coDual}
        \\ \MoveEqLeft[-1]
        +(\gamma^2-\lambda\inv\rho_1)\norm{\coDual}^2 +(1-\lambda\rho_1)\norm{\bar K\coPrimal-\nu}^2
        \\
        & \ge
        \norm{\bar G \coPrimal}^2-\inv\rho_2\norm{(\bar G +\lambda I-\gamma I)\coPrimal}^2+(1-\rho_2)\norm{\bar K^*\coDual}^2
        \\ \MoveEqLeft[-1]
        +(\gamma^2-\lambda\inv\rho_1)\norm{\coDual}^2
        +(1-\lambda\rho_1)\norm{\bar K\coPrimal-\nu}^2.
        \end{aligned}
    \end{equation}
    Let us choose $\rho_1=\inv\lambda$ and $\rho_2=1$. Then \eqref{eq:general-polar-projection-lower-bound-est2} becomes
    \[
        \norm{T^*w-\dcVar}^2
        \ge
        \norm{\bar G \coPrimal}^2-\norm{(\bar G +\lambda I-\gamma I)\coPrimal}^2
        +(\gamma^2-\lambda^2)\norm{\coDual}^2.
    \]
    Now,
    \[
        \norm{\bar G \coPrimal}^2-\norm{(\bar G +\lambda I-\gamma I)\coPrimal}^2
        =
        2(\gamma-\lambda)\iprod{\bar G \coPrimal}{\coPrimal}-(\gamma-\lambda)^2\norm{\coPrimal}^2,
    \]
    so that by \eqref{eq:general-c_g} we therefore require that
    \[
        2(\gamma-\lambda)c_G - (\gamma-\lambda)^2 > 0.
    \]
    This holds if $\lambda < \gamma$ is large enough, verifying case \ref{item:general-polar-projection-lower-bound-i} including the relationship $c=c(\gamma, c_G)$.

    Suppose next that $\gamma=0$. To verify the sufficiency of \ref{item:general-polar-projection-lower-bound-ii}, we proceed by contradiction, assuming \eqref{eq:general-polar-projection-lower-bound} not to hold for
    \[
        c= \frac{c_{K,\DerivCone}}{1+\norm{\bar K \inv{\bar G}}}.
    \]
    Thus, for some $c' \in (0, c)$, we can find $w \in X \times {\tilde\DerivCone}$ and $\nu \in \polar{\tilde\DerivCone}$ satisfying
    \[
        \norm{T^* w - (0, \nu)}^2 \le c'\norm{w}^2.
    \]
    We may assume that $\norm{w}=1$. Thus,
    \[
        T^* w-(0, \nu)=(e_1, e_2)
        \quad\text{where}\quad \norm{e_1}^2+\norm{e_2}^2 \le c',
    \]
    which means
    \[
        \bar G \coPrimal-\bar K^*\coDual=e_1
        \quad
        \text{and}
        \quad
        \bar K \coPrimal -\nu=e_2.
    \]
    Since $\bar G$ is invertible by \eqref{eq:general-c_g}, this shows that
    \[
        \bar K \inv{\bar G} \bar K^* \coDual-\nu=e_2-\bar K \inv{\bar G}e_1.
    \]
    Thus
    \[
        \norm{\bar K \inv{\bar G} \bar K^* \coDual - \nu}^2
        \le 2(1+\norm{\bar K \inv{\bar G}}) c'
        < c_{K,\DerivCone},
    \]
    in contradiction to \eqref{eq:ckv-0}.
    Therefore \ref{item:general-polar-projection-lower-bound-ii} is sufficient for \eqref{eq:general-polar-projection-lower-bound}. We also estimate
    \[
        \norm{\bar G \coPrimal} \ge c_G \norm{\coPrimal}.
    \]
    Using the standard relation
    \[
        \sup_{\norm{\coPrimal}=1} \norm{\inv{\bar G}\coPrimal} = \sup_{\coPrimal \ne 0} \frac{\norm{\coPrimal}}{\norm{\bar G \coPrimal}},
    \]
    we therefore have
    \[
        \norm{\bar K \inv{\bar G}} \le c_G^{-1}\norm{\bar K}.
    \]
    This verifies $c=c(\norm{K}, c_G, c_{K,\DerivCone})$.

    Having dealt with the sufficient conditions, let us now verify the necessity of \eqref{eq:ckv-0} when $\gamma=0$. We expand
    \begin{equation}
        \label{eq:general-polar-projection-lower-bound-necessary-conds}
        \norm{T^*w-\dcVar}^2 =
        \norm{\bar G \coPrimal - \bar K^* \coDual}^2
        +\norm{\bar K\coPrimal-\nu}^2.
    \end{equation}
    Using the invertibility of $\bar G$ from \eqref{eq:general-c_g}, let us choose $\coPrimal=\inv{\bar G}\bar K^*\coDual$. Then \eqref{eq:general-polar-projection-lower-bound-necessary-conds} gives
    \[
        \norm{T^*w-\dcVar}^2 = \norm{\bar K\inv{\bar G}\bar K^*\coDual-\nu}^2,
    \]
    immediately showing the necessity of \eqref{eq:ckv-0} and
    $c_{K,\DerivCone} \ge c$.
\end{proof}

\begin{remark}
    \label{rem:barkstar-simplified-necessary}
    It is easily seen that if $\gamma=0$, then existence of a $c>0$ such that
    \begin{equation}
        \label{eq:ckv-1}
        \norm{\bar K^*\coDual} \ge c \norm{\coDual}
        \quad
        (\coDual \in \DerivConeF)
    \end{equation}
    is necessary for the satisfaction of \eqref{eq:general-polar-projection-lower-bound}.
\end{remark}

We now combine the above low-level lemma with \cref{lemma:general-limit-polar-projection-lower-bound}.

\begin{lemma}
    \label{lemma:limit-polar-projection-lower-bound}
    Let $q, w \in Q=W=X \times Y$ and suppose that for $(\alt{q}, \alt{w})$ in a neighborhood $U$ of $(q, w)$, $\graph R \isect U$ is closed, \eqref{eq:somesetmap-cod-d-upper-adjoint} holds, and we have 
    \begin{equation}
        \label{eq:dr-expr}
        \widetilde{DR}(\alt{q}|\alt{w})(\dir{q}) \supset
        \begin{cases}
            T_{\alt{q}} \dir{q} + \polar{\DerivConeX[\alt{q}|\alt{w}]{R}}, & \dir{q} \in \DerivConeX[\alt{q}|\alt{w}]{R}, \\
            \emptyset, & \dir{q} \not\in \DerivConeX[\alt{q}|\alt{w}]{R},
        \end{cases}
    \end{equation}
    for
    \begin{equation}
        T_{\alt{q}} = 
        \begin{pmatrix}
            \bar G_{\alt{q}} & \bar K_{\alt{u}}^* \\
             - \bar K_{\alt{u}} & \gamma I
        \end{pmatrix},
    \end{equation}
    and
    \begin{equation}
        \label{eq:limit-polar-projection-lower-bound-derivcone}
        \DerivConeX[\alt{q}|\alt{w}]{R}= X \times \tilde\DerivCone(\alt{\Dual}|\alt{\coDual}) \subset X \times Y.
    \end{equation}
    In addition to these structural assumptions, suppose that the mappings $\alt{q} \mapsto \bar G_{\alt{q}}$ and $\alt{u} \mapsto \bar K_{\alt{u}}$ are continuous at $q$ and $u$, respectively. Assume, moreover, that each $\bar G_{\alt{q}}$ is self-adjoint and positive definite, i.e., there exists $c_G>0$ such that
    \begin{gather}
        \label{eq:limit-c_G}
        \iprod{\bar G_q\coPrimal}{\coPrimal} \ge c_G \norm{\coPrimal}^2 \qquad (\coPrimal \in X).
    \end{gather}
    Define further
   \begin{equation}
        \label{eq:kvbound}
        \kvbound(q|w; R) \defeq
            \sup_{t>0} \inf_{\substack{((0, \dir{\coDual}), (0, \nu)) \in \DerivConeDXt[q|w]{t}{R}, \\ \dir{\coDual} \ne 0}} \frac{\norm{\bar K_\Primal \inv{\bar G_\Primal} \bar K_\Primal^*\dir{\coDual}-\nu}}{\norm{\dir{\coDual}}}.
    \end{equation}
    Then
    \begin{equation}
        \label{eq:kvbound-lipnum}
        \lipnum_{\inv R}(w|q) < \infty 
    \end{equation}
    provided 
    \begin{gather}
        \label{eq:limit-cKF}
        \max\{\gamma, \kvbound(q|w; R)\} > 0. 
    \end{gather}
    If \eqref{eq:dr-expr} holds as an equality, then \eqref{eq:kvbound-lipnum} holds if and only if \eqref{eq:limit-cKF} holds.
\end{lemma}

\begin{proof}
    If $\gamma>0$, we may directly apply \cref{lemma:general-limit-polar-projection-lower-bound}. So we take $\gamma=0$.
    Suppose first that \eqref{eq:limit-cKF} holds. Then $\kvbound(q|w; R) =: c_{K, \DerivCone}>0$, and \eqref{eq:kvbound} gives
    \begin{equation}
        \label{eq:limit-polar-projection-lower-bound-step1}
        \frac{\norm{\bar K_\Primal \inv{\bar G_\Primal} \bar K_\Primal^*\dir{\coDual}-\nu}}{\norm{\dir{\coDual}}} \ge c_{K, \DerivCone}
    \end{equation}
    for every $\dir{\coDual} \ne 0$ and $\nu$ satisfying
    \[
        ((0, \dir{\coDual}), (0, \nu)) \in \DerivConeDXt[q|w]{t}{R}.
    \]
    That is, using the facts that $0 \in X$ and $0 \in \polar X$, as well as the expression \eqref{eq:limit-polar-projection-lower-bound-derivcone}, we see that \eqref{eq:limit-polar-projection-lower-bound-step1} holds whenever
    \[
        \dir{\coDual} \in \tilde\DerivCone(\alt{\Dual}|\alt{\coDual})
        \quad\text{and}\quad
        \nu \in \polar{\tilde\DerivCone(\alt{\Dual}|\alt{\coDual})}
    \]
    for some $\alt{q}=(\alt{\Primal}, \alt{\Dual})$ and $\alt{w}=(\alt{\coPrimal}, \alt{\coDual})$ satisfying
    \[
        \alt{w} \in R(\alt{q}),\quad \norm{\alt{q}-q} < t,\quad \norm{\alt{w}-w} < t.
    \]
    With $\alt{q}$ and $\alt{w}$ fixed, \cref{lemma:general-polar-projection-lower-bound} now shows the existence of a constant $c>0$ such that
    \begin{equation}
        \label{eq:general-polar-projection-lower-bound-use1}
        \norm{T^* \dir{w} - \dcVar}^2 \ge c\norm{\dir{w}}^2
    \end{equation}
    for all
    \[
        (\dir{w}, \dcVar) \in (X \times \tilde\DerivCone(\alt{\Dual}|\alt{\coDual})) \times \polar{(X \times \tilde\DerivCone(\alt{\Dual}|\alt{\coDual}))},
    \]
    with $c$ depending only on $\norm{\bar K}$, $c_G$, and $c_{K,\DerivCone}$.
    Therefore \eqref{eq:general-polar-projection-lower-bound-use1} holds for all
    \[
        (\dir{w}, \dcVar) \in \DerivConeDXt[q|w]{t}{R}.
    \]
    Applying \eqref{eq:general-polar-projection-lower-bound-use1} in the expression for $\grbound$ in \eqref{eq:grbound} now shows that
    \[
        \grbound(q|w; R) \ge c.
    \]
    Finally, an application of \cref{lemma:general-limit-polar-projection-lower-bound} shows that $\lipnum_{\inv R}(w|q) < \infty$.

    In the other direction, to show that $\lipnum_{\inv R}(w|q) = \infty$ if $\kvbound(q|w; R)=0$, we assume to the contrary that  $\lipnum_{\inv R}(w|q) < \infty$. Then $\grbound(q|w; R) \ge c$ for some constant $c>0$. Now we perform the above steps in the opposite direction to show that $\kvbound(q|w; R)>0$, in contradiction to the premise.
\end{proof}

The following theorem, which
specializes \cref{lemma:limit-polar-projection-lower-bound} to the specific structure assumed in this section and estimates the lower bounds slightly to derive easier conditions, is one of the main results of this work.
\begin{theorem}
    \label{thm:limit-polar-projection-lower-bound-fstar}
    Let $q, w \in Q=W=X \times Y$ and let $U$ be a neighborhood of $(q, w)$.
    Suppose that
    \begin{equation}
        \label{eq:limit-polar-projection-lower-bound-fstar-r}
        R(\alt{q})=P(\alt{q})+h(\alt{q})
        \quad\text{for}\quad
        P(\alt{q})=\begin{pmatrix} \grad G(\alt{\Primal}) \\ \subdiff F^*(\alt{\Dual})\end{pmatrix}
        \quad\text{and}\quad
        h(\alt{q})=\begin{pmatrix}\grad \bar h(\alt{u})^*\alt{v} \\ -\bar h(\alt{u}) \end{pmatrix}
    \end{equation}
    for $G$ and $F^*$ of the form \eqref{eq:g-pointwise-integral} for some regular integrands $g$ and $f^*$, respectively. 
    Assume further that $\bar h \in C^1(X; Y)$ and $G \in C^2(X)$, and that $F^*$ satisfies
    for some $\gamma \ge 0$ the inclusion
    \begin{equation}
        \label{eq:dfstar-expr}
        \widetilde{D[\subdiff F^*]}(\alt{\Dual}|\alt{\coDual})(\dir{\Dual}) \supset
        \begin{cases}
            \gamma \dir{\Dual} + \polar{\DerivConeF(\alt{\Dual}|\alt{\coDual})}, & \dir{\Dual} \in \DerivConeF[\alt{\Dual}|\alt{\coDual}], \\
            \emptyset, & \dir{\Dual} \not\in \DerivConeF[\alt{\Dual}|\alt{\coDual}].
        \end{cases}
    \end{equation}
    In addition to these structural assumptions, suppose that there exists a constant $c_G>0$ such that
    \begin{gather}
        \label{eq:limit-c_G-fstar}
        \iprod{\grad^2 G(u)\coPrimal + \grad_u[\grad \bar h(u)^*\Dual]\coPrimal}{\coPrimal} \ge c_G \norm{\coPrimal}^2 \qquad (\coPrimal \in X).
    \end{gather}
    Define for
    \[
        \bar B \defeq \grad \bar h(u)\inv{\bigl(\grad^2 G(u) + \grad_u[\grad \bar h(u)^*\Dual]\bigr)}\grad \bar h(u)^*
    \]
    and
    \begin{equation}
        \label{eq:vt-fstar}
        \DerivConeDFt[\Dual|\coDual]{t} \defeq \Union\left\{
            \DerivConeF[\alt{\Dual}|\alt{\coDual}] \times \polar{\DerivConeF[\alt{\Dual}|\alt{\coDual}]}
            \,\middle|\,
            \alt{\coDual} \in \subdiff F^*(\alt{\Dual}),\, \norm{\alt{\Dual}-\Dual} < t,\, \norm{\alt{\coDual}-\coDual} < t
            \right\}
    \end{equation}
    the quantity
    \begin{equation}
        \label{eq:kvbound-fstar}
        \bar\kvbound(q|w; R) \defeq
            \sup_{t>0} \inf\left\{ 
                \frac{\norm{\bar B \dcVar-\nu}}{\norm{\dcVar}}
                ~\middle|~
                (\dcVar, \nu) \in \DerivConeDFt[\Dual|\coDual+\bar h(\Primal)]{t},\, \dcVar \ne 0
                \right\}.
    \end{equation}
    Then
    \begin{equation}
        \label{eq:lipnum-kvbound-fstar}
        \lipnum_{\inv R}(w|q) < \infty 
    \end{equation}
    provided
    \begin{equation}
        \label{eq:limit-cKF-fstar}
        \max\{\gamma, \bar\kvbound(q|w; R)\} > 0. 
    \end{equation}
    If \eqref{eq:dfstar-expr} holds as an equality, then \eqref{eq:lipnum-kvbound-fstar} holds if and only if \eqref{eq:limit-cKF-fstar} holds.
\end{theorem}

\begin{proof}
    Similarly to \cref{prop:dr0}, we compute
    \begin{equation}
        D \somesetmap(\alt{q}|\alt{w})(\dir{q})
        =
        \begin{pmatrix}
            \grad^2 G(\alt{u})(\dir\Primal) + \grad_u [\grad \bar h(\alt{u})^*\alt\Dual]\dir\Primal + \grad \bar h(\alt{u})^* \dir\Dual \\
            D{[\subdiff F^*]}(\alt\Dual|\alt\coDual + \bar h(\alt{u}))(\dir\Dual) - \grad \bar h(\alt{u}) \dir{u} \\
        \end{pmatrix},
    \end{equation}
    and
    \begin{equation}
        \label{eq:limit-polar-projection-lower-bound-fstar-convexdr}
        \widetilde{D \somesetmap}(\alt{q}|\alt{w})(\dir{q})
        =
        \begin{pmatrix}
            \grad^2 G(\alt{u})(\dir\Primal) + \grad_u [\grad \bar h(\alt{u})^*\alt\Dual]\dir\Primal + \grad \bar h(\alt{u})^* \dir\Dual \\
            \widetilde{D{[\subdiff F^*]}}(\alt\Dual|\alt\coDual + \bar h(\alt{u}))(\dir\Dual) - \grad \bar h(\alt{u}) \dir{u} \\
        \end{pmatrix},
    \end{equation}
    where we denote
    \[
        \grad_u [\grad \bar h(\alt{u})^*\alt\Dual]\dir\Primal
        \defeq
        \grad\left(\tilde u \mapsto [\grad \bar h(\tilde{u})^*\alt\Dual]\dir\Primal\right)(\alt u).
    \]
    The structural assumptions of \cref{lemma:limit-polar-projection-lower-bound} are thus satisfied with 
    \[
        \bar G_{\alt{q}} \defeq \grad^2 G(\alt{u})+\grad_u[\grad \bar h(\alt{u})^*\alt{\Dual}],
        \quad
        \bar K_{\alt{u}} 
        \defeq
        \grad \bar h(\alt{u}),
        \quad\text{and}\quad
        \DerivConePrime[\alt{q}|\alt{w}] \defeq \DerivConeF[\alt{\Dual}|\alt{\coDual}+\bar h(\alt{\Primal})].
    \]
    Moreover, Graph $R \isect U$ is closed due to the assumptions on $\bar h$ and $G$ and to $G$ and $F^*$ being convex. Further, \eqref{lemma:limit-polar-projection-lower-bound} holds by \cref{cor:d-p-plus-smooth}.

    Condition \eqref{eq:limit-c_G} is guaranteed by \eqref{eq:limit-c_G-fstar}, while for \eqref{eq:limit-cKF}, we first of all observe that
    \[
        \bar B = \bar K_\Primal \inv{\bar G_\Primal} \bar K_\Primal^*\dir{\coDual},
    \]
    so \eqref{eq:kvbound} becomes
    \begin{equation}
        \label{eq:kvbound-fstar-lemma-orig}
        \kvbound(q|w; R) =
            \sup_{t>0} \inf_{\substack{((0, \dir{\coDual}), (0, \nu)) \in \DerivConeDXt[q|w]{t}{R}, \\ \dir{\coDual} \ne 0}} \frac{\norm{\bar B\dir{\coDual}-\nu}}{\norm{\dir{\coDual}}}.
    \end{equation}
    Here
    \[
        \DerivConeDXt[q|w]{t}{R} = \Union \left\{ (X \times V) \times (\{0\} \times \polar V) \mid V \in \widetilde{\DerivConeDXt[q|w]{t}{R}} \right\}
    \]
    with
    \[
        \widetilde{\DerivConeDXt[q|w]{t}{R}} = 
            \left\{
            {\DerivConeF[\alt{\Dual}|\alt{\coDual}-\bar h(\alt{\Primal})]}
            \,\middle|\,
            \begin{array}{ll}
            \alt{\coPrimal}=\grad G(\alt{\Primal})+\grad \bar h(\alt{\Primal})^*\alt{\Dual},& \norm{\alt{q}-q} < t, \\
            \alt{\coDual} \in \subdiff F^*(\alt{\Dual})-\bar h(\alt{\Primal}), &
             \norm{\alt{w}-w} < t 
            \end{array}
            \right\}.
    \]
    We derive for small $t>0$ and some constant $C > 1$ (depending on $(q, w)$) the inclusion
    \begin{equation}
        \begin{aligned}
        \widetilde{\DerivConeDXt[q|w]{t}{R}} 
            &
            =
             \left\{
            {\DerivConeF[\alt{\Dual}|\alt{\coDual}]}
            \,\middle|\,
            \begin{array}{l}
            \norm{\alt{u}-u}^2 + \norm{\alt{\Dual}-\Dual}^2 < t^2,\,
             \alt{\coDual} \in \subdiff F^*(\alt{\Dual}), \\
            \norm{\grad G(\alt{\Primal})+\grad \bar h(\alt{\Primal})^*\alt{\Dual}-\coPrimal}^2 +\norm{\alt{\coDual}-\bar h(\alt{\Primal})-\coDual}^2 < t^2
            \end{array}
            \right\}
            \\
            &
            \subset
             \left\{
            {\DerivConeF[\alt{\Dual}|\alt{\coDual}]}
            \,\middle|\,
            \begin{array}{l}
            \norm{\alt{\Primal}-\Primal}^2+\norm{\alt{\Dual}-\Dual}^2 < t^2,\,
             \alt{\coDual} \in \subdiff F^*(\alt{\Dual}), \\
            \norm{\alt{\coDual}-\bar h(\alt{\Primal})-\coDual}^2 < t^2
            \end{array}
            \right\}
            \\
            &
            \subset
             \left\{
            {\DerivConeF[\alt{\Dual}|\alt{\coDual}]}
            \,\middle|\,
            \begin{array}{l}
            \norm{\alt{\Primal}-\Primal} < t,\,
            \norm{\alt{\Dual}-\Dual} < t,\,
             \alt{\coDual} \in \subdiff F^*(\alt{\Dual}), \\
            \norm{\alt{\coDual}-\bar h(\Primal)-\coDual} - \norm{\bar h(\alt{\Primal})-\bar h(\Primal)} < t
            \end{array}
            \right\}
            \\
            &
            \subset
             \left\{
            {\DerivConeF[\alt{\Dual}|\alt{\coDual}]}
            \,\middle|\,
            \begin{array}{l}
            \norm{\alt{\Dual}-\Dual} < t,\,
             \alt{\coDual} \in \subdiff F^*(\alt{\Dual}), \\
            \norm{\alt{\coDual}-\bar h(\Primal)-\coDual} < C t
            \end{array}
            \right\}
            =: \mathcal{V}_C
            .
        \end{aligned}
    \end{equation}
    In the final step we have used the fact that $\bar h \in C^1(X; Y)$ is Lipschitz on $\B(u, t)$ for small $t>0$. Now
    \[
        \Union\{ V \times \polar V \mid V \in \mathcal{V}_C \}
        \subset \DerivConeDFt[\Dual|\coDual+\bar h(\Primal)]{Ct}.
    \]
    Since we take the supremum over  $t>0$ in \eqref{eq:kvbound}, the scaling factor $C>0$ disappears, and we deduce from \eqref{eq:kvbound-fstar} and \eqref{eq:kvbound-fstar-lemma-orig} that
    \[
        \kvbound(u|w; R) 
        \ge
        \bar \kvbound(u|w; R).
    \]
    Thus \eqref{eq:kvbound-fstar} guarantees \eqref{eq:kvbound}.
    Similarly, retracing the steps, we verify that
    \[
        \DerivConeDFt[\Dual|\coDual-\bar h(\Primal)]{t} = \Union\{ V \times \polar V \mid V \in \mathcal{V}_1 \}
        \quad\text{and}\quad
        \mathcal{V}_1 \subset \widetilde{\DerivConeDXt[q|w]{C_2 t}{R}}
    \]
    for some $C_2>1$.
    Indeed, using
    \[
        \grad G(\Primal)+\grad \bar h(\Primal)^*\Dual=\coPrimal,
    \]
    we compute for some $C_1>1$ that
    \begin{equation}
        \begin{aligned}
            \mathcal{V}_1
            & =
             \left\{
            {\DerivConeF[\alt{\Dual}|\alt{\coDual}]}
            \,\middle|\,
            \begin{array}{l}
            \norm{\alt{\Dual}-\Dual} < t,\,
             \alt{\coDual} \in \subdiff F^*(\alt{\Dual}), \\
            \norm{\alt{\coDual}-\bar h(\Primal)-\coDual} < t
            \end{array}
            \right\}
            \\
            & =
             \left\{
            {\DerivConeF[\alt{\Dual}|\alt{\coDual}]}
            \,\middle|\,
            \begin{array}{l}
            \norm{\alt{\Dual}-\Dual} < t,\,
             \alt{\coDual} \in \subdiff F^*(\alt{\Dual}), \\
            \norm{\grad G(\Primal)+\grad \bar h(\Primal)^*\Dual-\coPrimal} + \norm{\alt{\coDual}-\bar h(\Primal)-\coDual} < t
            \end{array}
            \right\}
            \\
            & \subset
             \left\{
            {\DerivConeF[\alt{\Dual}|\alt{\coDual}]}
            \,\middle|\,
            \begin{array}{l}
            \norm{\alt{\Dual}-\Dual} < t,\,
             \alt{\coDual} \in \subdiff F^*(\alt{\Dual}), \\
            \norm{\grad G(\Primal)+\grad \bar h(\Primal)^*\alt{\Dual}-\coPrimal} + \norm{\alt{\coDual}-\bar h(\Primal)-\coDual} < C_1 t
            \end{array}
            \right\}
            \\
            &
            \subset
            \left\{
            {\DerivConeF[\alt{\Dual}|\alt{\coDual}]}
            \,\middle|\,
            \begin{array}{l}
            \norm{\alt{\Primal}-\Primal} + \norm{\alt{\Dual}-\Dual} < t,\,
             \alt{\coDual} \in \subdiff F^*(\alt{\Dual}), \\
            \norm{\grad G(\alt{\Primal})+\grad \bar h(\alt{\Primal})^*\alt{\Dual}-\coPrimal}^2 + \norm{\alt{\coDual}-\bar h(\alt{\Primal})-\coDual}^2 < C_2 t^2
            \end{array}
            \right\}
            \\
            &
            \subset \widetilde{\DerivConeDXt[q|w]{C_2 t}{R}}.
        \end{aligned}
    \end{equation}
    Hence,
    \[
        \bar \kvbound(u|w; R) 
        \ge
        \kvbound(u|w; R),
    \]
    and in particular, \eqref{eq:kvbound} guarantees \eqref{eq:kvbound-fstar}.
    Our claims now follow from an application of \cref{lemma:limit-polar-projection-lower-bound}, since its continuity requirements on $\bar G$ and $\bar K$ follow from the assumptions on $\bar G$ and $\bar h$.
\end{proof}

\bigskip

In the remainder of this section, we apply \cref{thm:limit-polar-projection-lower-bound-fstar} to show several stability properties of saddle points to \eqref{eq:lagrangian}.

\subsection{Metric regularity of the linearized variational inclusion}

We begin with the simplest example of verifying $\ell_{H_{\gamma,\baseu}}(0|\realoptq)<\infty$ with fixed $\baseu=\realoptu$ for $\realoptq$ solving $0 \in H_{\gamma,\realoptu}(\realoptq)$. 
This is useful for showing convergence of the primal-dual algorithm of \cite{Valkonen:2014}.
By \cref{prop:dh}, we are in the setting of \cref{thm:limit-polar-projection-lower-bound-fstar}. Indeed, for $R=H_{\gamma,\baseu}$, we obtain an instance of \eqref{eq:dr-expr} with
\[
    \bar h(\alt{u}) = \grad K(\baseu)\alt{u}+c_\baseu.
\]
Furthermore, we also have
\[
    \grad^2 G(u)\coPrimal + \grad_u[\grad \bar h(u)^*\Dual]\coPrimal=\alpha I,
\]
so we may take $c_G=\alpha$ in \eqref{eq:limit-c_G-fstar}. If $\gamma>0$, \eqref{eq:limit-cKF-fstar} is trivially satisfied. By \cref{thm:limit-polar-projection-lower-bound-fstar}, we therefore obtain
\[
    \lipnum_{\inv H_{\gamma,\realoptu}}(0|\realoptq) < \infty. 
\]
Thus $\inv H_{\gamma,\realoptu}$ has the Aubin property at $(0, \realoptq)$ provided $\gamma>0$. We summarize these findings in the following proposition.
\begin{proposition}
    \label{prop:pde-metric-regularity}
    Let $G$ be as in \eqref{eq:pde-g-choice}, $K \in C^1(X; Y)$, and let $F^*$ satisfy \eqref{eq:pde-f-choice} and \eqref{eq:fstar-polar-form-huber}.
    Suppose that $\realoptq$ solves $0 \in H_{\gamma,\realoptu}(\realoptq)$ for some $\gamma \ge 0$.  Then $w \mapsto \inv H_{\gamma,\realoptu}(w)$ has the Aubin property at $(0|\realoptq)$ if and only if $\gamma>0$ or $\bar\kvbound(\realoptq|0; H_{\gamma,\realoptu})>0$.
\end{proposition}

If $\gamma=0$, we have to prove existence of a lower bound $c_{K,\DerivCone}>0$ through $\bar\kvbound$. This is significantly more difficult. With $\baseu=\realoptu$, we use \eqref{eq:h-def} to compute
\[
    \bar h(\realoptu)=
    \grad K(\realoptu)\realoptu+c_\baseu=K(\realoptu)
    \quad\text{and}\quad
    \grad \bar h(\realoptu)=\grad K(\realoptu).
\]
Consequently, \eqref{eq:kvbound-fstar} can be expressed in the setting of this proposition as
\begin{equation}
    \label{eq:kvbound-fstar-metric}
    \begin{aligned}
    \bar\kvbound(\realoptq|0; H_{\realoptu}) & =
        \sup_{t>0} \inf\left\{ \frac{\norm{\inv\alpha\grad K(\realoptu)\grad K(\realoptu)^* \dcVar-\nu}}{\norm{\dcVar}}
            \,\middle|\, (\dcVar, \nu) \in \DerivConeDFt[\realopt{\Dual}|K(\realoptu)]{t},\, \dcVar \ne 0 \right\}
        \\
        & =
        \inv\alpha\sup_{t>0} \inf\left\{ \frac{\norm{\grad K(\realoptu)\grad K(\realoptu)^* \dcVar-\nu}}{\norm{\dcVar}}
            \,\middle|\, 
            \begin{array}{l}
            0 \ne \dcVar \in \DerivConeF[\alt{\Dual}|\alt{\coDual}],\, \nu \in \polar{\DerivConeF[\alt{\Dual}|\alt{\coDual}]}, \\
            \alt{\coDual} \in \subdiff F^*(\alt{\Dual}),\, \norm{\alt{\Dual}-\realopt{\Dual}} < t,\\ \norm{\alt{\coDual}-K(\realoptu)} < t
            \end{array}
            \right\}.
    \end{aligned}
\end{equation}
We will return to the issue of verifying -- or disproving -- the lower bound on $\bar\kvbound$ with specific examples in \cref{sec:stability_paramid}.

\subsection{Stability with respect to data}

We now want to study the stability of the condition $0 \in H_{\realoptu}(\realoptq)$ with respect to perturbation of the data $y^\delta$. This of course only makes sense if we equate the base point $\baseu$ in $H_\baseu$ to the solution $\realoptu$. Therefore, we define for variations $\dir{y}$ in the data
\[
    \altsetmap_{\dir{y}}(u,\Dual) \defeq P(u, \Dual)+h_{\dir{y}}(u,\Dual)
\]
with
\[
    h_{\dir{y}}(u,\Dual) \defeq
        \begin{pmatrix}
            \grad K(u)^*\Dual \\
            \dir{y}-K(u)
        \end{pmatrix}
    \quad\text{and}\quad
    P(u, \Dual) \defeq
        \begin{pmatrix}
            \grad G(u) \\
            \subdiff F^*(\Dual) \\
        \end{pmatrix}.
\]
We remark that due to the linear dependence of the optimality conditions $0 \in \altsetmap_{\dir{y}}(\Primal,\Dual)$ on $\Delta y$, the stability with respect to $\Delta y$ can be seen as a form of tilt-stability \cite{Rockafellar:1998b,Mordukhovich:2012,Drusvyatskiy:2013,Eberhard:2012,Lewis:2013,Mordukhovich:2013,Levy:2000,Mordukhovich:2014} for saddle-point systems.

Observe now that
\[
    \altsetmap_{\dir{y}}(q)=\basesetmap(q)+\begin{pmatrix}0 \\ \dir{y}\end{pmatrix},
\]
and in particular that $\altsetmap_{0}=\basesetmap$.
Thus \eqref{eq:sensitivity1} with $\somesetmap=\basesetmap$ and $w=(0, \dir{y})$ yields
\begin{equation}
    \label{eq:sensitivity2}
    \inf_{q\,:\, 0 \in \altsetmap_{\dir{y}}(q)} \norm{\realoptq-q}
        \le \lipnum_{\inv \basesetmap}(0|\realoptq) \norm{\dir{y}}
    \quad\text{ whenever }\quad
     \norm{\dir{y}} \le \rho.
\end{equation}
If $K \in C^2(X; Y)$, by \cref{prop:dr0}, we can compute $D\basesetmap(q|w)$.
In fact, with
\[
    \bar h(u) = K(u),
\]
we see that $\basesetmap$ is an instance of the class covered by \cref{thm:limit-polar-projection-lower-bound-fstar}. Its application directly yields the following proposition.
\begin{proposition}
    \label{prop:pde-data-stability}
    Let $K \in C^2(X; Y)$ and suppose that $F^*$ satisfies  \eqref{eq:pde-f-choice} and \eqref{eq:fstar-polar-form-huber}.
    Denote by $\realoptq_{\dir{y}}$ a solution to the optimality conditions \eqref{eq:oc} for the problem
    \[
        \min_u \max_\Dual \frac{\alpha}{2}\norm{u}^2
        + \iprod{K(u)-\dir{y}}{\Dual} - F_\gamma^*(\Dual).
    \]
    Suppose that a solution $\realoptq=\realoptq_0$ exists, and there exists a constant $c_G>0$ such that
    \begin{equation}
        \label{eq:pde-cg}
        \alpha\norm{\coPrimal}^2+\iprod{\grad_u [\grad K(\realoptu)^*\realopt{\Dual}]\coPrimal}{\coPrimal} \ge c_G \norm{\coPrimal}^2 \qquad (\coPrimal \in X).
    \end{equation}
    If $\gamma>0$ or $\bar\kvbound(\realoptq|0; \basesetmap)>0$, then for some $\rho, \ell > 0$ there exist solutions $\realoptq_{\dir{y}}$ with
    \begin{equation}
        \label{eq:sensitivity3}
        \norm{\realoptq-\realoptq_{\dir{y}}}
            \le \lipnum \norm{\dir{y}}
        \quad\text{ whenever }\quad
         \norm{\dir{y}} \le \rho.
    \end{equation}
\end{proposition}

Note that for $\bar\kvbound(\realoptq|0; \basesetmap)$, we obtain from \eqref{eq:kvbound-fstar} exactly the same expression as for $\bar\kvbound(\realoptq|0; H_{\realoptu})$ in \eqref{eq:kvbound-fstar-metric}, i.e.,
\begin{equation}
    \label{eq:kvbound-fstar-data}
    \bar\kvbound(\realoptq|0; \basesetmap) 
        =
        \inv\alpha\sup_{t>0} \inf\left\{ \frac{\norm{\grad K(\realoptu)\grad K(\realoptu)^* \dcVar-\nu}}{\norm{\dcVar}}
            \,\middle|\, 
            \begin{array}{l}
            0 \ne \dcVar \in \DerivConeF[\alt{\Dual}|\alt{\coDual}],\, \nu \in \polar{\DerivConeF[\alt{\Dual}|\alt{\coDual}]}, \\
            \alt{\coDual} \in \subdiff F^*(\alt{\Dual}),\, \norm{\alt{\Dual}-\realopt{\Dual}} < t,\\ \norm{\alt{\coDual}-K(\realoptu)} < t
            \end{array}
            \right\}.
\end{equation}

\subsection{Stability with respect to the Moreau--Yosida parameter}

Finally, we study the stability of the regularized optimality condition $0 \in H_{\realoptu,\gamma}(\realoptq)$ with respect to the Moreau--Yosida parameter $\gamma$. With $P$ as in the previous section, we now set
\[
    \altsetmap_{\gamma}(u,\Dual) \defeq P_\gamma(u, \Dual)+h_{\gamma}(u,\Dual),
\]
with
\[
    h_{\gamma}(u,\Dual) \defeq
        \begin{pmatrix}
            \grad K(u)^*\Dual \\
            \gamma \Dual -K(u)
        \end{pmatrix}
    \quad\text{and}\quad
    P_\gamma(u, \Dual) \defeq
        \begin{pmatrix}
            \grad G(u) \\
            \subdiff F_\gamma^*(\Dual) \\
        \end{pmatrix}.
\]
Observe that $\altsetmap_0=\basesetmap$.
Let $\realoptq$ solve $0 \in \basesetmap(\realoptq)$.
Now \eqref{eq:inverse-aubin} applied to $\altsetmap_\gamma$ at $\realoptq$ and $\realopt{w} \in \altsetmap_\gamma(\realoptq)$ gives with $w=0$ and $q=\realoptq$ the estimate
\begin{equation}
    \label{eq:moreau-yosida-sensitivity-1}
    \inf_{p\,:\, 0 \in \altsetmap_\gamma(p)} \norm{p-\realoptq}
        \le \lipnum_{\inv \altsetmap_\gamma}(\realopt{w}|\realoptq) \norm{\altsetmap_\gamma(\realoptq)}
    \quad\text{ whenever }\quad
     \norm{\realopt{w}} \le \rho
     .
\end{equation}
Since $0 \in \basesetmap(\realoptq)$, we deduce that $\realopt{w}_\gamma \defeq (0, \gamma\realopt{\Dual}) \in \altsetmap_\gamma(\realoptq)$. 
This quickly leads to the following proposition.
\begin{proposition}
    \label{prop:pde-moreau-yosida-stability}
    Let $K \in C^2(X; Y)$, and suppose $F^*$ satisfies  \eqref{eq:pde-f-choice} and \eqref{eq:fstar-polar-form-huber}.
    Denote by $\realoptq_{\gamma}$ a solution to the optimality conditions \eqref{eq:oc} for the problem
    \[
        \min_u \max_\Dual \frac{\alpha}{2}\norm{u}^2
            + \iprod{K(u)}{\Dual} - F_\gamma^*(\Dual).
    \]
    Suppose a solution $\realoptq=\realoptq_0$ exists, $\bar\kvbound(\realoptq|0; \basesetmap)>0$,  and \eqref{eq:pde-cg} holds.
    Then for some $\rho, \ell > 0$ there exist solutions $\realoptq_{\gamma}$ with
    \begin{equation}
        \label{eq:moreau-yosida-sensitivity-2}
        \norm{\realoptq-\realoptq_{\gamma}}
            \le \lipnum \gamma
        \quad\text{ whenever }\quad
         0 \le \gamma \le \rho.
    \end{equation}
\end{proposition}

\begin{proof}    
    We may assume that $\norm{\realopt{v}} \ne 0$, because otherwise $\realoptq_\gamma=\realoptq$.
    With $\realopt{w}_\gamma = (0, \gamma\realopt{\Dual})$, as above, we expand \eqref{eq:moreau-yosida-sensitivity-1} into
    \begin{equation}
        \inf_{p\,:\, 0 \in \altsetmap_\gamma(p)} \norm{p-\realoptq}
            \le \gamma \lipnum_{\inv \altsetmap_\gamma}(\realopt{w}_\gamma|\realoptq) \norm{\realopt{\Dual}},
        \quad\text{ whenever }\quad
         0 \le \gamma \le \norm{\realopt{\Dual}}^{-1}\rho.
    \end{equation}
    In order to derive \eqref{eq:moreau-yosida-sensitivity-2}, we only need to show the existence of a finite constant $\lipnum_{\inv \altsetmap_\gamma}(\realopt{w}_\gamma|\realoptq) < \infty$ and integrate $\norm{\realopt{\Dual}}$ into the constant.
    For this, we simply apply \cref{thm:limit-polar-projection-lower-bound-fstar} to $R=\altsetmap_\gamma$ with $\bar h(\alt{u}) = K(\alt{u})$, 
    and observe that $\bar\kvbound(\realoptq|\realopt{w}_\gamma; \altsetmap_\gamma)=\bar\kvbound(\realoptq|0; \basesetmap)$. This follows from the fact that the expression \eqref{eq:kvbound-fstar} only depends on $\gamma$ through the base point $\eta+\bar h(u)$, which in this case is $\gamma\realopt{\Dual}+(-\gamma\realopt{\Dual}+K(\realoptu))=K(\realoptu)$. Observe that $\eqref{eq:pde-cg}$ is equally independent of $\gamma$. 
    We can thus bound  $\lipnum_{\inv \altsetmap_\gamma}(\realopt{w}_\gamma|\realoptq)$ from above uniformly in $\gamma\in[0,\rho]$.
\end{proof}

\section{Application to parameter identification problems}\label{sec:stability_paramid}

We now discuss the possibility of satisfying the assumptions of the preceding propositions in the context of the motivating parameter identification problems \eqref{eq:l1fit_problem} and \eqref{eq:linffit_problem}. Since this will depend on the specific structure of the parameter-to-observation mapping $S$,  we consider as a concrete example the problem of recovering the potential term in an elliptic equation. 

Let $\Omega\subset\mathbb{R}^d$ be
an open bounded domain with a Lipschitz boundary $\partial\Omega$. For a given parameter $u\in \{v\in L^\infty(\Omega):v\geq \eps\}\eqcolon U\subset X\defeq L^2(\Omega)$, denote by $S(u)\defeq y\in H^1(\Omega)\subset L^2(\Omega)\eqcolon Y$ the weak solution of 
\begin{equation}\label{eq:forward}
    \inner{\nabla y,\nabla v} + \inner{uy,v} = \inner{f,v} \qquad (v\in H^1(\Omega)).
\end{equation}
This operator has the following useful properties \cite{Kroener:2009a}:
\begin{enumerate}[label=(\textsc{a}\arabic*), ref=\textsc{a}\arabic*]
    \item The operator $S$ is uniformly bounded in $U\subset{X}$ and completely
        continuous:
        If for $u\in U$, the sequence $\{u_n\}\subset U$ satisfies
        $u_n \wkto u$ in ${X}$, then
        \begin{equation}
            S(u_n)\to  S(u) \quad\text{ in } Y.
        \end{equation}
    \item $S$ is twice Fr\'echet differentiable.
    \item\label{ass:a3} There exists a constant $C>0$ such that 
        \begin{equation}
            \norm{\grad{S}(u)h}_{L^2}\leq C\norm{h}_X\qquad (u\in U,h\in X).
        \end{equation}
    \item\label{ass:a4} There exists a constant $C>0$ such that 
        \begin{equation}
            \norm{\grad^2 S(u)(h,h)}_{L^2} \le C \norm{h}_X^2\qquad (u\in U,h\in X).
        \end{equation}
\end{enumerate}
Furthermore, from the implicit function theorem, the directional Fréchet derivative $\grad{S}(u)h$ for given $h\in X$ can be computed as the solution $w\in H^1(\Omega)$ to
\begin{equation}\label{eq:forward_lin}
    \inner{\nabla w,\nabla v} + \inner{uw,v} = \inner{-yh,v} \qquad(v\in H^1(\Omega)).
\end{equation}
Similarly, the directional adjoint derivative $\grad{S}(u)^*h$ is given by $yz$, where $z\in H^1(\Omega)$ solves
\begin{equation}\label{eq:forward_adj}
    \inner{\nabla z,\nabla v} + \inner{uz,v} = \inner{-h,v} \qquad(v\in H^1(\Omega)).
\end{equation}
Similar expressions hold for $\grad^2{S}(u)(h_1,h_2)$ and $\grad(\grad{S}(u)^*h_1)h_2$. Hence, assumptions (\ref{ass:a3}--\ref{ass:a4}) hold for $\grad{S}^*$ and $\grad_u(\grad{S}(u)^*v)$ for given $v$ as well.

Other operators satisfying the above assumptions are mappings from a Robin or diffusion coefficient to the solution of the corresponding elliptic partial differential equation \cite{ClasonJin:2011}.

\subsection{\texorpdfstring{$\scriptstyle L^1$}{L¹} fitting}
\label{sec:stability-l1}

Let us first consider the $L^1$ fitting problem \eqref{eq:l1fit_problem}.
We are in the setting of \eqref{eq:pde-f-choice}--\eqref{eq:pde-g-choice}.
More specifically now
\[
    F^*(\Dual)=\int_\Omega f^*(\Dual(x)) \,d x
    \quad
    \text{for}
    \quad
    f^*(\FdVar) = \ind_{[-1, 1]}(\FdVar),
\]
where we allow the integral to be possibly infinite if the integrand does not satisfy $f^* \circ \Dual \in L^1(\Omega)$. 
We also have
\[
    G(u)=\int_\Omega g(u(x)) \,d x
    \quad
    \text{for}
    \quad
    g(\FdVar) = \frac{\alpha}{2} \abs{z}^2,
\]
as well as
\[
    K(u)=S(u)-\theData.
\]
Thus, the saddle-point conditions \eqref{eq:oc-h} for \eqref{eq:l1fit_problem} are given by
\begin{equation}
    0\in \begin{pmatrix}
        \alpha \realoptu +\grad S(\realoptu) \realoptpsi\\
        \partial\ind_{[-1,1]}(\realoptpsi) + S(\realoptu)-\theData
    \end{pmatrix}.
\end{equation}

\paragraph{Metric regularity}
We first address metric regularity of $H_\realoptu$ (\cref{prop:pde-metric-regularity}) when $\gamma=0$. Recall from \cref{prop:pde-metric-regularity} that in this case we need to show that 
\begin{equation}
    \label{eq:l1-fitting-kvbound}
    \begin{aligned}[t]
    \bar\kvbound(\realoptq|0; H_{\realoptu})
    &=
    \sup_{t>0} \inf\left\{ \frac{\norm{\inv\alpha\grad S(\realoptu)\grad S(\realoptu)^* \dcVar-\nu}}{\norm{\dcVar}}
        \,\middle|\, (\dcVar, \nu) \in \DerivConeDFt[\realopt{\Dual}|\theData-S(\realoptu)]{t},\, \dcVar \ne 0 \right\}
    \\
     &=
    \inv\alpha\sup_{t>0} \inf\left\{ \frac{\norm{\grad S(\realoptu)\grad S(\realoptu)^* \dcVar-\nu}}{\norm{\dcVar}}
        \,\middle|\, 
        \begin{array}{l}
        0 \ne \dcVar \in \DerivConeF[\alt{\Dual}|\alt{\coDual}],\, \nu \in \polar{\DerivConeF[\alt{\Dual}|\alt{\coDual}]}, \\
        \alt{\coDual} \in \subdiff F^*(\alt{\Dual}),\, \norm{\alt{\Dual}-\realopt{\Dual}} < t,\\ \norm{\alt{\coDual}-(\theData-S(\realoptu))} < t
        \end{array}
        \right\}
        \\
        &>0.
        \end{aligned}
\end{equation}
Let us try to force $\dcVar(x) \ne 0$ as much as possible in \eqref{eq:l1-fitting-kvbound}. From \cref{cor:g-form-indicator}, we obtain that $\dcVar \in \DerivConeF[\alt{\Dual}|\alt{\coDual}]$ satisfies
\begin{equation}
    \label{eq:l1fitting-z}
    \dcVar(x) \in
    \begin{cases}
        \{0\}, & \abs{\alt\Dual(x)}=1 \text{ and } \alt\coDual(x) \ne 0,\\
        -\sign \alt\Dual(x) [0, \infty), & \abs{\alt\Dual(x)}=1 \text{ and } \alt\coDual(x) = 0,\\
        \R, & \abs{\alt\Dual(x)} < 1 \text{ and } \alt\coDual(x) = 0,\\
    \end{cases}
\end{equation}
while $\nu \in \polar{\DerivConeF[\alt{\Dual}|\alt{\coDual}]}$ satisfies
\begin{equation}
    \label{eq:l1fitting-nu}
    \nu(x) \in
    \begin{cases}
        \R, & \abs{\alt\Dual(x)}=1 \text{ and } \alt\coDual(x) \ne 0,\\
        \sign \alt\Dual(x) [0, \infty), & \abs{\alt\Dual(x)}=1 \text{ and } \alt\coDual(x) = 0,\\
        \{0\}, & \abs{\alt\Dual(x)} < 1 \text{ and } \alt\coDual(x) = 0.\\
    \end{cases}
\end{equation}
 Therefore, $\dcVar(x) \ne 0$ can only happen if $\alt\coDual(x) = 0$. If
\begin{equation}
    \label{eq:l1-unreachable}
    \realopt\coDual \defeq \theData-S(\realoptu) \ne 0 \qquad (x \in\Omega),
\end{equation}
that is, if the data is reached almost nowhere, then the condition $\norm{\realopt\coDual-\alt\coDual}<t$ gives for any $\eps>0$ for small enough $t>0$ the estimate
\[
    \L^d(\{x \in \Omega \mid \alt\coDual(x)=0\})<\eps.
\]
In consequence,
\[
    \L^d(\{x \in \Omega \mid \dcVar(x) \ne 0\})<\eps.
\]
With this, we deduce that
\[
    \norm{\dcVar}_{L^1(\Omega)}
    \le
    \norm{\chi_{\{\dcVar \ne 0\}}}_{L^2(\Omega)}
    \norm{\dcVar}_{L^2(\Omega)}
    \le 
    \sqrt{\eps} \norm{\dcVar}_{L^2(\Omega)}.
\]
We furthermore have that 
\[
    \norm{\grad S(\realoptu)^* \dcVar} \le C \norm{\dcVar}_{L^1(\Omega)};
\]
this follows from the fact that $\grad S(\realoptu):W^{-1,s'}(\Omega)\to C(\overline\Omega)$ for any $s'>d$ due to the regularity of $\partial\Omega$ and $\realoptu$ (see, e.g., \cite[Thm.~6.3]{Griepentrog:2001}) and hence that $\grad S(\realoptu)^*:C(\Omega)^*\to W^{1,s}(\Omega)$ for $s<d$ together with the embeddings $L^1(\Omega)\hookrightarrow C(\overline\Omega)^*$ and $W^{1,s}(\Omega)\hookrightarrow L^2(\Omega)$ for $s\geq 1$ ($d=2$) or $s\geq 6/5$ ($d=3$).
An application of these estimates with $\nu=0$ in \eqref{eq:l1-fitting-kvbound} yields
\[
    \bar\kvbound(\realoptq|0; H_{\realoptu})
    \le \inv\alpha C \norm{\grad S(\realoptu)} \sqrt{\eps}.
\]
Letting $\eps \downto 0$, we deduce that
\[
    \bar\kvbound(\realoptq|0; H_{\realoptu})
    =0.
\]

If, on the other hand, the data is reached on a set $E$ of positive measure, i.e., $\realopt\coDual=0$ on $E$, we may choose $\dcVar$ freely on $E$. However, the lower bound
\[
    \norm{\grad S(\realoptu)^* \dcVar} \ge c \norm{\dcVar}
    \qquad (\dcVar \in L^2(E)),
\]
does not hold in general (take any orthonormal basis of $L^2(E)$, which converges weakly but not strongly to zero, and use the fact that $\grad S(u)$ is a compact operator from $L^2(\Omega)$ to $L^2(\Omega)$ due to the Rellich--Kondrachev embedding theorem).
Again, $\bar\kvbound(\realoptq|0; H_{\realoptu})=0$.

Therefore by \cref{prop:pde-metric-regularity}, \emph{there is no metric regularity without some sort of regularization}.
On the other hand, with Moreau--Yosida regularization, i.e., for $\gamma>0$, we always have metric regularity of $H_\realoptu$ at $(\realoptq, 0)$ by the same proposition.

\paragraph{Data stability}

The situation is very similar for stability with respect to data (\cref{prop:pde-data-stability}) when $\gamma=0$. Comparing \eqref{eq:kvbound-fstar-metric} and \eqref{eq:kvbound-fstar-data}, we see that we have to study whether
\[
    \bar\kvbound(\realoptq|0; \basesetmap) = \bar \kvbound(\realoptq|0; H_{\realoptu})>0.
\]
Hence, we again cannot have data stability without Moreau--Yosida regularization. With regularization, i.e., for $\gamma>0$, we still need to prove \eqref{eq:pde-cg}. Using the reverse triangle inequality, the boundedness of the dual variable $\realoptpsi(x)\in [-1,1]$ due to the choice of $F^*$, and assumption \eqref{ass:a4}, we have that
\begin{equation}
        \alpha\norm{\coPrimal}^2+\iprod{\grad_u [\grad S(\realoptu)^*\realopt{\Dual}]\coPrimal}{\coPrimal} \ge (\alpha - C)\norm{\coPrimal}^2 \ge 
        c_G \norm{\coPrimal}^2
\end{equation}
for $\alpha$ sufficiently large and hence data stability.

\paragraph{Stability with respect to $\gamma$}

Since \cref{prop:pde-moreau-yosida-stability} holds under exactly the same conditions as \cref{prop:pde-data-stability}, we deduce that there is no stability with respect to the Moreau--Yosida parameter at $\gamma=0$. This is to be expected, as any addition of regularization will, whenever $\realopt\coDual(x)=0$, immediately force $\Dual(x)=0$.
At a point $\gamma>0$, the stability can be proved similarly to the arguments in \cref{prop:pde-moreau-yosida-stability}.

\subsection{\texorpdfstring{$\scriptstyle L^\infty$}{L∞} fitting}

Let us now consider the $L^\infty$ fitting problem \eqref{eq:linffit_problem}.
We are again in the setting of \eqref{eq:pde-f-choice}--\eqref{eq:pde-g-choice},
this time with
\[
    F^*(\Dual)=\int_\Omega f^*(\Dual(x)) \,d x
    \quad
    \text{for}
    \quad
    f^*(\FdVar) = \noise\abs{z},
\]
and $G$ and $K$ as in the previous subsection. Hence, the saddle-point conditions \eqref{eq:oc-h} are now given by
\begin{equation}
    0\in \begin{pmatrix}
        \alpha \realoptu +\grad S(\realoptu) \realoptpsi\\
        \delta \sign(\realoptpsi) + S(\realoptu)-\theData
    \end{pmatrix}.
\end{equation}

\paragraph{Metric regularity}

Again, for metric regularity of $H_\realoptu$  (\cref{prop:pde-metric-regularity}) we need to show 
\[
    \bar\kvbound(\realoptq|0; H_{\realoptu})>0.
\]

Let us force again $z(x)\ne 0$. From \cref{cor:g-form-l1norm}, we obtain that $\dcVar \in \DerivConeF[\alt{\Dual}|\alt{\coDual}]$ satisfies
\begin{equation}
    \label{eq:linffitting-z}
    \dcVar(x) \in
    \begin{cases}
        \{0\}, &  \abs{\alt\coDual(x)} \ne \delta,\\
        \sign \alt\coDual(x) [0, \infty), & \abs{\alt\coDual(x)}=\delta \text{ and } \alt\Dual(x) = 0,\\
        \R, & \abs{\alt\coDual(x)} = \delta \text{ and } \alt\Dual(x) \neq 0,
    \end{cases}
\end{equation}
while $\nu \in \polar{\DerivConeF[\alt{\Dual}|\alt{\coDual}]}$ satisfies
\begin{equation}
    \label{eq:linffitting-nu}
    \nu(x) \in
    \begin{cases}
        \R, & \abs{\alt\coDual(x)}\neq \delta \text{ and } \alt\Dual(x) = 0,\\
        \sign \alt\Dual(x) (-\infty, 0], & \abs{\alt\coDual(x)}= \delta \text{ and } \alt\Dual(x) = 0,\\
        \{0\}, & \alt\Dual(x) \neq 0,
    \end{cases}
\end{equation}
If $\realopt v(x)=0$ a.\,e., and $\esssup_{x \in \Omega} \abs{\realopt\coDual(x)} < \noise$ holds -- meaning the constraint
\begin{equation}
    \label{eq:linfty-constr-discussion}
    \abs{S(\realoptu)(x)-\theData(x)}\le\noise 
\end{equation}
is almost never active -- then we can proceed as in \cref{sec:stability-l1} to show for any $\eps>0$ for small enough $t>0$ the estimate
\[
    \L^d(\{x \in \Omega \mid \abs{\alt\coDual(x)}=\noise\}) \le \eps.
\]
In consequence, $\dcVar \in \DerivConeF[\Dual|\alt\coDual]$ satisfies
\[
    \L^d(\{x \in \Omega \mid \dcVar(x) \ne 0\})
    \le
    \L^d(\{x \in \Omega \mid \abs{\alt\coDual(x)}=\noise\})
    \le
    \eps,
\]
and we deduce following \cref{sec:stability-l1} that
\[
    \bar\kvbound(\realoptq|0; H_{\realoptu})
    =\bar\kvbound(\realoptq|0; \basesetmap)=0.
\]
Therefore, by \cref{prop:pde-metric-regularity}, \emph{we have no metric regularity if the constraint \eqref{eq:linfty-constr-discussion} is almost never active.} (Any small change could force it to be active, and hence cause a large change in the dual variable.)
However, also if the constraint \eqref{eq:linfty-constr-discussion} is active on an open set $E$, we may reason as in \cref{sec:stability-l1} to show instability.
The only way to obtain stability is therefore with Moreau--Yosida  regularization.

\paragraph{Data stability}

Stability with respect to data (\cref{prop:pde-data-stability}) again requires that 
\[
    \bar\kvbound(\realoptq|0; \basesetmap) = \bar \kvbound(\realoptq|0; H_{\realoptu})>0.
\]
Hence, we cannot have data stability without Moreau--Yosida regularization. If $\gamma>0$, we additionally need to prove \eqref{eq:pde-cg}. Using the reverse triangle inequality and assumption \eqref{ass:a4}, we have that
\begin{equation}
        \alpha\norm{\coPrimal}^2+\iprod{\grad_u [\grad S(\realoptu)^*\realopt{\Dual}]\coPrimal}{\coPrimal} \ge (\alpha - C)\norm{\coPrimal}^2 \ge 
        c_G \norm{\coPrimal}^2
\end{equation}
for $\alpha$ sufficiently large and hence data stability. Since in this case we do not have an a priori bound on $\realoptpsi$, the choice of $\alpha$ depends on $\realoptpsi$ and hence on the data $y^\delta$.

\paragraph{Stability with respect to $\gamma$}

As in the case of $L^1$-fitting, stability with respect to the Moreau--Yosida parameter only holds at $\gamma>0$. 

\subsection{Regularization through projection}
\label{sec:discretization}
Discretization provides an alternative to regularization. 
Indeed, in practice, the data $\theData$ lies in a finite-dimensional subspace $Y' \subset Y=L^2(\Omega)$. With $\mathcal{P}$ the orthogonal projection from $Y$ into $Y'$, we then replace the fitting term $F$ by $F_{\mathcal{P}} \defeq F\circ \mathcal{P}$. We then have with $y=y'+y^\bot \in Y'\oplus (Y')^\bot=Y$ that
\begin{equation}
    F_{\mathcal{P}}^*(v) = \sup_{y^\bot}\, \langle y^\bot, v \rangle + (\sup_{y'}\, \langle y', v \rangle - F(y'))
    = \begin{cases} F^*(v), & v\in ((Y')^\bot)^\bot = Y',\\
        \infty, & v \notin Y'. \end{cases}
\end{equation}
We emphasize that in this approach we only discretize the fitting term, while the nonlinear operator $S$ and the regularizer $G$ remain infinite-dimensional.

Hence,
\[
    \subdiff F_{\mathcal{P}}^*(v)
    =
    \begin{cases}
        \subdiff F^*(v)+(Y')^\perp, & v \in Y', \\
        \emptyset, & v \not \in Y'.
    \end{cases}
\]
From the definition \eqref{eq:graphderiv} of the graphical derivative, we calculate that if
either $\Dual\not\in Y'$ or $\dir\Dual \not\in Y'$, then $D[\subdiff F_{\mathcal{P}}^*](\Dual|\coDual)(\dir\Dual)=\emptyset$. When $\Dual\in Y'$ and $\dir\Dual \in Y'$, we have for any $\tilde \coDual \in \subdiff F^*(\Dual) \isect(\coDual + (Y')^\bot)$ the inclusion
\[
    \begin{split}
    D[\subdiff F_{\mathcal{P}}^*](\Dual|\coDual)(\dir\Dual)
    &
    =
    D[\mathcal{P}\subdiff F^*](\Dual|\coDual)(\dir\Dual)+(Y')^\bot
    \\
    &
    \supset
    \mathcal{P}D[\subdiff F^*](\Dual|\tilde\coDual)(\dir\Dual)+(Y')^\bot
    \\
    &
    =
    D[\subdiff F^*](\Dual|\tilde\coDual)(\dir\Dual)+(Y')^\bot.
    \end{split}
\]
Consequently, by basic properties of convex hulls, we also have the inclusion
\begin{equation}
    \label{eq:discr-inner-approx}
    \widetilde{D[\subdiff F_{\mathcal{P}}^*]}(\Dual|\coDual)(\dir\Dual)
    \supset
    \begin{cases}
        \widetilde{D[\subdiff F^*]}(\Dual|\tilde\coDual)(\dir\Dual)+(Y')^\bot,
            & v, \Delta v \in Y',\, \coDual \in \subdiff F^*(\Dual) + (Y')^\bot, \\
        \emptyset, & \text{otherwise.}
    \end{cases}
\end{equation}
Suppose now that $DF^*$ satisfies \eqref{eq:dfstar-expr}, that is,
\begin{equation}
    \notag
    \widetilde{D[\subdiff F^*]}(\alt{\Dual}|\alt{\coDual})(\dir{\Dual}) \supset
    \begin{cases}
        \gamma \dir{\Dual} + \polar{\DerivConeF(\alt{\Dual}|\alt{\coDual})}, & \dir{\Dual} \in \DerivConeF[\alt{\Dual}|\alt{\coDual}], \\
        \emptyset, & \dir{\Dual} \not\in \DerivConeF[\alt{\Dual}|\alt{\coDual}].
    \end{cases}
\end{equation}
Since any cone $V$ and subspace $Y'$ satisfy the easily verified identity
\[
    \polar{(V \isect Y')} = \polar V + (Y')^\bot,
\]
it follows for $v, \Delta v \in Y'$ and $\coDual \in \subdiff F^*(\Dual) + (Y')^\bot $ that
\begin{equation}
    \label{eq:structure-f-project}
    \widetilde{D[\subdiff F_{\mathcal{P}}^*]}(\alt{\Dual}|\alt{\coDual})(\dir{\Dual}) \supset
    \begin{cases}
        \gamma \dir{\Dual} + \polar{\DerivConeF(\alt{\Dual}|\tilde{\coDual})} + (Y')^\bot, & \dir{\Dual} \in \DerivConeF[\alt{\Dual}|\tilde{\coDual}] \isect Y', \\
        \emptyset, & \text{otherwise}.
    \end{cases}
\end{equation}
This holds for arbitrary $\tilde \coDual \in \subdiff F^*(\Dual) \isect(\coDual + (Y')^\bot)$. In fact, in the following, we let $\tilde \coDual \in \subdiff F^*(\Dual)$ be the free parameter, and take, for example, $\coDual=\tilde\coDual$.

Let now $H_{\realoptu,\mathcal{P}}$ and $\somesetmap_{0,\mathcal{P}}$ be defined by \eqref{eq:h-def} and \eqref{eq:r0-def}, respectively, with $F_{\mathcal{P}}$ in place of $F$.
An application of \cref{thm:limit-polar-projection-lower-bound-fstar} to the inclusion \eqref{eq:structure-f-project} then gives the lower bound
\begin{equation}
    \label{eq:l1-metric-regularity-kvbound}
    \begin{aligned}[t]
    \bar\kvbound(\realoptq|0;\, & H_{\realoptu,\mathcal{P}})
    =\bar\kvbound(\realoptq|0; \somesetmap_{0,\mathcal{P}})
    \\
    &
    \ge
    \inv\alpha\sup_{t>0} \inf\left\{ \frac{\norm{\grad S(\realoptu)\grad S(\realoptu)^* \dcVar-\nu}}{\norm{\dcVar}}
        \,\middle|\, 
        \begin{array}{l}
            0 \ne \dcVar \in \DerivConeF[\alt{\Dual}|\tilde{\coDual}] \isect Y', \\ \nu \in \polar{\DerivConeF[\alt{\Dual}|\tilde{\coDual}]} + (Y')^\perp, \\
            \tilde{\coDual} \in \subdiff F^*(\alt{\Dual}),\, \alt{\Dual} \in Y',\\ \norm{\alt{\Dual}-\realopt{\Dual}} < t,\, \norm{\mathcal{P}(\tilde{\coDual}-\realopt{\coDual})} < t
        \end{array}
    \right\}
    \\
    &
    =
    \inv\alpha\sup_{t>0} \inf\left\{ \frac{\norm{\mathcal{P}\grad S(\realoptu)\grad S(\realoptu)^* \dcVar-\mathcal{P}\tilde\nu}}{\norm{\dcVar}}
        \,\middle|\, 
        \begin{array}{l}
            0 \ne \dcVar \in \DerivConeF[\alt{\Dual}|\tilde{\coDual}] \isect Y', \\ \tilde\nu \in \polar{\DerivConeF[\alt{\Dual}|\tilde{\coDual}]}, \\
            \tilde{\coDual} \in \subdiff F^*(\alt{\Dual}),\, \alt{\Dual} \in Y',\\ \norm{\alt{\Dual}-\realopt{\Dual}} < t,\, \norm{\mathcal{P}(\tilde{\coDual}-\realopt{\coDual})} < t
        \end{array}
    \right\}.
    \end{aligned}
\end{equation}

Let $\{e_1, \ldots, e_N\}$ be an orthonormal basis for $Y'$, and for any $v \in Y$‚ denote $v_i \defeq \iprod{e_i}{v}$.
Then \eqref{eq:l1-metric-regularity-kvbound} forces
\begin{equation}
    \label{eq:discr-t-bound}
    \sum_{i=1}^N \abs{\alt{\Dual}_i-\realopt{\Dual}_i}^2 \le t^2
    \quad\text{and}\quad
    \sum_{i=1}^N \abs{\tilde{\coDual}_i-\realopt{\coDual}_i}^2 \le t^2.
\end{equation}
Suppose there exist for each $i=1,\ldots,N$ closed sets $A_i, B_i \subset \R$ satisfying for each $\tilde{\coDual}$ and $\alt{\Dual}$ with $\tilde{\coDual} \in \subdiff F^*(\alt{\Dual})$ the conditions
\begin{subequations}
\label{eq:discr-cond}
\begin{align}
    \label{eq:discr-cond-sc}
    \text{either}\quad  (\alt{\Dual}_i \not\in A_i \text{ and }  \alt{\coDual}_i \in B_i)
    &\quad\text{or}\quad (\alt{\Dual}_i \in A_i \text{ and }  \alt{\coDual}_i \not\in B_i), \\
    \label{eq:discr-cond-b}
     \tilde\coDual_i \not\in B_i  \text{ and } \dcVar \in \DerivConeF[\alt{\Dual}|\tilde{\coDual}] \isect Y' 
    & \implies z_i=0, \\
    \label{eq:discr-cond-a}
     \alt{\Dual}_i \not\in A_i  \text{ and } \tilde\nu \in \polar{\DerivConeF[\alt{\Dual}|\tilde{\coDual}]} & \implies \tilde\nu_i=0.
\end{align}
\end{subequations}
Observe that the condition \eqref{eq:discr-cond-sc} is a type of strict complementarity condition. We observe the following two situations.
\begin{enumerate}[label=(\roman*)]
    \item If $\realopt{\Dual}_i \not\in A_i$, \eqref{eq:discr-t-bound} for small enough $t>0$ forces $\alt{\Dual}_i \not\in A_i$, and through \eqref{eq:discr-cond-sc}, $\tilde\coDual_i \in B_i$. Thus by \eqref{eq:discr-cond-a}, $\tilde\nu_i=0$.
    \item Likewise, if $\realopt{\coDual}_i \not\in B_i$, \eqref{eq:discr-t-bound} gives $\alt{\coDual}_i \not\in A_i$. Thus by \eqref{eq:discr-cond-b}, $z_i=0$, and by \eqref{eq:discr-cond-sc}, $\Dual_i \in A_i$.
\end{enumerate}
Thus the situations are mutually exclusive, and, by \eqref{eq:discr-cond-sc}, one of them has to occur. \emph{Importantly, therefore,  $\tilde\nu_i=0$ has to hold when we do not have the constraint $z_i=0$.} 
Therefore, letting $\tilde\nu_i$ vary freely whenever we have the constraint $z_i=0$, and $z_i$ vary freely whenever we have the constraint $\tilde\nu_i=0$, we obtain from \eqref{eq:l1-metric-regularity-kvbound} the lower estimate
\begin{equation}
    \label{eq:l1-metric-regularity-kvbound-discr-2}
    \begin{aligned}[t]
    \bar\kvbound(\realoptq|0; H_{\realoptu,\mathcal{P}})
    &
    =\bar\kvbound(\realoptq|0; \somesetmap_{0,\mathcal{P}})
    \\
    &
    \ge
    \inv\alpha \inf\left\{ \frac{\norm{\mathcal{P}\grad S(\realoptu)\grad S(\realoptu)^*\mathcal{P} \dcVar}}{\norm{\dcVar}}
        \,\middle|\, 
        z \in Y',\, z_i=0 \text{ if } \realopt{\Dual}_i \in A_i 
    \right\}
    \\
    &
    \ge
    \inv\alpha \inf\left\{ \frac{\norm{\mathcal{P}\grad S(\realoptu)\grad S(\realoptu)^*\mathcal{P} \dcVar}}{\norm{\dcVar}}
        \,\middle|\, 
        z \in Y'
    \right\}.
    \end{aligned}
\end{equation}
Since $\grad S(\realoptu)$ is invertible (as the inverse of a linear partial differential operator; cf. \eqref{eq:forward_lin}) and the orthogonal projection $\mathcal{P}$ is self-adjoint, the restriction of $\mathcal{P}\grad S(\realoptu)\grad S(\realoptu)^*\mathcal{P}$ to $Y'$ is a self-adjoint positive definite operator on the finite-dimensional space $Y'$ and therefore boundedly invertible. This implies the existence of a constant $c>0$ such that
\begin{equation}
    \norm{\mathcal{P}\grad S(\realoptu)\grad S(\realoptu)^*\mathcal{P} \dcVar} \geq c \norm{\dcVar}\qquad (\dcVar \in Y'),
\end{equation}
which yields $\bar\kvbound(\realoptq|0; H_{\realoptu,\mathcal{P}})=\bar\kvbound(\realoptq|0; \somesetmap_{0,\mathcal{P}}) \geq  \alpha^{-1}c >0$ and therefore metric regularity.

\bigskip

It remains to verify the conditions \eqref{eq:discr-cond}.
For this, let us further assume that the discretization is piecewise constant, that is $e_i=\abs{\Omega_i}^{-1}\chi_{\Omega_i}$, for some subdomains $\Omega_i \subset \Omega$ with $\sum_{i=1}^N \chi_{\Omega_i}=\chi_\Omega$. 
For the $L^1$-fitting example, \eqref{eq:D-fstar-l1-constr} then gives
\begin{align}
\abs{\alt{\Dual}_i}<1 &\implies \tilde{\coDual}\chi_{\Omega_i}=0.
\intertext{For $t>0$ small enough, \eqref{eq:l1fitting-z} and \eqref{eq:l1fitting-nu} therefore give}
    \label{eq:discr-cond1}
    \abs{\alt{\Dual}_i}<1 
    &\implies \tilde\nu_i=0 \text{ and } z_i \in \R.
\end{align}
This suggests to take $A_i=\{-1,1\}$ to satisfy \eqref{eq:discr-cond-a}.
Then, for \eqref{eq:discr-cond-sc} to be satisfied, the condition \eqref{eq:D-fstar-l1-constr} gives the only possibility of $B_i=\{0\}$. Further, for the strict complementarity within \eqref{eq:discr-cond-sc} to hold, it is  necessary to impose that the middle two cases in \eqref{eq:l1fitting-z} and \eqref{eq:l1fitting-nu} do not occur. This strict complementary condition may also be stated as
\begin{equation}
    \label{eq:l1-strict-compl}
    (1-\abs{\realopt\Dual_i})\realopt{\coDual}_i=0
    \quad\text{and}\quad
    0<(1-\abs{\realopt\Dual_i})+\abs{\realopt{\coDual}_i}.
\end{equation}
To verify \eqref{eq:discr-cond-b}, suppose that $\tilde\coDual_i \not \in B_i$.
It may happen that $\tilde\coDual\chi_{\Omega_i}$ behaves wildly.
Nevertheless, by $\tilde\coDual_i \not \in B_i$, there are points $x \in \Omega_i$ where necessarily $\tilde\coDual(x)\ne 0$. Hence the condition \eqref{eq:D-fstar-l1-constr} forces $0 \ne \alt{v}(x)=\sign \tilde\coDual(x)$. Thus $z(x)=0$, which through $z \in Y'$ forces $z_i=0$, proving \eqref{eq:discr-cond-b}. Thus \eqref{eq:discr-cond} and the estimate in \eqref{eq:l1-metric-regularity-kvbound-discr-2} hold for projection-regularized $L^1$ fitting under the strict complementarity condition \eqref{eq:l1-strict-compl}.

For $L^\infty$-fitting, still using piecewise constant discretization, studying \eqref{eq:linfty-fitting-constr}, \eqref{eq:linffitting-z}, and \eqref{eq:linffitting-nu}, we see that we can take $A_i=\{0\}$ and $B_i=\{-\delta, \delta\}$ to obtain the same results under the strict complementarity condition
\begin{equation}
    \label{eq:linfty-strict-compl}
    \realopt\Dual_i(\delta-\abs{\realopt{\coDual}_i})=0
    \quad\text{and}\quad
    0<\abs{\realopt\Dual_i}+(\delta-\abs{\realopt{\coDual}_i}).
\end{equation}

Note that the strict complementarity conditions \eqref{eq:l1-strict-compl} and \eqref{eq:linfty-strict-compl} can always be satisfied after a small perturbation of $\realopt\Dual$, if necessary. Indeed, if $\realopt\coDual_i$ is active ($\realopt\coDual_i=0$ resp.~$\abs{\realopt\coDual_i}=\delta$), then by \eqref{eq:D-fstar-l1-constr} resp.~\eqref{eq:linfty-fitting-constr}, $\realopt\Dual_i$ can be made inactive -- in which case the strict complementarity condition is satisfied -- while maintaining the optimality condition $\realopt\coDual \in \subdiff F_{\mathcal{P}}(\realopt\Dual)$. This change in $\realopt\Dual$ can, however, alter the constant in \eqref{eq:pde-cg}.

Observe further that for the original infinite-dimensional problem, similar strict complementarity conditions (pointwise almost everywhere) could be derived, but these would not be sufficient to obtain metric regularity since we still have the problem that the inverse of $\grad S(\realoptu)\grad S(\realoptu)^*$ is unbounded on $L^2(\Omega)$.
Further, in the $L^2$ topology, even a strong complementarity condition (with $\eps$ lower bound in the inequality) would not be sufficient to transport it from the optimal solution to the perturbed variables $\alt\Dual$ and $\alt\coDual$.

\bigskip

For stability with respect to perturbations of the data,
condition 
\eqref{eq:pde-cg} is also required. Since this condition is independent of $F$, it will hold under discretization whenever it holds for the original problem (with a possibly different constant $c_G$ since now $\realoptpsi\in Y'\subset Y$).

\section{Conclusion}

The purpose of this work was to derive explicit stability criteria for solutions to saddle-point problems in Hilbert spaces, in particular those arising from the minimization of nonsmooth nonlinear functionals commonly occurring in parameter identification, image processing, or PDE-constrained optimization problems. Our main results are a pointwise characterization of regular coderivatives of convex subdifferentials of integral functionals and explicit conditions for metric regularity of the corresponding variational inclusions. These make it possible to verify the Aubin property for concrete problems. While the results for our model problems are mostly negative (no regularity unless regularization or discretization is introduced), they are still useful: Our function-space analysis provides a unified framework for \emph{any} conforming regularization; in particular, it shows that the stability properties are independent of the discretization of the unknown parameter. Furthermore, for \emph{arbitrary small} fixed Moreau--Yosida parameters, the properties are also independent of the discretization of the data; this is especially important for the convergence of numerical algorithms, where this translates in a discretization-independent number of iterations required to reach a given tolerance.

This work can be extended in a number of directions. In a follow-up paper, we will apply our results on the Aubin property of pointwise set-valued mappings to the convergence analysis of the nonlinear primal-dual extragradient method from \cite{Valkonen:2014} in function spaces. We also plan to investigate the possibility of obtaining \emph{partial} stability results with respect to only the primal variable without regularization or discretization. 
An alternative would be to exploit the uniform stability with respect to regularization for fixed discretization, and with respect to discretization for fixed regularization, to obtain a combined convergence for a suitably chosen net $(\gamma,h)\to (0,0)$; this is related to the adaptive regularization and discretization of inverse problems \cite{KaltenbacherKirchnerVexler11}.
Furthermore, it would be of interest to extend our analysis to include nonsmooth regularizers $G$, which were excluded in the current work for the sake of the presentation.
It would also be worthwhile to try to adapt the stability analysis to make use of the limiting coderivative and its richer calculus; in particular to 
remove the geometric derivability assumption by directly working with the limiting coderivative.
Finally, the pointwise characterization of coderivatives could be useful in deriving more explicit optimality conditions for bilevel optimization problems.

\appendix

\section{Coderivative formula}
\label{sec:coderivative}

In this appendix, we give an explicit characterization of the regular coderivative for a class of set-valued mapping covering the examples in \cref{sec:pointwise:examples}.
\begin{proposition}
    \label{prop:frechetcod-cone-operator}
    For a set-valued operator $R: Q \setto Q$, on a Hilbert space $Q$, suppose that
    \begin{equation}
        \label{eq:linear-polar-form-app}
        \widetilde{D \somesetmap}(q|w)(\dir{q})
        =
        \begin{cases}
            T_q \dir{q} + \polar{\DerivCone[q|w]}, & \dir{q} \in \DerivCone[q|w], \\
            \emptyset, & \text{otherwise},   
        \end{cases}
    \end{equation}
    for a linear operator $T\defeq T_q: Q \to Q$, dependent on $q$ but not $w$, and a cone $\DerivCone[q|w] \subset Q$, dependent on both $q$ and $w$.
    Then
    \begin{equation}
        \label{eq:frechetcod-cone-operator}
        \frechetCod \somesetmap(q|w)(\dir{w})
        =
        \begin{cases}
            T_q^*\dir{w} + \polar{\DerivCone[q|w]}, & -\dir{w} \in {\DerivCone[q|w]}^{\circ\circ} \\
            \emptyset, & \text{otherwise}.
        \end{cases}
    \end{equation}
\end{proposition}
Note that if the cone $\DerivCone[q|w]$ is closed and convex -- in particular, a closed subspace -- then ${\DerivCone[q|w]}^{\circ\circ} = \DerivCone[q|w]$.
\begin{proof}
    We denote for short $\DerivCone \defeq \DerivCone[q|w]$.
    The expression \eqref{eq:linear-polar-form-app} simply says that
    \[
        \conv\bigl(T((q,w); {\graph D\somesetmap})\bigr)
        =\{ (\Delta q,\Delta w) \in Q \times Q \mid
            \Delta q \in \DerivCone,\,
            \Delta w \in T_q \Delta q + \polar{\DerivCone}
        \}.
    \]
    We already know from \cref{sec:findimco} that
    \[
        \widehat N((q,w); {\graph D\somesetmap})=
        \polar{T((q,w); {\graph D\somesetmap})}
        =\polar{\conv\bigl(T((q,w); {\graph D\somesetmap})\bigr)}.
    \]
    Thus,
    \[
        \begin{split}
        \widehat N((q,w); {\graph D\somesetmap})
        &
        =
        \left\{
            (\Delta q',\Delta w') \in Q^2
            \,\middle|\,
            \begin{array}{l}
            \iprod{\Delta q'}{\Delta q} +\iprod{\Delta w'}{\Delta w}\le 0
            \\
            \hfill
            \text{ for }
            \Delta q \in \DerivCone,\,
            \Delta w \in T_q \Delta q + \polar{\DerivCone}
            \end{array}
        \right\}
        \\
        &
        =
        \left\{
            (\Delta q',\Delta w') \in Q^2
            \,\middle|\,
            \begin{array}{l}
            \iprod{\Delta q'+T_q^* \Delta w'}{\Delta q} +\iprod{\Delta w'}{\Delta w}\le 0
            \\
            \hfill
            \text{ for }
            \Delta q \in \DerivCone,\,
            \Delta w \in \polar{\DerivCone}
            \end{array}
        \right\}
        \\
        &
        =
        \left\{
            (\Delta q',\Delta w') \in Q^2
            \,\middle|\,
            \Delta q' + T_q^* \Delta w' \in \polar{\DerivCone},\,
            \Delta w' \in {\DerivCone}^{\circ\circ}
        \right\}.
        \end{split}
    \]
    By \cref{def:frechetcod}, we have
    \begin{equation}
        \frechetCod \somesetmap(q | w)(\dir{w}) \defeq
        \left\{ \dir{q} \in Q \mid (\dir{q}, -\dir{w}) \in \widehat N((q, w); \graph \somesetmap) \right\},
    \end{equation}
    from which \eqref{eq:frechetcod-cone-operator} follows.
\end{proof}

\section*{Acknowledgments}

In Cambridge, T.~Valkonen has been supported by the King Abdullah University of Science and Technology (KAUST) Award No.~KUK-I1-007-43, and EPSRC grants Nr.~EP/J009539/1 ``Sparse \& Higher-order Image Restoration'', and Nr.~EP/M00483X/1 ``Efficient computational tools for inverse imaging problems''. While in Quito, T.~Valkonen has moreover been supported by a Prometeo scholarship of the Senescyt (Ecuadorian Ministry of Science, Technology, Education, and Innovation).

\printbibliography

\end{document}